\documentclass[12]{article}
\usepackage[utf8]{inputenc}

\if@twoside
\else 
\fi 

\usepackage{amsmath}
\usepackage{amsthm}
\usepackage{amssymb}
\usepackage{cite}
\usepackage{url}
\usepackage[multiple]{footmisc}
\usepackage{graphicx}
\usepackage{epstopdf}
\usepackage{chngcntr}
\usepackage{enumitem}
\usepackage{hhline}
\usepackage{array}
\usepackage{hyperref}
\usepackage{color}
\usepackage{array}
\usepackage{multirow}
\usepackage{xcolor}
\hypersetup{colorlinks=false}
\usepackage{tikz}
\usetikzlibrary{patterns}
\usetikzlibrary{arrows.meta}
\usepackage{float}
\usepackage{booktabs}

\makeatletter
\let\@fnsymbol\@arabic
\makeatother

\begin{document}


\counterwithin{figure}{section}
\counterwithin{table}{section}

\theoremstyle{plain}
\newtheorem{theorem}{Theorem}[section]
\newtheorem{lemma}[theorem]{Lemma}
\newtheorem{proposition}[theorem]{Proposition}
\newtheorem{algorithm}  [theorem]{Algorithm}
\newtheorem{corollary}[theorem]{Corollary}
\newtheorem{problem}{Problem}

\theoremstyle{definition}
\newtheorem{definition}{Definition}[section]
\newtheorem{assumption}{Assumption}[section]
\newtheorem{example}{Example}[section]
\newtheorem*{notation}{Notation}

\theoremstyle{remark}
\newtheorem{remark}{Remark}[section]


\newcommand{\dis}{\displaystyle}
\newcommand{\norm}[1]{\left\|#1\right\|}
\newcommand{\newnorm}[1]{\left|\left|\left|#1\right|\right|\right|}
\newcommand{\abs}[1]{\left\vert#1\right\vert}
\newcommand{\spr}[1]{\left\langle\,#1\,\right\rangle}
\newcommand{\kl}[1]{\left(#1\right)}
\newcommand{\Kl}[1]{\left\{#1\right\}}
\newcommand{\vl}{\, \vert \,}
\newcommand{\supp}[1]{\text{supp}\left(#1\right)}

\newcommand{\RED}[1]{\textcolor{red}{#1}}
\newcommand{\BLUE}[1]{\textcolor{blue}{#1}}
\definecolor{aog}{rgb}{0.0, 0.5, 0.0}
\newcommand{\GREEN}[1]{\textcolor{aog}{#1}}
\newcommand{\gray}[1]{\textcolor{gray}{#1}}

\newcommand{\LandauO}{\mathcal{O}}
\newcommand{\Landauo}{o}

\newcommand{\R}{\mathbb{R}} 
\newcommand{\C}{\mathbbThe{C}}
\newcommand{\N}{\mathbb{N}}
\newcommand{\Z}{\mathbb{Z}}
\newcommand{\Q}{\mathbb{Q}}

\newcommand{\mfunc}[1]{\text{#1} \,}
\def\div{\operatorname{div}} 
\newcommand{\rot}{\mfunc{rot}}
\newcommand{\ess}{\mfunc{ess}}
\newcommand{\sgn}{\mfunc{sgn}}

\newcommand{\rge}[1]{\mathcal{R}(#1)}
\newcommand{\nsp}[1]{\mathcal{N}(#1)}
\newcommand{\dom}[1]{\mathcal{D}(#1)}

\newcommand{\xd}{x^\dagger}
\newcommand{\xa}{x_\alpha}
\newcommand{\xad}{x_\alpha^\delta}
\newcommand{\yd}{y^{\delta}}

\newcommand{\eps}{\varepsilon}
\newcommand{\wto}{\rightharpoonup}


\newcommand{\ddt}{\frac{\partial}{\partial t}}
\newcommand{\ddx}{\frac{\partial}{\partial x}}

\newcommand{\mm}{\text{mm}}
\newcommand{\cm}{\text{cm}}
\newcommand{\dm}{\text{dm}}
\newcommand{\m}{\text{m}}
\newcommand{\s}{\text{s}}


\newcommand{\RN}{{\R^N}}

\newcommand{\CmO}{{C^m(\Omega)}}
\newcommand{\CiO}{{C^\infty(\Omega)}}
\newcommand{\CizO}{{C^\infty_0(\Omega)}}

\newcommand{\Lt}{{L_2}}
\newcommand{\LtO}{{L_2(\Omega)}}
\newcommand{\LtOp}{{L_2(\Omega')}}
\newcommand{\LtR}{{L_2(\R)}}
\newcommand{\LtRN}{{L_2(\R^N)}}

\newcommand{\Hs}{{H^s}}
\newcommand{\HsO}{{H^s(\Omega)}}
\newcommand{\HsR}{{H^s(\R)}}
\newcommand{\HsRN}{{H^s(\R^N)}}

\newcommand{\Hm}{{H^m}}
\newcommand{\HmO}{{H^m(\Omega)}}
\newcommand{\HmzO}{{H_0^m(\Omega)}}
\newcommand{\HmR}{{H^m(\R)}}
\newcommand{\HmRN}{{H^m(\R^N)}}

\newcommand{\vphi}{\varphi}
\newcommand{\laplace}{\Delta}

\newcommand{\uf}{\mathbf{u}}
\newcommand{\uft}{\tilde{\mathbf{u}}}
\newcommand{\vf}{\mathbf{v}}
\newcommand{\zf}{\mathbf{z}}
\newcommand{\zft}{\tilde{\mathbf{z}}}
\newcommand{\wf}{\mathbf{w}}
\newcommand{\ufhi}{\hat{\uf}^i}
\newcommand{\x}{\textbf{x}}
\newcommand{\xhi}{\hat{\x}^i}
\newcommand{\HoO}{{H^1(\Omega)}}
\newcommand{\intV}[1]{\int\limits_\Omega #1 \, dx}
\newcommand{\intVx}[1]{\int\limits_\Omega #1 \, d\mathbf{x}}
\newcommand{\nv}{\vec{\boldsymbol{n}}}
\newcommand{\mE}[1]{\mathcal{E}(#1)}
\newcommand{\E}{\mathcal{E}}

\newcommand{\mI}{\mathcal{I}}
\newcommand{\af}{\mathbf{a}}
\newcommand{\hf}{\mathbf{h}}
\newcommand{\gf}{\mathbf{g}}
\newcommand{\gfD}{\mathbf{g}_D}
\newcommand{\gfT}{\mathbf{g}_T}
\newcommand{\ff}{\mathbf{f}}
\newcommand{\pf}{\mathbf{p}}
\newcommand{\F}{\mathcal{F}}

\newcommand{\Msm}{\mathcal{M}_s(\underline{\mu})}
\newcommand{\mub}{\underline{\mu}}

\newcommand{\prox}{\operatorname{prox}}


\title{On the Intensity-based Inversion Method for Quantitative Quasi-Static Elastography}
\author{
Ekaterina Sherina\footnote{University of Vienna, Faculty of Mathematics, Oskar Morgenstern-Platz 1, 1090 Vienna, Austria (ekaterina.sherina@univie.ac.at), corresponding author}\,\,\textsuperscript{,}\footnote{Christian Doppler Laboratory for Mathematical Modeling and Simulation of Next Generations of Ultrasound Devices (MaMSi), Oskar Morgenstern-Platz 1, 1090 Vienna, Austria} ,
Simon Hubmer\footnote{Johannes Kepler University Linz, Institute of Industrial Mathematics, Altenbergerstra{\ss}e 69, A-4040 Linz, Austria, (simon.hubmer@jku.at)}}
\date{\today}
\maketitle

\begin{abstract}

In this paper, we consider the intensity-based inversion method (IIM) for quantitative material parameter estimation in quasi-static elastography. In particular, we consider the problem of estimating the material parameters of a given sample from two internal measurements, one obtained before and one after applying some form of deformation. These internal measurements can be obtained via any imaging modality of choice, for example ultrasound, optical coherence or photo-acoustic tomography. Compared to two-step approaches to elastography, which first estimate internal displacement fields or strains and then reconstruct the material parameters from them, the IIM is a one-step approach which computes the material parameters directly from the internal measurements. To do so, the IIM combines image registration together with a model-based, regularized parameter reconstruction approach. This combination has the advantage of avoiding some approximations and derivative computations typically found in two-step approaches, and results in the IIM being generally more stable to measurement noise. In the paper, we provide a full convergence analysis of the IIM within the framework of inverse problems, and detail its application to linear elastography. Furthermore, we discuss the numerical implementation of the IIM and provide numerical examples simulating an optical coherence elastography (OCE) experiment.


\end{abstract}


\section{Introduction}\label{sect_introduction}

Elastography, as an imaging modality in general, aims at mapping the mechanical properties of a given sample which is subjected to some type of deformation \cite{Parker_2011,Ophir_1991,Singh_Hepburn_Kennedy_Larin_2025}. The idea behind this modality goes back to the palpation examination used by doctors, which is motivated by the different response of soft (i.e., healthy) and stiff (i.e., malignant) tissue to applied touch, turning it into an effective, albeit sometimes inaccurate, method for diagnostics over the years. Thanks to technological advancements, a number of elastographic techniques were developed in recent years to access information about the mechanical properties and stiffness of tissues and organs \cite{Manduca_Oliphant_Dresner_Mahowald_Kruse_Amromin_Felmlee_Greenleaf_Ehman_2001,ZaiMatMatSovHepMowKen21,SinTho2019,SigLiaKafChaWil2017,Singh_Hepburn_Kennedy_Larin_2025}, which are based on classical imaging modalities such as ultrasound (US), computerized tomography (CT), magnetic resonance imaging (MRI), or optical coherence tomography (OCT). By themselves, these modalities provide only qualitative information about the scanned sample, such as images showing some internal structure but without any physically meaningful values. However, in terms of diagnostic accuracy, one is interested in quantitative values of stiffness and sometimes strain mapped on top of the visualization of a sample, rather than only in qualitative images. By combining both, i.e., by performing elastography on top of a classical imaging modality, quantitative elastography was established as a new technique allowing for quantitative estimates of the mechanical properties of samples and tissues \cite{Manduca_Oliphant_Dresner_Mahowald_Kruse_Amromin_Felmlee_Greenleaf_Ehman_2001,Doyley_2012,ZaiMatMatSovHepMowKen21,SinTho2019,SigLiaKafChaWil2017}.

Historically, elastography was first pioneered in the 1990s in US imaging, where it is nowadays used to add mechanical contrast on top of imaging data, e.g., in cancer detection \cite{cespedes_elastography_1993, sommerfeld_prostatakarzinomdiagnostik_2003, garra_elastography_1997, sigrist_ultrasound_2017}, the evaluation of benign lesions of the musculoskeletal system \cite{kim_usefulness_2016}, or the assessment of tendon injuries \cite{prado-costa_ultrasound_2018}. The concept of elastography was also adapted in the field of MRI, where it is used to visualize the mechanical behavior of internal organs \cite{kruse_tissue_2000, kennedy_magnetic_2020, low_general_2016} and other human body parts, e.g., breast tumors\cite{sinkus_high-resolution_2000} or arterial walls \cite{kolipaka_chapter_2020}. Both US imaging and MRI work with resolutions in the millimeter range, but for some applications like early cancer detection \cite{highnam_mammographic_1997, karssemeijer_adaptive_1993}, investigation of various eye diseases \cite{edrington_keratoconus_1995, girard_translating_2015} or detection of arterial rigidifications \cite{segers_brief_2019}, micro-scale elastographic imaging is necessary. Applications in this imaging range became possible with the emergence of optical coherence elastography (OCE) in 1998 and the demonstration of the first results of using a time-domain optical coherence imaging system for elastography on a gelatin phantom, pork meat, and an in-vivo human finger \cite{Schmitt_1998}. Further advances followed the introduction of spectral domain OCT, which allows for high speed, phase sensitivity, and resolution. In particular, OCE has potential for the non-invasive identification of malignant formations inside the human skin, eye retina examination, or tissue biopsies during surgeries \cite{ZaiMatMatSovHepMowKen21,Wang_Larin_2015}.

Regardless of the concrete application or the specific imaging modality, quantitative elastography generally follows a common three-step workflow, which is briefly outlined below: 
    \begin{enumerate}
        \item Perturb the sample/tissue using some form of mechanical source/force.
        \item Measure the sample/tissue response to the deformation using a suitable imaging technique.
        \item Reconstruct the mechanical parameters from the measured deformation/response.
    \end{enumerate}
Each of these steps can be implemented differently depending on the chosen technique \cite{Doyley_2012,SinTho2019, SigLiaKafChaWil2017,ZaiMatMatSovHepMowKen21}. In Step~1, the deformation can for example be imposed on the sample in a quasi-static, harmonic or transient manner \cite{Doyley_2012}. The resulting deformation or the deformed sample is then recorded in Step~2, either on the boundary or everywhere inside the sample using an imaging modality such as US, MRI, CT, or OCT. Internal deformations in the form of displacement or strain, either spatially or spatio-temporally resolved, are evaluated either from successive scans of the sample before and after deformation \cite{DunKir01,RogPatFujBre04,Schmitt_1998,SunStaYan11,Krainz_Sherina_Hubmer_Liu_Drexler_Scherzer_2022,ZaiMatMatSovHepMowKen21,Sherina_Krainz_Hubmer_Drexler_Scherzer_2020,sherina_challenges_2021} or directly calculated from imaging data obtained during compression (e.g., in phase-sensitive OCT, the axial component of displacement is calculated from the phase difference between the scans \cite{kennedy_strain_2012,ZaiMatMatSovHepMowKen21}). In Step~3, in order to find mechanical parameters such as Young's modulus and the Poisson ratio, the Lam\'e parameters, the shear modulus and density or the shear wave speed, an inverse problem has to be formulated and solved with a material model fitting the experiment, e.g., linear elastic, different visco-elastic, hyper-elastic, plastic and combined models \cite{Doyley_2012}.

\tikzstyle{process} = [rectangle, draw, text centered, minimum height=3em, align=center, fill=gray!10]
\tikzstyle{processIP} = [rectangle, double, draw, text centered, minimum height=5em, minimum width=5em, align=center, fill=gray!30]
\tikzstyle{connector} = [draw, line width=2pt, -Latex]
\begin{figure}[ht!] \centering
\resizebox{0.8\textwidth}{!}
{
\begin{tikzpicture}
    \node [process] at (0,0) (a1) {\textbf{Image} sample};
    \node [process] at (0,-2) (a2) {\textbf{Apply load} \\ to sample};
    \node [process] at (0,-4) (a3) {\textbf{Image} sample};
    \node [process] at (2.5,-2) (a4) {\textbf{Imaging} \\ \textbf{data}};
    \node [draw=none] at (5,-0.7) {\textbf{IP1}};
    \node [processIP] at (5,-2) (a5) {\textbf{Solve} \\ motion \\ estimation \\ problem};
    \node [process] at (8,-2) (a6) {\textbf{Displacement} \\ or \textbf{strain}};
    \node [draw=none] at (11,-0.7) {\textbf{IP2}};
    \node [processIP] at (11,-2) (a7) {\textbf{Solve} \\ material \\ parameters \\ problem};
    \node [process] at (14,-2) (a8) {\textbf{Material} \\ \textbf{parameters}};
    \node [processIP] at (8,-4.5) (a9) {\textbf{Solve} \\ material \\ parameters \\ problem};
    \node [draw=none] at (8,-3.2) {\textbf{IP3}};
    \path [connector] (a1) -- (a2);
    \path [connector] (a2) -- (a3);
    \path [connector] (a3) -- (a4);
    \path [connector] (a1) -- (a4);
    \path [connector] (a4) -- (a5);
    \path [connector] (a5) -- (a6);
    \path [connector] (a6) -- (a7);
    \path [connector] (a7) -- (a8);
    \path [connector] (a4) -- (2.5,-4.5) -- (a9) -- (14,-4.5) -- (a8);
\end{tikzpicture}
}
\caption{Schematic depiction of the workflow in quasi-static quantitative elastography. Imaging of the sample before/after deformation and inverse problems IP1-IP3 for obtaining material parameters.}
\label{fig-1-schematic}
\end{figure}

Figure~\ref{fig-1-schematic} schematically depicts this general workflow in the case of a quasi-static elastography experiment, where the considered sample is imaged before and after deformation, in this case induced by an applied load. As indicated, there are two main approaches for obtaining material parameter estimates following the imaging of the sample, which may itself involve the solution of an inverse problem. The first, and most common one, is known as the \emph{two-step approach} to quantitative elastography, and involves the solution of two inverse problems: First, in inverse problem~1 (IP1), the displacement or strain inside the sample/tissue undergoing deformation is estimated from its imaging data. Then, in inverse problem~2 (IP2), the material parameters of the sample/tissue are reconstructed from the obtained displacement or strain estimates. IP1 and IP2 are typically considered separately from each other, partly because they belong to quite distinct classes of inverse problems. While IP1 often takes the form of a motion estimation or optical flow problem \cite{Sherina_Krainz_Hubmer_Drexler_Scherzer_2020,sherina_challenges_2021,Ammari_Bretin_Millien_Seppecher_Seo_2015}, the material parameter estimation problem IP2 usually falls within the class of PDE-based coefficient-estimation inverse problems, and is in general much better analyzed. Among the vast literature on identifiability of the material parameters, stability, and different reconstruction methods for IP2, see e.g.\ \cite{Bal_Bellis_Imperiale_2014,Bal_Uhlmann_2012,Bal_Uhlmann_2013,Barbone_Gokhale_2004,Barbone_Oberai_2007,Doyley_2012,Doyley_Meaney_Bamber_2000,Fehrenbach_Masmoudi_Souchon_2006,Gokhale_Barbone_Oberai_2008,Huang_Shih_1997,Jadamba_Khan_Raciti_2008,Ji_McLaughlin_Renzi_2003,McLaughlin_Renzi_2006,Oberai_Gokhale_Doyley_2004,Oberai_Gokhale_Feijoo_2003,Widlak_Scherzer_2015,Lechleiter_Schlasche_2017,Kirsch_Rieder_2016,Hubmer_Sherina_Neubauer_Scherzer_2018, Carrillo_Waters_2022,Gimperlein_Waters_2016,Ammari_Waters_Zhang_2015,Ammari_Bretin_Garnier_Kang_Lee_Wahab_2015,Ammari_Seo_Zhou_2015,Ammari_Garapon_Jouve_2010,Barbeiro_Henriques_Santos_2024,Imanuvilov_Isakov_Yamamoto_2003,Leonov_Sharov_Yagola_2017,Leonov_Sharov_Yagola_2018,Leonov_Sharov_Yagola_2021}. Note that many of these works deal with linear elastic material models and time-dependent sample measurements, often inspired by certain applications, but sometimes also for reasons of (numerical) stability and the beneficial uniqueness properties of the time-dependent case. However, in many applications, no dynamic imaging (and hence displacement data) is available and thus, one has to work within the quasi-static setting.

While the two-step approach has generally enjoyed a lot of success in the past, the separate consideration of IP1 and IP2 also has a number of downsides. For example, standard optical flow approaches for IP1 are often ``blind'' to physical restrictions in the elastography experiment, leading to unrealistic displacement field estimates. This issue, which was recently addressed in \cite{sherina_challenges_2021}, also compounds with the issue of non-uniqueness in the material parameter reconstruction problem IP2. Informally, one can say that both IP1 and IP2 may have a large nullspace, while their combination IP1 + IP2 may not. Additionally, since IP1 and IP2 are solved sequentially, data noise and numerical inaccuracies amplify through the decoupled treatment of the corresponding inverse problems. This then motivates so-called \emph{one-step approaches} to quantitative elastography: instead of sequentially solving IP1 and IP2, one considers the combined inverse problem~3 (IP3), i.e., the direct reconstruction of the material parameters of the sample/tissue undergoing deformation from its imaging data, without an intermediate calculation of the displacement or strain. In this paper, we particularly focus on the \textit{intensity-based inversion method} (IIM), an intensity-based approach specifically designed for solving the combined IP3. The IIM was introduced in \cite{Krainz_Sherina_Hubmer_Liu_Drexler_Scherzer_2022} for the specific application of recovering the Young's modulus of a set of silicone rubber samples imaged with OCT before and after compression. Combining image registration with a linear elasticity-based deformation model, the IIM was able to accurately reconstruct the Young's modulus for a wide range of sample configurations \cite{Krainz_Sherina_Hubmer_Liu_Drexler_Scherzer_2022}. 

Here, we consider a generalized variant of the IIM, applicable to any quasi-static elastography experiment, independent of the chosen imaging modality, and compatible with any deformation model. Taking the form of a minimization problem resembling Tikhonov regularization, the IIM has a number of advantages over two-step approaches: First, its flexibility wrt.\ the choice of the underlying material model allows for easy adaptation to different experimental settings and use cases. Second, the IIM can avoid the differentiation of noisy displacement/strain data, which is typically required in implementations of iterative regularization methods in two-step approaches. From a theoretical perspective, most two-step approaches require a certain regularity of the measured displacement/strain, which is however not attainable in practice. In contrast, the IIM works directly on the measured sample images, and thus avoids the differentiation of a reconstructed, and thus noisy, displacement/strain. Third, the IIM can easily be adapted to incorporate additional physical prior information, such as sample segmentation data, leading to a reduced dimensionality of the considered elastography problem. Finally, its relation to Tikhonov regularization (with a noisy operator) allows the IIM to be analyzed within the framework of inverse problems \cite{Engl_Hanke_Neubauer_1996}, which is our main focus in this paper. In particular, we establish convergence and convergence rate results for the IIM under minimal mathematical assumptions, which we then verify for the specific use case of linear elasticity. Furthermore, we discuss extensions of the IIM going beyond the quasi-static elastography setting, and demonstrate its practical usefulness on a number of numerical examples which, motivated by \cite{Krainz_Sherina_Hubmer_Liu_Drexler_Scherzer_2022}, simulate a quasi-static OCE experiment.

The outline of this paper is as follows: In Section~\ref{sect_setting}, we first introduce the precise mathematical setting of the quasi-static elastography problem considered in this paper. Then, in Section~\ref{sect_IIM} we define the IIM for the solution of this problem, and provide a full convergence analysis in Section~\ref{sect_convergence}, which in particular includes an order-optimal convergence rate. In Section~\ref{sect_extensions_practical}, we discuss several extensions and practical aspects of the IIM, and in Section~\ref{sect_appl_lin}, we show that all assumptions of its convergence analysis are satisfied for the specific case of linear elastography. Finally, in Section~\ref{sect_numerics}, we present a number of numerical examples simulating an OCE experiment, and end with a short conclusion in Section~\ref{sect_conclusion}.

\section{Setting and Inverse Problem Formulation}\label{sect_setting}

In this section, we introduce a precise mathematical formulation of the general quasi-static elastography experiment considered in this paper, and formulate the resulting inverse parameter estimation problem. 

First, recall that we are interested in reconstructing material parameters of a given sample from two internal measurements, one obtained before and one after applying some form of deformation. To model this mathematically, let $\Omega_1, \Omega_2 \subset \R^2$ denote two open, bounded subsets of $\R^2$, where $\Omega_1$ represents the initial, non-deformed geometry of the sample, and $\Omega_2$ correspond to the sample's shape after deformation. Both non-deformed and deformed states of the sample are scanned using a tomographic imaging modality of choice, which results in two internal measurements represented by $\mI_1:\Omega_1\to\R$, $\mI_1=\mI_1(\x)$, and $\mI_2:\Omega_2\to\R$, $\mI_2=\mI_2(\mathbf{x})$, respectively. Here, $\mI_1$ corresponds to the image of the sample before deformation, and $\mI_2$ to that after deformation. Furthermore, assume that the quasi-static deformation of the sample can be described via the set of equations
    \begin{subequations}\label{eq:bvp}
    \begin{alignat}{1}
        \mathcal{L}(\af,\uf) & = \mathbf{f} \,, \quad \text{in} \quad \Omega_1 \,, \label{eq:pde-main}
        \\ 
        \mathcal{B}(\af,\uf) & = \hf \,, \quad \text{on} \quad \partial\Omega_1\setminus\Gamma_D \,,\label{eq:bc-any}
        \\
        \uf & = \gf \,,  \quad \text{on} \quad \Gamma_D \,. \label{eq:bc-dir} 
    \end{alignat}
    \end{subequations}
where $\uf:\Omega_1\to\R^2$ is the displacement and $\af:\Omega_1\to\R^n$ is a vector of material parameters. We note that $\uf = \uf(\x,\af)$ and $\af = \af(\x)$, and that it is the (spatially varying) material parameter vector $\af$ that we are interested in recovering. The concrete form of the operators $\mathcal{L}$ and $\mathcal{B}$, which encode some form of deformation law relating the material parameters $\af$ to the deformation $\uf$, is not essential for the upcoming analysis. All that is required is that under suitable assumptions on $\mathcal{L}$, $\mathcal{B}$, $\af$, $\mathbf{f}$, $\hf$, and $\gf$, the displacement $\uf$ is uniquely determined by \eqref{eq:pde-main}-\eqref{eq:bc-dir}. Note that we may have $\Gamma_D = \emptyset$ or $\Gamma_D = \partial \Omega_1$.

\begin{example}
In Section~\ref{sect_appl_lin}, we consider the specific setting of linear elasticity, where $\mathcal{L}$ is a second-order elliptic differential operator. In this case, condition \eqref{eq:bc-dir} models both fixed boundaries and applied external displacements or shear, while \eqref{eq:bc-any} is used to define traction(-free) boundary conditions.
\end{example}

\begin{figure}[ht!]
    \centering
    \begin{tikzpicture}[x=0.75pt,y=0.75pt,yscale=-0.85,xscale=0.85]
       
        \draw [line width=1.5]  (150,10) -- (370,10) -- (370,160) -- (150,160) -- cycle;
        \draw [dash pattern={on 4.5pt off 4.5pt}, fill=gray!15][line width=0.75]  (152,42) -- (368,42) -- (368,158) -- (152,158) -- cycle;

        \draw (250,75) node [anchor=north west][inner sep=0.75pt][font=\Large]{$\Omega_1$};
        \draw (155,135) node [anchor=north west][inner sep=0.75pt][font=\Large]{$\Omega$};
    
        \draw (375,65) node [anchor=north west][inner sep=0.75pt][align=left]{$G(\af)(\x):=\x+\uf(\x,\af)$};
        \draw (380,85) -- (500,85) -- (530,85);
        \draw [shift={(530,85)}, rotate = 180] [color={rgb, 255:red, 0; green, 0; blue, 0 }][line width=0.75](10.93,-3.29) .. controls (6.95,-1.4) and (3.31,-0.3) .. (0,0) .. controls (3.31,0.3) and (6.95,1.4) .. (10.93,3.29);

        \draw (565,10) -- (565,40);
        \draw (580,10) -- (580,40);
        \draw (595,10) -- (595,40);
        \draw (610,10) -- (610,40);
        \draw (625,10) -- (625,40);
        \draw (640,10) -- (640,40);
        \draw (655,10) -- (655,40);
        \draw (670,10) -- (670,40);
        \draw (685,10) -- (685,40);
        \draw (700,10) -- (700,40);
        \draw (715,10) -- (715,40);
        \draw (730,10) -- (730,40);
        \draw (745,10) -- (745,40);
        \draw (760,10) -- (760,40);
        \draw (775,10) -- (775,40);
        \draw (790,10) -- (790,40);
        \draw [shift={(565,40)}, rotate = -90] [color={rgb, 255:red, 0; green, 0; blue, 0 }][line width=0.75](10.93,-3.29) .. controls (6.95,-1.4) and (3.31,-0.3) .. (-1,0) .. controls (3.31,0.3) and (6.95,1.4) .. (10.93,3.29);
        \draw [shift={(580,40)}, rotate = -90] [color={rgb, 255:red, 0; green, 0; blue, 0 }][line width=0.75](10.93,-3.29) .. controls (6.95,-1.4) and (3.31,-0.3) .. (-1,0) .. controls (3.31,0.3) and (6.95,1.4) .. (10.93,3.29);
        \draw [shift={(595,40)}, rotate = -90] [color={rgb, 255:red, 0; green, 0; blue, 0 }][line width=0.75](10.93,-3.29) .. controls (6.95,-1.4) and (3.31,-0.3) .. (-1,0) .. controls (3.31,0.3) and (6.95,1.4) .. (10.93,3.29);
        \draw [shift={(610,40)}, rotate = -90] [color={rgb, 255:red, 0; green, 0; blue, 0 }][line width=0.75](10.93,-3.29) .. controls (6.95,-1.4) and (3.31,-0.3) .. (-1,0) .. controls (3.31,0.3) and (6.95,1.4) .. (10.93,3.29);
        \draw [shift={(625,40)}, rotate = -90] [color={rgb, 255:red, 0; green, 0; blue, 0 }][line width=0.75](10.93,-3.29) .. controls (6.95,-1.4) and (3.31,-0.3) .. (-1,0) .. controls (3.31,0.3) and (6.95,1.4) .. (10.93,3.29);
        \draw [shift={(640,40)}, rotate = -90] [color={rgb, 255:red, 0; green, 0; blue, 0 }][line width=0.75](10.93,-3.29) .. controls (6.95,-1.4) and (3.31,-0.3) .. (-1,0) .. controls (3.31,0.3) and (6.95,1.4) .. (10.93,3.29);
        \draw [shift={(655,40)}, rotate = -90] [color={rgb, 255:red, 0; green, 0; blue, 0 }][line width=0.75](10.93,-3.29) .. controls (6.95,-1.4) and (3.31,-0.3) .. (-1,0) .. controls (3.31,0.3) and (6.95,1.4) .. (10.93,3.29);
        \draw [shift={(670,40)}, rotate = -90] [color={rgb, 255:red, 0; green, 0; blue, 0 }][line width=0.75](10.93,-3.29) .. controls (6.95,-1.4) and (3.31,-0.3) .. (-1,0) .. controls (3.31,0.3) and (6.95,1.4) .. (10.93,3.29);
        \draw [shift={(685,40)}, rotate = -90] [color={rgb, 255:red, 0; green, 0; blue, 0 }][line width=0.75](10.93,-3.29) .. controls (6.95,-1.4) and (3.31,-0.3) .. (-1,0) .. controls (3.31,0.3) and (6.95,1.4) .. (10.93,3.29);
        \draw [shift={(700,40)}, rotate = -90] [color={rgb, 255:red, 0; green, 0; blue, 0 }][line width=0.75](10.93,-3.29) .. controls (6.95,-1.4) and (3.31,-0.3) .. (-1,0) .. controls (3.31,0.3) and (6.95,1.4) .. (10.93,3.29);
        \draw [shift={(715,40)}, rotate = -90] [color={rgb, 255:red, 0; green, 0; blue, 0 }][line width=0.75](10.93,-3.29) .. controls (6.95,-1.4) and (3.31,-0.3) .. (-1,0) .. controls (3.31,0.3) and (6.95,1.4) .. (10.93,3.29);
        \draw [shift={(730,40)}, rotate = -90] [color={rgb, 255:red, 0; green, 0; blue, 0 }][line width=0.75](10.93,-3.29) .. controls (6.95,-1.4) and (3.31,-0.3) .. (-1,0) .. controls (3.31,0.3) and (6.95,1.4) .. (10.93,3.29);
        \draw [shift={(745,40)}, rotate = -90] [color={rgb, 255:red, 0; green, 0; blue, 0 }][line width=0.75](10.93,-3.29) .. controls (6.95,-1.4) and (3.31,-0.3) .. (-1,0) .. controls (3.31,0.3) and (6.95,1.4) .. (10.93,3.29);
        \draw [shift={(760,40)}, rotate = -90] [color={rgb, 255:red, 0; green, 0; blue, 0 }][line width=0.75](10.93,-3.29) .. controls (6.95,-1.4) and (3.31,-0.3) .. (-1,0) .. controls (3.31,0.3) and (6.95,1.4) .. (10.93,3.29);
        \draw [shift={(775,40)}, rotate = -90] [color={rgb, 255:red, 0; green, 0; blue, 0 }][line width=0.75](10.93,-3.29) .. controls (6.95,-1.4) and (3.31,-0.3) .. (-1,0) .. controls (3.31,0.3) and (6.95,1.4) .. (10.93,3.29);
        \draw [shift={(790,40)}, rotate = -90] [color={rgb, 255:red, 0; green, 0; blue, 0 }][line width=0.75](10.93,-3.29) .. controls (6.95,-1.4) and (3.31,-0.3) .. (-1,0) .. controls (3.31,0.3) and (6.95,1.4) .. (10.93,3.29);

        \draw (545,15) node [anchor=north west][inner sep=0.75pt][align=left]{$\gf$};
   
        \draw [line width=1.5]  (565,40) -- (800,40);
        \draw [line width=1.5] (800,160) -- (565,160);
        \draw [dash pattern={on 4.5pt off 4.5pt}, fill=gray!15][line width=0.75]  (565,42) -- (800,42) -- (800,158) -- (565,158) -- cycle;
        \draw [line width=1.5] (800,160) arc
        [
        start angle=90,        end angle=-90,
        x radius=0.5cm,        y radius =1.58cm
        ] ; 
        \draw [line width=1.5] (565,160) arc
        [
        start angle=90,        end angle=270,
        x radius=0.5cm,        y radius =1.58cm
        ] ; 
        \draw (670,75) node [anchor=north west][inner sep=0.75pt][font=\Large]{$\Omega_2$};
        \draw (570,135) node [anchor=north west][inner sep=0.75pt][font=\Large]{$\Omega$};
    
    \end{tikzpicture}
    \caption{Illustration of our mathematical setting for quasi-static elastography: initial, non-deformed object geometry $\Omega_1$ (left) and deformed object geometry $\Omega_2$ (right); here $\gf$, is a fixed applied downward displacement. The shaded domain in both images corresponds to one possible choice of $\Omega$ in \eqref{domain_condition}.} 
    \label{fig_subdivision_subdomains}
\end{figure}

Next, given the displacement $\uf = \uf(\x,\af)$, we define the corresponding deformation $G = G(\af)$ via
    \begin{equation*}
        G = G(\af) : \Omega_1 \to \R^2 \,, \qquad \x \mapsto G(\af)(\x) := \x + \uf(\x,\af) \,,
    \end{equation*}
which describes to where a point $\x \in \Omega_1$ is moved as a result of the deformation $\uf$. Hence, we obtain
    \begin{equation}\label{eq:IIM:init}
        (\mI_2 \circ G(\af))(\x) = \mI_2(\x+\uf(\x,\af)) = \mI_1(\mathbf{x}) \,,
    \end{equation}
a connection between the measured images $\mI_1$ and $\mI_2$ and the material parameters $\af$, which serves as the starting point for our further considerations. Next, observe that for a given $\af$, the connection \eqref{eq:IIM:init} is only well-defined if both $\x \in \Omega_1$ and $G(\af)(\x) \in \Omega_2$ hold. Hence, in the following we assume  
    \begin{equation}\label{domain_condition}
        \exists \, \Omega \subset \Omega_1  \,\, \forall \, \af \in \mathcal{M} \,\, \forall \, \x \in \Omega \, : \, G(\af)(\x) \in \Omega_2 \,,
    \end{equation}
where $\mathcal{M}$ is a predefined admissible set of material parameters. Note that \eqref{domain_condition} is essentially a requirement on the considered elastography experiment, requiring that there is a domain $\Omega$ on which it makes sense to compare the measurements $\mI_1$ and $\mI_2$; cf.~Figure~\ref{fig_subdivision_subdomains}. Now, combining \eqref{eq:IIM:init} and \eqref{domain_condition}, we obtain 
    \begin{equation}\label{eq:IIM}
        (\mI_2 \circ G(\af))(\x)  = \mI_1(\mathbf{x}) \,, \qquad \forall \, \x \in \Omega \,,
    \end{equation}
which is a well-defined relation between $\mI_1$, $\mI_2$, and the material parameters $\af$. After these preliminaries, we can now formally define the quasi-static elastography problem considered in this paper:

\vspace{15pt}
\fbox{\parbox{0.9\textwidth}{\textbf{Quasi-Static Elastography Problem:} \textit{Compute the material parameter vector $\af(\x) \in \mathcal{M}$ from \eqref{eq:IIM}, where $\mI_1(\x)$ and $\mI_2(\x)$ are given measured images, the deformation $\uf(\x,\af)$ satisfies the deformation model \eqref{eq:bvp}, and the domain $\Omega$ satisfies \eqref{domain_condition} for a given admissible set $\mathcal{M}$.}}}
\vspace{15pt}

Stated like this, we can observe that our considered quasi-static elastography problem is essentially a combined image registration and parameter estimation problem; cf.\ the schematic of  Figure~\ref{fig-1-schematic}.

\section{The Intensity-based Inversion Method}\label{sect_IIM}

In this section, we introduce the \emph{intensity-based inversion method (IIM)} for solving the quasi-static elastography problem, which we then analyse within the framework of inverse problems. For this, we recall the connection \eqref{eq:IIM} between the measured images $\mI_1$, $\mI_2$, and the material parameters $\af$, i.e.,
    \begin{equation*}
        (\mI_2 \circ G(\af))(\x)  = \mI_1(\mathbf{x}) \,, \qquad \forall \, \x \in \Omega \,.
    \end{equation*}    
Now, assuming that both $\mI_1 \in L_2(\Omega)$ and $\mI_2 \circ G(\af) \in L_2(\Omega)$ for all $\af \in \mathcal{M}$, one possible approach for determining the material parameters $\af$ from $\mI_1$ and $\mI_2$ is to solve the minimization problem 
    \begin{equation*}
        \min\limits_{\af \in \mathcal{M}} \; \norm{\mI_2 \circ G(\af) - \mI_1}^2_{L_2(\Omega)}
    \end{equation*}  
However, the quasi-static elastography problem is ill-posed, and thus the above minimization problem is expected to be unstable, in particular if $\mI_1$ and $\mI_2$ contain measurement noise. Hence, the problem has to be stabilized, e.g., by adding a regularization functional $\mathcal{R}$. Doing so, we arrive at the IIM
    \begin{equation}\label{eq:mp1}
        \boxed{\qquad \, \min\limits_{\af \in \mathcal{M}} \; \mathcal{T}_\alpha(\af) \,, \qquad \text{where} \qquad \mathcal{T}_\alpha(\af) := \norm{\mI_2 \circ G(\af)- \mI_1}^2_{L_2(\Omega)} + \alpha \mathcal{R}(\af)\,, \qquad }
    \end{equation}
with $\alpha>0$ being a suitably chosen regularization parameter. In this paper, we consider the choice
    \begin{equation}\label{def_R_norm}
        \mathcal{R}(\af) := \norm{\af - \af_0}_X^2 \,,
    \end{equation}
where $X$ is a separable Hilbert space with $\mathcal{M} \subseteq X$ and $\af_0 \in X$. While this choice is sufficient for the applications we have in mind, we emphasize that other choices of $\mathcal{R}$ such as the TV-norm are also possible. These then require mainly technical modifications to the analysis presented in Section~\ref{sect_convergence}.

Next, note that for $\mathcal{T}_\alpha$ to be well-defined, we need that $\mI_2 \circ G(\af) \in L_2(\Omega)$ for all $\af \in \mathcal{M}$. To be able to guarantee this, we now make the following minimal set of assumptions for the IIM:

\begin{assumption}[Minimal assumptions for the IIM] \label{ass_minimal} \hfill
\begin{enumerate}
    \item The domains $\Omega_1, \Omega_2 \subset \R^2$ are open bounded subsets of $\R^2$ with Lipschitz continuous boundaries.
    \item The functional $\mathcal{R}$ is defined by \eqref{def_R_norm}, where $X$ is a separable Hilbert space, $\mathcal{M} \subseteq X$, and $\af_0 \in X$.
    \item The measured images $\mI_1$ and $\mI_2$ satisfy $\mI_1 \in L_2(\Omega_1)$ and $\mI_2 \in L_\infty(\Omega_2)$, respectively.
    \item For each $\af \in \mathcal{M}$, there exists a unique solution $\uf(x,\af) \in L_2(\Omega_1)$ of the set of equations \eqref{eq:bvp}.
    \item The domain $\Omega \subset \R^2$ is bounded, open, with a Lipschitz continuous boundary, and satisfies\eqref{domain_condition}. 
\end{enumerate}
\end{assumption}

Using the above assumptions, we obtain the following well-definedness result:

\begin{lemma}
Let Assumption~\ref{ass_minimal} hold. Then, $\mathcal{T}_\alpha(\af)$ introduced in \eqref{eq:mp1} is well-defined for all $\af \in \mathcal{M}$.
\end{lemma}
\begin{proof}
Due to \eqref{domain_condition}, we have $\Omega \subset \Omega_1$, and thus $\mI_1 \in L_2(\Omega_1)$ implies  $\mI_1 \in L_2(\Omega)$. Furthermore, by its definition \eqref{def_R_norm}, $\mathcal{R}$ is well defined for all $\af \in X \supseteq \mathcal{M}$. Hence, it remains to show that $\mI_2 \circ G(\af) \in L_2(\Omega)$ for all $\af \in \mathcal{M}$. For this, note that for any $\af \in \mathcal{M}$ there holds $\uf(\cdot,\af) \in L_2(\Omega_1)$, and thus, since $\Omega_1$ is bounded, that $G(\af) \in L_2(\Omega_1)$. Furthermore, due to \eqref{domain_condition} there holds $G(\af)(\x) \in \Omega_2$ for all $\x \in \Omega$. Hence, 
    \begin{equation*}
        \norm{\mI_2 \circ G(\af)}_{L_2(\Omega)} 
        \leq 
        \abs{\Omega} \norm{\mI_2 \circ G(\af)}_{L_\infty(\Omega)} 
        \leq \abs{\Omega} \norm{\mI_2}_{L_\infty(\Omega_2)} 
        < \infty \,,
    \end{equation*}
where we have used $\mI_2 \in L_\infty(\Omega_2)$. This shows that $\mI_2 \circ G(\af) \in L_2(\Omega)$, completing the proof. 
\end{proof}

Note that the main conceptual advantage of the IIM  defined in \eqref{eq:IIM} is its combined treatment of the image registration and parameter estimation steps typically found in other elastography approaches. In particular, the displacement field $\uf(x,\af)$ is always a physically plausible displacement field for any $\af \in \mathcal{M}$, which cannot be easily guaranteed in split approaches, where the image registration needs to be supplemented by physical priors. Moreover, the restriction to physically plausible displacement fields also strongly limit the non-uniqueness typically observed in other image registration approaches. 

Furthermore, note that Assumption~\ref{ass_minimal} places only very minimal requirements on the measured images $\mI_1$ and $\mI_2$. While in the upcoming convergence analysis we have to increase the required smoothness of these images, in a numerical implementation of the IIM this increase is not necessary. Also, note that our assumption $\af \in \mathcal{M} \subset X$ encompasses the commonly encountered setting
    \begin{equation*}
        \af(\x) = \sum_{k=1}^{K} a_k \chi_{D_k}(\x) \,,
    \end{equation*}
where the domains $D_k \subseteq \Omega_1$ are known, but the constant parameters $a_k \in \mathcal{M} \subseteq \R^n$ are unknown. This corresponds to a sample with known inclusion locations but unknown material parameter values.

\section{Convergence analysis of the IIM}\label{sect_convergence}

In this section, we present a convergence analysis of the IIM defined in \eqref{eq:mp1} within the framework of inverse problems \cite{Engl_Hanke_Neubauer_1996,Scherzer_Grasmair_Grossauer_Haltmeier_Lenzen_2008}. As noted before, the quasi-static elastography problem is ill-posed, and thus in particular unstable with respect to noise in the measured images $\mI_1$ and $\mI_2$. Hence, in the IIM the regularization functional $\mathcal{R}$ was added in order to stabilize the problem. The resulting functional 
    \begin{equation*}
        \mathcal{T}_\alpha(\af) = \norm{\mI_2 \circ G(\af)- \mI_1}^2_{L_2(\Omega)} + \alpha \mathcal{R}(\af)
    \end{equation*}
closely resembles a classic Tikhonov functional \cite{Engl_Hanke_Neubauer_1996,Scherzer_Grasmair_Grossauer_Haltmeier_Lenzen_2008}. In fact, together with the definition
    \begin{equation}\label{def_F}
        \mathcal{F}(\af) : \, D(\mathcal{F}) =: \mathcal{M} \subseteq X \to Y := L_2(\Omega) \,, \qquad \af \mapsto \F(\af) := \mI_2 \circ G(\af) \,,
    \end{equation}
the functional $\mathcal{T}_\alpha$ in the IIM \eqref{eq:mp1} can be written in the standard form
    \begin{equation*}
        \mathcal{T}_\alpha(\af) = \norm{\F(\af) - \mI_1}^2_{Y} + \alpha \norm{\af - \af_0}_X^2 \,,
    \end{equation*}
to which the classical theory of Tikhonov regularization may be applied \cite{Engl_Hanke_Neubauer_1996,Scherzer_Grasmair_Grossauer_Haltmeier_Lenzen_2008}. However, one key difference to the classical setting is that in our case, both measured images $\mI_1$ and $\mI_2$ are typically contaminated by noise, and thus by definition the operator $\F$ is noisy as well. Fortunately, this setting was already implicitly considered in \cite{Neubauer_Scherzer_1990}, as well as in a more general Banach space setting in \cite{Lu_Fleming_2012}. In this paper, we make use of these results to obtain a convergence analysis of our proposed IIM.

In order to indicate that the measured images contain noise, from now on we write $\mI_1^\delta$ and $\mI_2^\delta$ instead of $\mI_1$ and $\mI_2$, respectively. Furthermore, we also consider the case that the displacement field $\uf(\x,\af)$ is contaminated by noise as well. This may for example be the case if some or all of the involved quantities in \eqref{eq:bvp} are only known inexactly. Correspondingly, we write $\uf^\delta(\x,\af)$, $G^\delta(\af)$, and set
    \begin{equation}\label{def_Fd}
        \F^\delta(\af) : \, D(\F^\delta) =: \mathcal{M} \subseteq X \to Y := L_2(\Omega) \,, \qquad \af \mapsto \F^\delta(\af) := \mI_2^\delta \circ G^\delta(\af) \,.
    \end{equation}   
Furthermore, in the upcoming convergence analysis, we distinguish between two specific noise cases:
    \begin{enumerate}
        \item \textbf{Restricted noise case}: Only the image $\mI_1^\delta$ is contaminated by noise, and thus the IIM reads
            \begin{equation}\label{def_Tad_restricted}
                \min\limits_{\af \in \mathcal{M}} \mathcal{T}_\alpha^\delta(\af) \,, 
                \qquad \text{where} \qquad 
                \mathcal{T}_\alpha^\delta(\af) := \norm{\F(\af) - \mI_1^\delta}^2_{L_2(\Omega)} + \alpha \norm{\af - \af_0}_X^2 \,.
            \end{equation}
        \item \textbf{Full noise case}: Both $\mI_1^\delta$, $\mI_2^\delta$, and $\uf^\delta(x,\af)$ are contaminated by noise, and thus the IIM reads
            \begin{equation}\label{def_Tad_full}
                \min\limits_{\af \in \mathcal{M}} \mathcal{T}_\alpha^\delta(\af) \,, 
                \qquad \text{where} \qquad 
                \mathcal{T}_\alpha^\delta(\af) := \norm{\F^\delta(\af) - \mI_1^\delta}^2_{L_2(\Omega)} + \alpha \norm{\af - \af_0}_X^2 \,.
            \end{equation}
    \end{enumerate}
For both the restricted and full noise case, we require the following minimal set of model assumptions:

\begin{assumption}[Minimal model assumptions for the IIM] \label{ass_model_minimal} \hfill
\begin{enumerate}
    \item The operator $\F(\af) = \mI_2 \circ G(\af)$ as defined in \eqref{def_F} is continuous and weakly sequentially closed.
    \item For a given (noise-free) image $\mI_1$ there exists an $\af_0$-minimum-norm-solution (MNS) $\af^*$, i.e.,
        \begin{equation*}
            \mI_2 \circ G(\af^*) = \mI_1 \,,
            \qquad \text{and} \qquad
            \norm{\af^*-\af_0}_X = \min_{\af\in\mathcal{M}}\Kl{ \norm{\af-\af_0}_X \, \vert \, \mI_2 \circ G(\af^*) = \mI_1 } \,.
        \end{equation*}
    \item There is a noise level $\delta_{\mI_1} > 0$ such that  
        \begin{equation*}
            \norm{ \mI_1 - \mI_1^\delta }_\LtO \leq \delta_{\mI_1} \,.
        \end{equation*}
\end{enumerate}
\end{assumption}

\begin{remark}
Clearly, the restricted noise case is only one specific instance of the full noise case. The main benefit of considering these two separately is that the convergence analysis of the restricted noise case requires less stringent assumptions than that of the full noise case. While in practice both $\mI_1^\delta$ and $\mI_2^\delta$ are contaminated by noise, one typically selects one of them, $\mI_1^\delta$ or $\mI_2^\delta$, as the reference image. Hence, if in addition the deformation $\uf$ is noise-free, one may informally ascribe the noise in the reference image $\mI_1^\delta$ to $\mI_2^\delta$, and thus approximately transform the full into the restricted noise case.
\end{remark}

\subsection{Convergence analysis I: restricted noise case}

First, we consider the restricted noise case \eqref{def_Tad_restricted}, for which we obtain the following convergence result.

\begin{theorem}\label{theorem_convergence_restricted}
Let Assumption~\ref{ass_minimal} and \ref{ass_model_minimal} hold. Furthermore, let $\alpha = \alpha(\delta_{\mI_1})$ be chosen such that
    \begin{equation*}
        \alpha(\delta_{\mI_1}) \to 0 \,,
        \qquad
        \text{and}
        \qquad
        \frac{(\delta_{\mI_1})^2}{\alpha(\delta_{\mI_1})} \to 0 \,,
        \qquad
        \qquad
        \text{for}
        \qquad
        \delta_{\mI_1}\to 0 \,.
    \end{equation*}
Then every sequence $\Kl{\af_k := \af_{\alpha(\delta_k)}^{\delta_k}}_{k\in\N}$ with $\delta_k \to 0$ as $k\to \infty$, and where $\af_k$ is a minimizer of \eqref{def_Tad_restricted} has a convergent subsequence $\boldsymbol{a}_{n_k}$. Furthermore, the limit of every convergent subsequence is an $\af_0$ minimum norm solution. Moreover, if the $\af_0$ minimum norm solution $\af^\dagger$ is unique, then
    \begin{equation*}
        \lim\limits_{\boldsymbol{\delta}_k \to 0} \af_{\alpha(\boldsymbol{\delta}_k)}^{\boldsymbol{\delta}_k} = \af^\dagger\,.
    \end{equation*}
\end{theorem}
\begin{proof}
Using Assumption~\ref{ass_minimal} and \ref{ass_model_minimal}, this follows from classic Tikhonov convergence results \cite{Engl_Hanke_Neubauer_1996}.
\end{proof}

For establishing convergence rates, we require additional differentiability assumptions on $\mI_2$ and $\uf$.

\begin{assumption}[Differentiability assumptions for the IIM] \label{ass_differentiability} \hfill
\begin{enumerate}
    \item The (noise-free) image $\mI_2$ is continuously Fr\'echet differentiable, satisfying $\mI_2 \in W^{1,\infty}(\Omega_2)$ with 
        \begin{equation}\label{Lipschitz_I1}
            C_L := \norm{\mI_2}_{W^{1,\infty}(\Omega_2)} < \infty \,.
        \end{equation}
    \item The function $\uf$ is continuously Fr\'echet differentiable wrt.\ $\af$, and there are $C_1^*, C_2^*,C_3^* > 0$ with
        \begin{align}
            \norm{\uf'(\af^*) \hf}_{L_\infty(\Omega)} &\leq C_1^* \norm{\hf}_X \,, \label{Lipschitz_u_infty}
            \\
            \norm{ \uf(\af^*) - \uf(\af)}_{L_2(\Omega)} &\leq C_2^* \norm{\af^* - \af}_X \,, \label{Lipschitz_u}
            \\
            \norm{\uf'(\af^*)\hf - \uf'(\af)\hf}_{L_2(\Omega)} &\leq C_3^* \norm{\af^*-\af}_X \norm{\hf}_X \,, \label{Lipschitz_Du}
        \end{align}
    for all $\hf \in \mathcal{M}$ and $\af \in \mathcal{M} \cap B_\eps(\af^*)$, where $\af^*$ is an $\af_0$-MNS and $\epsilon>2\norm{\af^* - \af_0}_X$.
\end{enumerate}
\end{assumption}

Using these assumptions, we now establish some necessary derivative estimates for the operator $\mathcal{F}$.

\begin{proposition}\label{prop_differentiability}
Let Assumption~\ref{ass_minimal}, \ref{ass_model_minimal}, and \ref{ass_differentiability} hold. Then the operator $\F$ defined in \eqref{def_F} satisfies
    \begin{equation}\label{F_continuous}
         \norm{\F(\af_1) - \F(\af_2)}_{L_2(\Omega)} \leq C_L \norm{ \uf(\af_1) - \uf (\af_2)}_{L_2(\Omega)}  \,,
        \qquad 
         \forall \, \af_1,\af_2 \in \mathcal{M} \,,
     \end{equation}
and is continuous and continuously Fr\'echet differentiable with
    \begin{equation}\label{F_Frechet}
        \F'(\af) \hf = (\mI_2' \circ G(\af)) \uf'(\af) \hf \,,
        \qquad 
        \forall \, \af \,, \hf \in \mathcal{M} \,.
    \end{equation}
Furthermore, for all $\hf \in \mathcal{M}$ and $\af \in \mathcal{M} \cap B_\eps(\af^*)$ there holds
    \begin{equation}\label{DF_Lipschitz}
        \norm{\F'(\af^*)\hf - \F'(\af)\hf}_{L_2(\Omega)} \leq C \norm{\af^* - \af}_{X} \norm{\hf}_X \,,
    \end{equation}
where $C:= C_L \kl{C_1^*C_2^* + C_3^*}$, $\af^*$ is an $\af_0$-MNS, and $\epsilon>0$ is as in Assumption~\ref{ass_differentiability}.
\end{proposition}
\begin{proof}
First, let $\af_1, \af_2 \in \mathcal{M}$ be arbitrary but fixed. Then together with \eqref{Lipschitz_I1} there follows
    \begin{equation*}
    \begin{split}
        \norm{\F(\af_1) - \F(\af_2)}_{L_2(\Omega)} 
        &= 
        \norm{\mI_2 \circ G(\af_1) - \mI_2 \circ G(\af_2)}_{L_2(\Omega)} 
        \\
        &\overset{\eqref{Lipschitz_I1}}{\leq} 
        C_L \norm{ G(\af_1) - G(\af_2)}_{L_2(\Omega)} 
        =  
        C_L \norm{ \uf(\af_1) - \uf (\af_2)}_{L_2(\Omega)} \,,
    \end{split}
    \end{equation*}
which establishes \eqref{F_continuous}, and thus the continuity of $\uf$ also implies the continuity of $\F$. Next, since both $\mI_2$ and $\uf$ are also assumed to be continuously Fr\'echet differentiable, the Fr\'echet differentiability of $\F$ and \eqref{F_Frechet} follow from the chain rule. Finally, for $\hf \in \mathcal{M}$ and $\af \in \mathcal{M} \cap B_\eps(\af^*)$ we have
    \begin{equation}\label{helper_DF_01}
    \begin{split}
        &\norm{\F'(\af^*)\hf - \F'(\af)\hf}_{L_2(\Omega)}
        =
        \norm{(\mI_2' \circ G(\af^*)) \uf'(\af^*) \hf - (\mI_2' \circ G(\af)) \uf'(\af) \hf}_{L_2(\Omega)}
        \\
        & \qquad 
        \leq 
        \norm{(\mI_2'\circ G(\af^*)) \uf'(\af^*) \hf - (\mI_2' \circ G(\af)) \uf'(\af^*) \hf}_{L_2(\Omega)}
        \\
        & \qquad \qquad 
        +
        \norm{(\mI_2'\circ G(\af)) \uf'(\af^*) \hf - (\mI_2'\circ G(\af)) \uf'(\af) \hf}_{L_2(\Omega)} \,.
    \end{split} 
    \end{equation}
For the first term, we use \eqref{Lipschitz_I1}, \eqref{Lipschitz_u_infty}, and \eqref{Lipschitz_u} to obtain
    \begin{equation*}
    \begin{split}
        &\norm{(\mI_2' \circ G(\af^*)) \uf'(\af^*) \hf - (\mI_2'\circ G(\af)) \uf'(\af^*) \hf}_{L_2(\Omega)}
        =
        \norm{\kl{ \mI_2' \circ G(\af^*) - \mI_2'\circ G(\af)} \uf'(\af^*) \hf}_{L_2(\Omega)}
        \\
        & \qquad 
        \leq
        \norm{ \mI_2' \circ G(\af^*) - \mI_2'\circ G(\af)}_{L_2(\Omega)}
        \norm{\uf'(\af^*) \hf}_{L_\infty(\Omega)}
        \\
        & \qquad 
        \leq
        C_L \norm{ G(\af^*) - G(\af)}_{L_2(\Omega)}
        C_1^* \norm{\hf}_{X}
        \\
        & \qquad 
        =
        C_L C_1^* \norm{ \uf(\af^*) - \uf(\af)}_{L_2(\Omega)} \norm{\hf}_{X}
        \\
        & \qquad 
        \leq
        C_L C_1^* C_2^*  \norm{ \af^* - \af}_{X}
        \norm{\hf}_{X} \,.
    \end{split}
    \end{equation*}
For the second term, we use \eqref{Lipschitz_I1} and \eqref{Lipschitz_Du} to obtain
    \begin{equation*}
    \begin{split}
        &\norm{(\mI_2'\circ G(\af)) \uf'(\af^*) \hf - (\mI_2'\circ G(\af)) \uf'(\af) \hf}_{L_2(\Omega)} 
        =
        \norm{\kl{\mI_2'\circ G(\af)}\kl{\uf'(\af^*) \hf - \uf'(\af) \hf}}_{L_2(\Omega)} 
        \\
        \qquad 
        &\leq 
        \norm{\mI_2'\circ G(\af)}_{L_\infty(\Omega)} \norm{\uf'(\af^*) \hf - \uf'(\af) \hf}_{L_2(\Omega)} 
        \\
        \qquad
        &\leq 
        C_L C_3^* \norm{\af^*-\af}_X \norm{\hf}_X
    \end{split}
    \end{equation*}
Hence, inserting these estimates into \eqref{helper_DF_01} we obtain
    \begin{equation*}
    \begin{split}
        \norm{\F'(\af^*)\hf - \F'(\af)\hf}_{L_2(\Omega)} 
        \leq 
        C_L \kl{C_1^*C_2^* + C_3^*} \norm{\af^*-\af}_X \norm{h}_X
        \,,
    \end{split}
    \end{equation*}
which now yields \eqref{DF_Lipschitz} and thus completes the proof.
\end{proof}

With this, we now obtain the following convergence rate results in the restricted noise case \eqref{def_Tad_restricted}.

\begin{theorem}
Let Assumption~\ref{ass_minimal}, \ref{ass_model_minimal}, and \ref{ass_differentiability} hold, and assume that there is a $w \in L_2(\Omega)$ with
    \begin{equation*}
        \af^\dagger - \af_0  = \mathcal{F}'(\af^\dagger)^*w \,,
        \qquad \text{and} \qquad
        C \norm{w} < 1 \,,
    \end{equation*}
where $C$ is as in Proposition~\ref{prop_differentiability}. Then for $\alpha \sim \delta_{\mI_1}$, a minimizer $\af_\alpha^\delta$ of \eqref{def_Tad_restricted} satisfies
    \begin{equation*}
        \norm{\af_\alpha^\delta - \af^\dagger}_{X} = \mathcal{O}\kl{\sqrt{\delta_{\mI_1}}} \,,
        \qquad \text{and} \qquad 
        \norm{\mI_2\circ G(\af_\alpha^\delta) - \mI_1^{\delta}}_{L_2(\Omega)} = \mathcal{O}\kl{\delta_{\mI_1}} \,.
    \end{equation*}
\end{theorem}
\begin{proof}
Due to Proposition~\ref{prop_differentiability}, all assumptions of the classic convergence rate analysis for nonlinear Tikhonov regularization (see, e.g., \cite[Theorem~10.4]{Engl_Hanke_Neubauer_1996}) are satisfied, which thus yields the assertion.
\end{proof}

\subsection{Convergence analysis II: full noise case}

Next, we consider the full noise case \eqref{def_Tad_full}, for which we require some additional noise assumptions.

\begin{assumption}[Additional noise assumptions for the IIM in the full noise case] \label{ass_noise} \hfill
\begin{enumerate}
    \item There is a function $\xi \in L_\infty(\Omega_2)$ and a noise level $\delta_{\mI_2} > 0$ such that
        \begin{equation}\label{cond_noise_I1}
            \mI_2^\delta = \mI_2 + \delta_{\mI_2} \xi \,, \qquad \text{where} \qquad \norm{\xi}_{L_\infty(\Omega_2)} \leq 1 \,.
        \end{equation}
    \item There is a noise level $\delta_G > 0$ such that
        \begin{equation}\label{cond_noise_model}
            \sup_{\af \in \mathcal{M} \cap B_\eps(\af^*)} \norm{\uf(\cdot,\af) - \uf^\delta(\cdot, \af)}_{L_2(\Omega)} 
            = 
            \sup_{\af \in \mathcal{M}  \cap B_\eps(\af^*)} \norm{G(\af) - G^\delta(\af)}_{L_2(\Omega)}
            \leq \delta_G \,.
        \end{equation}
    \item There is a noise level $\Bar{\delta} > 0$ such that for all $\delta \leq \Bar{\delta}$ there holds
        \begin{equation}\label{cond_noise_domain}
            \forall \, \af \in \mathcal{M} \cap B_\eps(\af^*) \,\, \forall \, \x \in \Omega \, : \, G^\delta(\af)(\x) \in \Omega_2 \,.
    \end{equation}
\end{enumerate}
\end{assumption}

\begin{remark}
Note that \eqref{cond_noise_domain} is mainly a technical assumption, which is only required to guarantee that the noisy operator $F^\delta(\af) := \mI_2^\delta \circ G^\delta(\af)$ is well defined for sufficiently small noise levels $\delta$. One way to guarantee it is to assume that the noise-free operator $G(\af)$ maps $\Omega$ only onto a subset $\Tilde{\Omega}_2 \subset \Omega_2$ with $\operatorname{dist}(\Tilde{\Omega}_2,\partial \Omega_2) > 0$. Then, if in \eqref{cond_noise_model} the $L_2$ norm is replaced by the $L_\infty$ norm, \eqref{cond_noise_domain} is satisfied with
    \begin{equation*}
        \Bar{\delta} = \operatorname{dist}(\tilde{\Omega}_2,\partial \Omega_2) / \delta_G \,.
    \end{equation*}
\end{remark}

Using the noise Assumption~\ref{ass_noise}, we now derive a noise bound for $\F^\delta$ in the following proposition.

\begin{proposition}\label{prop_noise}
Let Assumption~\ref{ass_minimal}, \ref{ass_model_minimal}, and \ref{ass_noise} hold. Then, for $\F$ and $\F^\delta$ defined in \eqref{def_F} and \eqref{def_Fd}, respectively, and with $\delta_\F = \delta_\F(\delta_{\mI_2},\delta_G) := \delta_{\mI_2} \abs{\Omega} + C_L \delta_G$ there holds
    \begin{equation*}
        \sup_{\af \in \mathcal{M} \cap B_\eps(\af^*)} \norm{\F(\af) - \F^\delta(\af)}_{L_2(\Omega)} \leq \delta_\F  \,.
    \end{equation*}
\end{proposition}
\begin{proof}
First, we use the definitions of $\F$ and $\F^\delta$ as well as \eqref{cond_noise_I1} to obtain
    \begin{equation*}
    \begin{split}
        \norm{\F(\af) - \F^\delta(\af)}_{L_2(\Omega)} 
        &=
        \norm{\mI_2 \circ G(\af) - \mI_2^\delta \circ G^\delta(\af)}_{L_2(\Omega)} 
        \overset{\eqref{cond_noise_I1}}{=} 
        \norm{\mI_2 \circ G(\af) - (\mI_2 + \delta_{\mI_2} \xi) \circ G^\delta(\af)}_{L_2(\Omega)} 
        \\
        &\leq
        \norm{\mI_2 \circ G(\af) - \mI_2 \circ G^\delta(\af)}_{L_2(\Omega)} + \delta_{\mI_2} \norm{\xi \circ G^\delta(\af)}_{L_2(\Omega)} \,. 
    \end{split}
    \end{equation*}
Hence, using the Lipschitz estimate \eqref{Lipschitz_I1} as well as $\norm{\xi}_{L_\infty(\Omega_2)} \leq 1$ from \eqref{cond_noise_I1}, we find that
    \begin{equation*}
        \norm{\F(\af) - \F^\delta(\af)}_{L_2(\Omega)} 
        \leq
        C_L \norm{G(\af) - G^\delta(\af)}_{L_2(\Omega)} + \delta_{\mI_2} \abs{\Omega} \,, 
    \end{equation*}
which after taking the supremum over $\af \in \mathcal{M} \cap B_\eps(\af^*)$ and using \eqref{cond_noise_model} now yields the assertion.   
\end{proof}

With this estimate, we can now derive the following convergence result in the full noise case \eqref{def_Tad_full}.

\begin{theorem}\label{theorem_convergence}
Let Assumption~\ref{ass_minimal}, \ref{ass_model_minimal}, \ref{ass_differentiability}, and \ref{ass_noise} hold, and define $\boldsymbol{\delta} := (\delta_{\mI_1},\delta_{\mI_2},\delta_G)$. Furthermore, let $\alpha = \alpha(\boldsymbol{\delta})$ be chosen such that for $\boldsymbol{\delta} \to 0$ there holds
    \begin{equation*}
        \alpha(\boldsymbol{\delta}) \to 0 \,,
        \qquad
        \frac{(\delta_{\mI_1})^2}{\alpha(\boldsymbol{\delta})} \to 0 \,,
        \qquad
        \text{and}
        \qquad
        \frac{(\delta_{\mI_2} + \delta_G)^2}{\alpha(\boldsymbol{\delta})} \to 0 \,.
    \end{equation*}
Then every sequence $\Kl{\af_k := \af_{\alpha(\boldsymbol{\delta}_k)}^{\boldsymbol{\delta}_k}}_{k\in\N}$ with $\boldsymbol{\delta}_k \to 0$ as $k\to \infty$, and where $\af_k$ is a minimizer of \eqref{def_Tad_full} has a convergent subsequence $\boldsymbol{a}_{n_k}$. Furthermore, the limit of every convergent subsequence is an $\af_0$ minimum norm solution. Moreover, if the $\af_0$ minimum norm solution $\af^\dagger$ is unique, then
    \begin{equation*}
        \lim\limits_{\boldsymbol{\delta}_k \to 0} \af_{\alpha(\boldsymbol{\delta}_k)}^{\boldsymbol{\delta}_k} = \af^\dagger\,.
    \end{equation*}
\end{theorem}
\begin{proof}
Due to Proposition~\ref{prop_differentiability} and \ref{prop_noise}, all assumptions of \cite[Theorem~2.1]{Neubauer_Scherzer_1990} are satisfied, which yields the assertion. Note that while \cite{Neubauer_Scherzer_1990} considers finite-dimensional approximations of Tikhonov functionals, the results and proofs themselves also apply in our infinite dimensional setting with noisy operator.    
\end{proof}

Finally, we also obtain the following convergence rate results in the full noise case \eqref{def_Tad_full}.

\begin{theorem}\label{theorem_rates}
Let Assumption~\ref{ass_minimal}, \ref{ass_model_minimal}, \ref{ass_differentiability}, and \ref{ass_noise} hold, and assume that there is a $\wf \in L_2(\Omega)$ with
    \begin{equation*}
        \af^\dagger - \af_0 = \mathcal{F}'(\af^\dagger)^*\wf \,,
        \qquad \text{and} \qquad
        C \norm{\wf} < 1 \,,
    \end{equation*}
where $C$ is as in Proposition~\ref{prop_differentiability}. Furthermore, assume that $\mathcal{O}(\delta_{\mI_2} + \delta_G) = \mathcal{O}(\delta_{\mI_1})$. Then for $\alpha \sim \delta_{\mI_1}$, a minimizer $\af_\alpha^\delta$ of \eqref{def_Tad_full} satisfies
    \begin{equation*}
        \norm{\af_\alpha^\delta - \af^\dagger}_{X} = \mathcal{O}\kl{\sqrt{\delta_{\mI_1}}} \,.
    \end{equation*}
\end{theorem}
\begin{proof}
Due to Proposition~\ref{prop_differentiability} and \ref{prop_noise}, the assertion follows from \cite[Theorem~2.3(a)]{Neubauer_Scherzer_1990}.
\end{proof}

\section{Extensions and practical considerations}\label{sect_extensions_practical}

In this section, we discuss several practically relevant aspects of the IIM, such as the inclusion of convex constraints, multiple or time-dependent measurements, and the (approximate) minimization of \eqref{eq:mp1}.

\subsection{General Material Models and Approximate Minimization}

First, note that our IIM \eqref{eq:mp1} is not restricted to the 2D case of $\Omega_1,\Omega_2 \subset \R^2$, but that instead one can consider two open, bounded subsets $\Omega_1,\Omega_2 \subset \R^N$ for arbitrary $N \geq 1$. This generalization can, e.g., be useful when the imaging modality used to record the images $\mI_1,\mI_2$ yields volumetric (3D) data, such as in certain settings of OCE. This change of dimension does not affect the convergence analysis presented above, but may implicitly change the severity of the required convergence assumptions. 

Next, note that instead of the boundary value problem \eqref{eq:bvp} a more general material model for determining the quasi-static deformation $\uf$ can be used, such as a fully implicit equation  of the form
    \begin{equation}\label{L_general}
        \mathcal{L}(\af,\uf) = 0 \,,
        \qquad \qquad \text{or} \qquad \qquad 
        \mathcal{L}(\af,\uf,\pf) = 0 \,,
    \end{equation}
for $\pf \in \mathcal{P}$, a set of admissible (and potentially function-valued) parameters. As before, it can be easily seen from the proofs of the above results that our convergence analysis of the IIM \eqref{eq:mp1} remains unaffected by this generalization, as long as $\uf$ is uniquely determined by \eqref{L_general} for any admissible parameters $\af \in \mathcal{M}$ and $\pf \in \mathcal{P}$, and satisfies all the required (differentiability) assumptions. Furthermore, note that in the full noise case \eqref{def_Tad_full}, the noise in the deformation $\uf^\delta$ may be due to noisy parameters $\pf^\delta$ (such as measured forces), an imprecise material model $\mathcal{L}^\delta$, or both.

An important practical aspect which can also be included in our analysis of the IIM \eqref{eq:mp1} is that numerically, one cannot compute the minimizer of $\mathcal{T}_\alpha(\af)$ exactly. Instead, one looks for $\af_{\alpha,\eta}^\delta$ satisfying
    \begin{equation*}
    \begin{split}
        \norm{\mI_2^\delta \circ G^\delta(\af_{\alpha,\eta}^\delta)- \mI_1^\delta}^2_{L_2(\Omega)} + \alpha \mathcal{R}(\af_{\alpha,\eta}^\delta)
        \leq 
        \norm{\mI_2^\delta \circ G^\delta(\af)- \mI_1^\delta}^2_{L_2(\Omega)} + \alpha \mathcal{R}(\af) + \eta \,,
        \qquad 
        \forall \, \af \in \mathcal{M} \,,
    \end{split}
    \end{equation*}
where $\eta \geq 0$ is small. Following \cite{Neubauer_Scherzer_1990}, we find that both our convergence and convergence rates results remain valid, given that $\eta/\alpha(\boldsymbol{\delta}) \to 0$ and $\eta = \mathcal{O}(\delta_{\mI_1}^2)$, respectively. In particular, we still have the rate
    \begin{equation*}
        \norm{\af_{\alpha,\eta}^\delta - \af^\dagger}_{X} = \mathcal{O}\kl{\sqrt{\delta_{\mI_1}}} \,.
    \end{equation*}

\subsection{Piecewise constant parameters and convex regularizers}\label{subsect_piecewise}

In the above analysis of the IIM, we have assumed that the regularizer $\mathcal{R}$ is of the form \eqref{def_R_norm}, i.e.,
    \begin{equation*}
        \mathcal{R}(\af) = \norm{\af - \af_0}_X^2 \,,
    \end{equation*}
which in particular covers the classic Sobolev-norm penalties. However, the IIM \eqref{eq:mp1} and its convergence analysis can be generalized to also include general convex regularization functionals $\mathcal{R}$, such as the commonly used TV norm \cite{Scherzer_Grasmair_Grossauer_Haltmeier_Lenzen_2008}. E.g., following \cite{Scherzer_Grasmair_Grossauer_Haltmeier_Lenzen_2008,Lu_Fleming_2012}, assume that the variational source condition
    \begin{equation*}
        \beta D_{\xi^\dagger}(\af,\af^\dagger) \leq \mathcal{R}(\af) - \mathcal{R}(\af^\dagger) + \varphi\kl{ \norm{ F(\af) - F(\af^\dagger) }^2 } \,, 
        \qquad \forall \, \af \in \mathcal{M} \,,
    \end{equation*}
holds for some parameter $\beta > 0$ and $\xi^\dagger \in \partial \mathcal{R}(\af^\dagger)$, where $\varphi : [0,\infty) \to [0, \infty)$ is a monotonously increasing and concave index function and $D_{\xi^\dagger}$ denotes the Bregman distance with respect to $\mathcal{R}$ \cite{Bauschke_Combettes_2017}. Then under suitable assumptions on the regularization functional $\mathcal{R}$, and for an appropriate choice of the regularization parameter $\alpha$ (e.g., $\alpha \sim \delta$, cf.~\cite{Scherzer_Grasmair_Grossauer_Haltmeier_Lenzen_2008,Lu_Fleming_2012,Flemming_2011}), we obtain the optimal convergence rate
    \begin{equation*}
        D_{\xi^\dagger}(\af_\alpha^\delta,\af^\dagger) = \mathcal{O}\kl{\varphi\kl{\kl{\delta_{\mI_1}}^2}} \,.
    \end{equation*}
    
The choice of the TV norm for the regularization functional $\mathcal{R}$ suggests itself in particular if the material parameter $\af$ is known to be piecewise constant. This is e.g.\ the case if the considered sample consists of a homogeneous or layered medium with several uniform inclusions. In some elastography experiments, the location and shape of these inclusions can be determined a-priori from the measured images $\mI_1^\delta$, $\mI_2^\delta$. Hence, one then only needs to estimate the coefficients $a_k \in \mathcal{M} \subseteq \R^n$ in the expansion
    \begin{equation}\label{af_piecewise}
        \af(\x) = \sum_{k=1}^K a_k \chi_{D_k}(\x) \,,
    \end{equation}
where the domains $D_k \subseteq \Omega_1$ correspond to the inclusion and background (layers) of the material. In this setting, the admissible set $\mathcal{M} \subseteq \R^n$ typically encodes feasible ranges of the unknown material parameters $a_k$. Hence, the modified \emph{IIM for known inclusion locations} is then defined as follows:
    \begin{equation*}
        \boxed{\,\, \, \min\limits_{a_k \in \mathcal{M} \subseteq \R^n} \; \mathcal{T}_{\alpha,K}^\delta(a_k ) \,, \quad \text{where} \quad \mathcal{T}_{\alpha,K}^\delta(a_k ) :=  \norm{\mI_2^\delta \circ G^\delta\kl{\sum_{k=1}^K a_k \chi_{D_k}}- \mI_1^\delta}^2_{L_2(\Omega)} + \alpha \sum_{k=1}^M \abs{a_k}^2\,. \,\, } 
    \end{equation*}
In this modified version, the IIM is a finite dimensional optimization problem which can often be solved more efficiently than the standard IIM \eqref{eq:mp1} in the continuous setting. Note that while our convergence and convergence rate results technically carry over to this modified setting, the piecewise constant nature of $\af$ as defined in \eqref{af_piecewise} severely limits the material models \eqref{eq:bvp} for which the required differentiability assumptions on $\uf$ hold. However, the convergence result for the restricted noise case in Theorem~\ref{theorem_convergence_restricted} does not require differentiability, and thus remains applicable in this modified setting.

\subsection{Multiple or time-dependent measurements}

In the form \eqref{eq:mp1}, the IIM is designed for quasi-static elastography from two measurements $\mI_1$, $\mI_2$. However, it may easily be extended to accommodate for multiple or time-dependent measurements. These measurements may have been obtained with different imaging modalities (OCT, MRI, PAT, Ultrasound), or with the same modality under different experimental conditions such as different induced types of deformation (e.g., plane and tilted compression, or various types of indentation).

To adapt the IIM to this situation, assume first that we are given $M$ different pairs of measured images $(\mI_{1,m}^\delta$, $\mI_{2,m}^\delta)$ with $m \in \{1\,,\dots\,,M\}$, which were obtained in a series of elastography experiments with a mixture of one or several different imaging modalities. Furthermore, depending on the type of experiment, we assume that we are given a (potentially noisy) deformation model $G_m^\delta$ for each of the image pairs, which correspond to specific material models $\mathcal{L}_m$ such as in \eqref{eq:bvp} or \eqref{L_general}. Then analogously to the quasi-static IIM \eqref{eq:mp1}, the \emph{IIM for multiple measurements} can be defined in the following way:
    \begin{equation*}
        \boxed{\qquad \, \min\limits_{\af \in \mathcal{M}} \; \mathcal{T}_{\alpha,M}^\delta(\af) \,, \qquad \text{where} \qquad \mathcal{T}_{\alpha,M}^\delta(\af) := \sum_{m=1}^M \norm{\mI_{2,m}^\delta \circ G_m^\delta(\af)- \mI_{1,m}^\delta}^2_{L_2(\Omega)} + \alpha \mathcal{R}(\af)\,. \qquad } 
    \end{equation*}
Note that with minor modifications, our convergence analysis for the IIM remains valid in this setting.

Finally, we generalize the IIM \eqref{eq:mp1} to the case that the measured images are also time dependent, i.e., $\mI_{1,m} :\Omega_1\times[0,T] \to \R$ and $\mI_{2,m} : \Omega_2 \times [0,T] \to \R$. For this, assume that $\mI_{1,m} \in L_2\kl{[0,T],L_2(\Omega_1)}$ and $\mI_{2,m} \in L_2\kl{[0,T],L_\infty(\Omega_2)}$, which are understood to be standard Bochner-Lebesgue spaces \cite{Evans_1998}. Furthermore, in this case the displacement $\uf_m$ and deformation $G_m$ have to be modeled as time-dependent quantities as well, which in turn also requires the material models $\mathcal{L}_m$ to be time-dependent. With this, the \emph{IIM for multiple and time-dependent measurements} can now be defined as follows:
    \begin{equation*}
        \boxed{\qquad \, \min\limits_{\af \in \mathcal{M}} \; \mathcal{T}_{\alpha,M,T}^{\delta}(\af)  := \sum_{m=1}^M \int_0^T \norm{\mI_{2,m}^\delta(G_m^\delta(\af)(\cdot,t),t)- \mI_{1,m}^\delta(\cdot,t)}^2_{L_2(\Omega)} \, dt + \alpha \mathcal{R}(\af)\,. \qquad } 
    \end{equation*}
Again, our convergence of the IIM can be generalized in a straightforward way to also cover this time-dependent setting. Note also that in this setting, also the material parameters $\af$ may be time dependent, i.e., $\af = \af(\x,t)$, which is easily accommodated by a proper choice of the Hilbert space $X$.

\subsection{Minimization of the IIM functional}

Next, we discuss the minimization of the IIM functional, starting with the noise-free case \eqref{eq:mp1}, i.e., 
    \begin{equation*}
        \mathcal{T}_\alpha(\af) = \norm{\mathcal{F}(\af) - \mI_1}^2_{L_2(\Omega)} + \alpha \mathcal{R}(\af) \,,
        \qquad \text{where} \qquad
        \mathcal{F}(\af) = \mI_2 \circ G(\af) \,.
    \end{equation*}
As we showed in Proposition~\ref{prop_differentiability}, if both $\mI_2$ and $\uf$ are continuously Fr\'echet differentiable, then
    \begin{equation*}
        \F'(\af) \hf = (\mI_2' \circ G(\af)) \uf'(\af) \hf \,,
        \qquad 
        \forall \, \af \,, \hf \in \mathcal{M} \,.
    \end{equation*}
Furthermore, the adjoint of the Fr\'echet derivative of $\mathcal{F}$ can be explicitly calculated, and has the form
    \begin{equation*}
        \F'(\af)^* \wf = \uf'(\af)^*\kl{\kl{\mI_2' \circ G(\af)}\wf} \,, 
        \qquad 
        \forall \, \af \in \mathcal{M} \,, \, \wf \in L_2(\Omega) \,.
    \end{equation*}
Hence, if also the regularization functional $\mathcal{R}$ is differentiable, classic first order methods can be used to minimize $\mathcal{T}_\alpha(\af)$. For example, if $\mathcal{R}$ is as in \eqref{def_R_norm}, then the classic gradient method takes the form 
    \begin{equation*}
        \af_{k+1} = \af_k - \omega_k\kl{ \uf'(\af_k)^*\kl{\kl{\mI_2' \circ G(\af_k)} \kl{\mI_2\circ G(\af_k) - \mI_1}} 
        + (\af_k - \af_0)  
        }\,,
    \end{equation*}
where $\omega_k$ is a suitably chosen stepsize. On the other hand, if $\mathcal{R}$ is not differentiable but convex, then one can e.g.\ use convex optimization algorithms such as subgradient descent \cite{Bauschke_Combettes_2017}, which takes the form
    \begin{equation}\label{LW_prox}
        \af_{k+1} = \prox_{\alpha \mathcal{R} }\kl{ \af_k - \uf'(\af_k)^*\kl{\kl{\mI_2' \circ G(\af_k)} \kl{\mI_2\circ G(\af_k) - \mI_1}}} \,,
    \end{equation}
where $\prox_{\alpha \mathcal{R} }$ denotes the standard proximal operator, defined by
    \begin{equation*}
        \prox_{\alpha \mathcal{R} }\kl{\af} := \arg\min_{\mathbf{b} \in X} \kl{ \frac{1}{2}\norm{\mathbf{b}-\af}_X + \alpha \mathcal{R}(\mathbf{b})} \,.
    \end{equation*}
Furthermore, \eqref{LW_prox} can be supplemented with Nesterov acceleration \cite{Nesterov_1983,Attouch_Peypouquet_2016} for numerical efficiency.

Next, we consider the minimization of the IIM functional in the restricted and full noise case. First, consider the restricted noise case \eqref{def_Tad_restricted}: Since there only the first image $\mI_1$ is assumed to be contaminated by noise, the minimization approaches discussed above remain applicable. However, in the full noise case \eqref{def_Tad_full}, also the second image $\mI_2$ and $G$ may be contaminated by noise. Hence, the differentiability of $\mathcal{F}$ may be lost, which precludes the applicability of the approaches discussed above. In this situation, non-smooth optimization methods may be used instead \cite{Bonnans_Gilbert_Lemarechal_Sagastizabal_2006}, with classic choices such as the Nelder-Mead algorithm performing reasonably well in numerical experiments \cite{Krainz_Sherina_Hubmer_Liu_Drexler_Scherzer_2022}. Alternatively, if $\mI_2$ but not $G$ is contaminated by noise, a pre-smoothing of $\mI_2$ may be used to restore differentiability. While this may seem artificial at first, it should be noted that a certain smoothness of $\mI_1^\delta$, $\mI_2^\delta$ can be expected in practice even for material samples where $\mI_1$, $\mI_2$ are non-smooth in theory, since the reconstruction algorithms used to obtain these measured images often induces smoothness implicitly.

Finally, note that instead of the IIM \eqref{def_Tad_full}, one may also consider the restricted optimization problem
    \begin{equation*}
        \min_{\af \in \mathcal{M}} \mathcal{R}(\af) \,,
        \qquad \text{subject to} \qquad
        \norm{\mI_2^\delta \circ G^\delta(\af) - \mI_1^\delta}_{L_2(\Omega)}^2 \leq \tau \delta \,,
    \end{equation*}
for some parameter $\tau > 1$, which corresponds to Morozov regularization for $\mathcal{F}(\af) = \mI_1$. In this setting, techniques from optimal control may be used for minimization even if $\mI_2^\delta$ and $G^\delta$ are not differentiable.

\section{Application to Linear Elasticity}\label{sect_appl_lin}

In this section, we apply our theoretical results on the IIM to a particular elastography experiment with a specific material model \eqref{eq:bvp}. More precisely, we assume that the displacement field $\uf$ satisfies the quasi-static linear elasticity equations with displacement-traction boundary conditions, i.e.,
    \begin{equation}\label{linear_elastic}
    \begin{alignedat}{3}
        -\operatorname{div}{\sigma_{\lambda,\mu}(\uf)} & = \ff \,, \qquad &&\text{in} \quad \Omega_1 \,, 
        \\ 
        \uf & = \gfD \,, \qquad &&\text{on} \quad \Gamma_D  \,,
        \\
        \sigma_{\lambda,\mu}(\uf) \nv &= \gfT \,,  \qquad &&\text{on} \quad \Gamma_T \,.
    \end{alignedat} 
    \end{equation}
where $\nv$ is an outward unit normal vector of $\partial\Omega_1 = \overline{\Gamma_D \cup \Gamma_T}$, $\gfD$ and $\gfT$ are (potentially non-constant) boundary functions defined on $\Gamma_D$ and $\Gamma_T$, respectively, and the stress tensor $\sigma_{\lambda,\mu}$ is defined by
    \begin{equation*}
        \sigma_{\lambda,\mu}(\uf) := \lambda \operatorname{div}(\uf) I + 2\mu \mathcal{E}(\uf) \,,
        \qquad 
        \text{and} 
        \qquad
        \mathcal{E}(\uf) := \frac{1}{2}(\nabla \uf + \nabla \uf^T) \,.
    \end{equation*}
Here, $I$ is the identity matrix, and $\lambda, \mu$ are the unknown Lam\'e parameters characterizing the medium. 

\begin{figure}[ht!]
    \centering
    \begin{tikzpicture}[x=0.75pt,y=0.75pt,yscale=-0.85,xscale=0.85]
       
        \draw [line width=1.5]  (0,40) -- (220,40) -- (220,190) -- (0,190) -- cycle;

        \draw (100,110) node [anchor=north west][inner sep=0.75pt][font=\Large]{$\Omega_1$};
		\draw (-30,110) node [anchor=north west][inner sep=0.75pt][font=\Large]{$\Gamma_T$};
		\draw (223,110) node [anchor=north west][inner sep=0.75pt][font=\Large]{$\Gamma_T$};
		\draw (100,45) node [anchor=north west][inner sep=0.75pt][font=\Large]{$\Gamma_D$};
		\draw (100,165) node [anchor=north west][inner sep=0.75pt][font=\Large]{$\Gamma_D$};
		\draw (225,20) node [anchor=north west][inner sep=0.75pt][align=left]{$\gf_D$};

        \draw[line width=0.75] (9,10) -- (9,40);
        \draw[line width=0.75] (29,10) -- (29,40);
        \draw[line width=0.75] (49,10) -- (49,40);
        \draw[line width=0.75] (69,10) -- (69,40);
        \draw[line width=0.75] (89,10) -- (89,40);
        \draw[line width=0.75] (109,10) -- (109,40);
        \draw[line width=0.75] (129,10) -- (129,40);
        \draw[line width=0.75] (149,10) -- (149,40);
        \draw[line width=0.75] (169,10) -- (169,40);
        \draw[line width=0.75] (189,10) -- (189,40);
        \draw[line width=0.75] (209,10) -- (209,40);
        \draw [shift={(9,40)}, rotate = -90] [color={rgb, 255:red, 0; green, 0; blue, 0 }][line width=0.75](10.93,-3.29) .. controls (6.95,-1.4) and (3.31,-0.3) .. (-1,0) .. controls (3.31,0.3) and (6.95,1.4) .. (10.93,3.29);
        \draw [shift={(29,40)}, rotate = -90] [color={rgb, 255:red, 0; green, 0; blue, 0 }][line width=0.75](10.93,-3.29) .. controls (6.95,-1.4) and (3.31,-0.3) .. (-1,0) .. controls (3.31,0.3) and (6.95,1.4) .. (10.93,3.29);
        \draw [shift={(49,40)}, rotate = -90] [color={rgb, 255:red, 0; green, 0; blue, 0 }][line width=0.75](10.93,-3.29) .. controls (6.95,-1.4) and (3.31,-0.3) .. (-1,0) .. controls (3.31,0.3) and (6.95,1.4) .. (10.93,3.29);
        \draw [shift={(69,40)}, rotate = -90] [color={rgb, 255:red, 0; green, 0; blue, 0 }][line width=0.75](10.93,-3.29) .. controls (6.95,-1.4) and (3.31,-0.3) .. (-1,0) .. controls (3.31,0.3) and (6.95,1.4) .. (10.93,3.29);
        \draw [shift={(89,40)}, rotate = -90] [color={rgb, 255:red, 0; green, 0; blue, 0 }][line width=0.75](10.93,-3.29) .. controls (6.95,-1.4) and (3.31,-0.3) .. (-1,0) .. controls (3.31,0.3) and (6.95,1.4) .. (10.93,3.29);
        \draw [shift={(109,40)}, rotate = -90] [color={rgb, 255:red, 0; green, 0; blue, 0 }][line width=0.75](10.93,-3.29) .. controls (6.95,-1.4) and (3.31,-0.3) .. (-1,0) .. controls (3.31,0.3) and (6.95,1.4) .. (10.93,3.29);
        \draw [shift={(129,40)}, rotate = -90] [color={rgb, 255:red, 0; green, 0; blue, 0 }][line width=0.75](10.93,-3.29) .. controls (6.95,-1.4) and (3.31,-0.3) .. (-1,0) .. controls (3.31,0.3) and (6.95,1.4) .. (10.93,3.29);
        \draw [shift={(149,39)}, rotate = -90] [color={rgb, 255:red, 0; green, 0; blue, 0 }][line width=0.75](10.93,-3.29) .. controls (6.95,-1.4) and (3.31,-0.3) .. (-1,0) .. controls (3.31,0.3) and (6.95,1.4) .. (10.93,3.29);
        \draw [shift={(169,39)}, rotate = -90] [color={rgb, 255:red, 0; green, 0; blue, 0 }][line width=0.75](10.93,-3.29) .. controls (6.95,-1.4) and (3.31,-0.3) .. (-1,0) .. controls (3.31,0.3) and (6.95,1.4) .. (10.93,3.29);
        \draw [shift={(189,39)}, rotate = -90] [color={rgb, 255:red, 0; green, 0; blue, 0 }][line width=0.75](10.93,-3.29) .. controls (6.95,-1.4) and (3.31,-0.3) .. (-1,0) .. controls (3.31,0.3) and (6.95,1.4) .. (10.93,3.29);
        \draw [shift={(209,39)}, rotate = -90] [color={rgb, 255:red, 0; green, 0; blue, 0 }][line width=0.75](10.93,-3.29) .. controls (6.95,-1.4) and (3.31,-0.3) .. (-1,0) .. controls (3.31,0.3) and (6.95,1.4) .. (10.93,3.29);

        \draw[line width=0.75] (-30,190.5) -- (250,190.5);
        \draw[line width=0.75] (-30,200) -- (-10,190.5);
        \draw[line width=0.75] (-10,200) -- (10,190);
        \draw[line width=0.75] (10,200) -- (30,190);
        \draw[line width=0.75] (30,200) -- (50,190);
        \draw[line width=0.75] (50,200) -- (70,190);
        \draw[line width=0.75] (70,200) -- (90,190);
        \draw[line width=0.75] (90,200) -- (110,190);
        \draw[line width=0.75] (110,200) -- (130,190);
        \draw[line width=0.75] (130,200) -- (150,190);        
        \draw[line width=0.75] (150,200) -- (170,190);        
        \draw[line width=0.75] (170,200) -- (190,190);        
        \draw[line width=0.75] (190,200) -- (210,190);        
        \draw[line width=0.75] (210,200) -- (230,190.5);        
        \draw[line width=0.75] (230,200) -- (250,190.5);        
    \end{tikzpicture}
    \caption{Schematic drawing of the object domain $\Omega_1$ with displacement-traction boundaries $\Gamma_D$ and $\Gamma_T$. Here, $\gf_D$, is a fixed applied downward displacement, which does not have to be constant.} 
    \label{fig_schematic}
\end{figure}

The linear elasticity model \eqref{linear_elastic} is commonly used in elastography for uniformly isotropic material samples under the assumption of small deformations (infinitesimal strain). The Lam\'e parameters $\lambda,\mu$ can be directly related to the Young's modulus $E$ and the Poison ratio $\nu$, which have physical or diagnostic value, depending on the application \cite{ZaiMatMatSovHepMowKen21,Singh_Hepburn_Kennedy_Larin_2025}. Our motivation for considering this linear elastic model are quasi-static optical coherence elastography (OCE) experiments such as those conducted in \cite{Krainz_Sherina_Hubmer_Liu_Drexler_Scherzer_2022} and schematically depicted in Figure~\ref{fig_schematic}. There, a rectangular sample is fixed from below and uniformly compressed from above, while the sides of the sample are left to move freely. This corresponds to uniform displacement boundary conditions on the top ($\gfD = \text{const.}$) and bottom ($\gfD = 0$) of the sample, traction free boundary conditions ($\gfT = 0$) on the sides, and no body forces ($\ff = 0$). (Note that in the subsequent analysis, $\gfD$, $\gfT$, and $\ff$ do not have to be constant but can be functions.) A set of volumetric OCT scans is acquired both before and after compression, which could already be used as the images $\mI_1^\delta, \mI_2^\delta$ for the IIM. However, due to the symmetry of the samples considered in \cite{Krainz_Sherina_Hubmer_Liu_Drexler_Scherzer_2022}, the OCT scans are first reduced to 2D maximum intensity projections, which has some practical benefits.

In order to apply our theoretical results on the IIM to the linear elastic setting \eqref{linear_elastic}, we have to ensure that \eqref{linear_elastic} has a unique (weak) solution $\uf$ for each parameter $\af = (\lambda,\mu)$ in the admissible set
    \begin{equation}\label{def_Msm}
        \Msm := \Kl{ (\lambda,\mu) \in H^s(\Omega_1)^2 \, \vert \, \lambda \geq 0 \,, \mu > \mub > 0 } \subset X := H^s(\Omega_1)^2 \,,
    \end{equation}
where $0 < \mub \in \R$ and $N/2 < s \in \N$, with $N$ being the dimension of the domain $\Omega_1$. For this, we need

\begin{assumption}\label{ass_linear_elastic}
The domain $\Omega_1 \subset \R^N$, $N \in \N$, is nonempty, bounded, open, and connected, with a Lipschitz continuous boundary $\partial \Omega_1$, which has two subsets $\Gamma_D$ and $\Gamma_T$, satisfying $\partial \Omega_1 = \overline{\Gamma_D \cup \Gamma_T}$, $\Gamma_D \cap \Gamma_T = \emptyset$, and $\operatorname{meas}(\Gamma_D) > 0$. Furthermore, $\ff \in H^{-1}(\Omega_1)^{N}$, $\gfD \in H^{\frac{1}{2}}(\Gamma_D)^N$, $\gfT \in H^{-\frac{1}{2}}(\Gamma_T)^N$, and there is a function $\Phi \in H^1(\Omega_1)^N$ such that $\Phi \vert_{\Gamma_D} = \gfD$.
\end{assumption}

Following \cite{Hubmer_Sherina_Kindermann_Raik_2022}, we now homogenize \eqref{linear_elastic} and seek $\uft := \uf - \Phi$ satisfying the homogenized equations
    \begin{equation}\label{linear_elastic_homo}
    \begin{alignedat}{3}
        -\operatorname{div}{\sigma_{\lambda,\mu}(\uft)} & = \ff + \operatorname{div}{\sigma_{\lambda,\mu}(\Phi)} \,, \qquad &&\text{in} \quad \Omega_1 \,, 
        \\ 
        \uft & = 0 \,, \qquad &&\text{on} \quad \Gamma_D  \,,
        \\
        \sigma_{\lambda,\mu}(\uft) \nv &= \gfT - \sigma(\Phi) \nv \,,  \qquad &&\text{on} \quad \Gamma_T \,.
    \end{alignedat}
    \end{equation}
Introducing the bilinear and linear forms
    \begin{equation*}
    \begin{split}
        a_{\lambda,\mu}(\uft,\vf) &:= \int_{\Omega_1} \kl{\lambda \operatorname{div}(\uft) \operatorname{div}(\vf) + 2 \mu \mathcal{E}(\uft):\mathcal{E}(\vf)} \, dx \,,
        \\
        l(\vf) &:= \spr{\ff,\vf}_{H^{-1}(\Omega_1),H^1(\Omega_1)} + \spr{\gfT,\vf}_{H^{-\frac{1}{2}}(\Gamma_T),H^{\frac{1}{2}}(\Gamma_T)} \,,
    \end{split}
    \end{equation*}
it follows that the weak form of the homogenized BVP \eqref{linear_elastic_homo} is given by
    \begin{equation}\label{a=l}
        a_{\lambda,\mu}(\uft,\vf) = l(\vf) - a_{\lambda,\mu}(\Phi,\vf) \,,  
        \qquad
        \forall \, \vf \in V \,,
    \end{equation}
where
    \begin{equation*}
        V:= H_{0,\Gamma_D}^1(\Omega_1)^N \,,
        \qquad
        \text{with}
        \qquad
        H_{0,\Gamma_D}^1(\Omega_1)^N := \Kl{\vf \in H^{1}(\Omega_1) \, \vert \, \vf \vert_{\Gamma_D} = 0} \,.
    \end{equation*}
Concerning the solvability of the variational problem, \eqref{a=l}, we have the following result.

\begin{lemma}\label{lem_existence}
Let Assumption~\ref{ass_linear_elastic} hold and assume that $(\lambda,\mu) \in \Msm \subset X = H^s(\Omega_1)^2$ for some $\mub > 0$ and $s > N/2$. Then there exists a unique solution $\uft \in V$ of \eqref{a=l} and a constant $C > 0$ satisfying
    \begin{equation}\label{uft_bound}
        \norm{\uft}_{H^1(\Omega_1)} 
        \leq
        C \kl{ \norm{\ff}_{H^{-1}(\Omega_1)} + \norm{\gfT}_{H^{-\frac{1}{2}}(\Gamma_T)} 
        + \norm{(\lambda,\mu)}_X \norm{\Phi}_{H^1(\Omega_1)} } \,.
    \end{equation}
\end{lemma}
\begin{proof}
First, note that since $s > N/2$, there exists a constant $c^s_E > 0$ such that \cite{Adams_1970}
    \begin{equation}\label{embedding_inequality}
        \norm{\vf}_{L_\infty(\Omega_1)} \leq c^s_E \norm{\vf}_{H^s(\Omega_1)} \,,
        \qquad \forall \, \vf \in H^s(\Omega_1) \,.
    \end{equation}  
Hence, we have that $\Msm \subset H^s(\Omega_1)^2 \subset L_\infty(\Omega_1)^2$, and thus it follows from \cite[Theorem~2.1]{Hubmer_Sherina_Neubauer_Scherzer_2018} that for any $(\lambda,\mu) \in \Msm$ there exists a unique solution $\uft \in V$ of \eqref{a=l} and a constant $c_{\text{LM}} > 0$ satisfying
    \begin{equation*}
        \norm{\uft}_{H^1(\Omega_1)} 
        \leq
        c_{\text{LM}} \kl{ \norm{\ff}_{H^{-1}(\Omega_1)} + c_T \norm{\gfT}_{H^{-\frac{1}{2}}(\Gamma_T)} + \kl{N\norm{\lambda}_{L_\infty(\Omega_1)} + 2 \norm{\mu}_{L_\infty(\Omega_1)} } \norm{\Phi}_{H^1(\Omega_1)} }
    \end{equation*}
where $c_T$ denotes the constant of the trace inequality. Together with \eqref{embedding_inequality}, this yields the assertion.
\end{proof}

Hence, we have that $\uf = \uft + \Phi \in H^1(\Omega_1)^N$ is well-defined for each $(\lambda,\mu) \in \Msm$, and thus Part~4 of the minimal Assumption~\ref{ass_minimal} is satisfied. Concerning Assumption~\ref{ass_model_minimal}, we have the following result.

\begin{proposition}\label{prop_uft_continuous}
Let Assumption~\ref{ass_linear_elastic} hold and assume that $(\lambda,\mu) \in \Msm$ for some $\mub > 0$ and $s > N/2$. Then $\uft = \uft(\lambda,\mu) \in V$ defined by \eqref{a=l} is continuous wrt.\ $(\lambda,\mu)$, and there is a $C_{\lambda,\mu} > 0$ with
    \begin{equation}\label{uft_continuous}
        \norm{\uft(\lambda,\mu) - \uft(\bar{\lambda},\bar{\mu}) }_V \leq C_{\lambda,\mu} \norm{(\lambda,\mu) - (\bar{\lambda},\bar{\mu})}_X \,,
        \qquad 
        \forall \, (\bar{\lambda},\bar{\mu}) \in \Msm \,. 
    \end{equation}
Furthermore, $\uft$ considered a mapping $\uft: \Msm \to V \,, (\lambda,\mu) \mapsto \uft(\lambda,\mu)$ is weakly sequentially closed.
\end{proposition}
\begin{proof}
First, note that due to \cite[Equation~(3.11)]{Hubmer_Sherina_Neubauer_Scherzer_2018}, there exists a constant $C > 0$ such that
    \begin{equation*}
        \norm{\uft(\lambda,\mu) - \uft(\bar{\lambda},\bar{\mu})}_V \leq C \norm{(\lambda,\mu) - (\bar{\lambda},\bar{\mu})}_{L_\infty(\Omega_1)} \kl{ \norm{\uft(\lambda,\mu)}_{H^1(\Omega_1)} + \norm{\Phi}_{H^1(\Omega_1)}  } \,.
    \end{equation*}
Hence, together with \eqref{embedding_inequality} and $X = H^s(\Omega_1)^2$, this yields \eqref{uft_continuous} and thus the continuity of $\uft$. Furthermore, since the embedding \eqref{embedding_inequality} is compact \cite{Adams_1970}, we have that $\uft: \Msm \to V$ is weakly sequentially closed.
\end{proof}

The above results also transfer to the non-homogenized solution $\uf = \uft + \Phi \in H^1(\Omega_1)^N$. In particular, $\uf$ is continuous, weakly sequentially closed, and with $C_{\lambda,\mu}$ as in Proposition~\ref{prop_uft_continuous} there holds
    \begin{equation}\label{helper_uf_01}
        \norm{\uft(\lambda,\mu) - \uft(\bar{\lambda},\bar{\mu}) }_V \leq C_{\lambda,\mu} \norm{(\lambda,\mu) - (\bar{\lambda},\bar{\mu})}_X \,,
        \qquad 
        \forall \, (\bar{\lambda},\bar{\mu}) \in \Msm \,. 
    \end{equation}
Hence, together with $\Omega \subset \Omega_1$ and $L_2(\Omega_1) \subset V$, we find that \eqref{Lipschitz_u} in Part~2 of Assumption~\ref{ass_differentiability} is satisfied. Furthermore, if $\mI_2$ is continuous, then also Part~1 of Assumption~\ref{ass_model_minimal} holds. This now allows us to transfer the convergence result of Theorem~\ref{theorem_convergence_restricted} to our linear elasticity setting.

\begin{theorem}\label{theorem_convergence_elastic}
Let Assumption~\ref{ass_linear_elastic} hold, and let $0 < \mub \in \R$ and $N/2 < s \in \N$. In addition, assume that the domains $\Omega\,, \Omega_2 \subset \R^N$ are open and bounded with a Lipschitz continuous boundary and that
    \begin{equation}\label{cond_domain_elasticity}
        \forall \, (\lambda,\mu) \in \Msm \,, \forall \, x \in \Omega \, : \, \uf(\lambda,\mu) = \uft(\lambda,\mu) + \Phi \in \Omega_2 \,,
    \end{equation}
where $\uft$ is determined by \eqref{a=l}. Furthermore, let $\mI_1 \in L_2(\Omega_2)$, $\mI_2 \in C(\Omega_1)$, $(\lambda_0,\mu_0) \in \Msm$, assume that there exists an $(\lambda_0,\mu_0)$-MNS $(\lambda^*,\mu^*)$ of the equation $\mI_2 \circ G(\lambda,\mu) = \mI_1$, and that there is a noise level $\delta_{\mI_1} > 0$ such that $\Vert \mI_1 - \mI_1^\delta\Vert \leq \delta_{\mI_1}$ for $\mI_1^\delta \in L_2(\Omega_1)$. Moreover, let $\alpha = \alpha(\delta_{\mI_1})$ be chosen such that
    \begin{equation*}
        \alpha(\delta_{\mI_1}) \to 0 \,,
        \qquad
        \text{and}
        \qquad
        \frac{(\delta_{\mI_1})^2}{\alpha(\delta_{\mI_1})} \to 0 \,,
        \qquad
        \qquad
        \text{for}
        \qquad
        \delta_{\mI_1}\to 0 \,.
    \end{equation*}
Then every sequence $\Kl{(\lambda_k,\mu_k) := (\lambda_{\alpha(\delta_k)}^{\delta_k},\mu_{\alpha(\delta_k)}^{\delta_k})}_{k\in\N}$ where $\kl{\lambda_{\alpha(\delta_k)}^{\delta_k},\mu_{\alpha(\delta_k)}^{\delta_k}}$ is a minimizer of 
    \begin{equation*}
        \norm{\mI_2\circ G(\lambda,\mu) - \mI_1^\delta}_{L_2(\Omega)}^2 + \alpha \norm{(\lambda,\mu) - (\lambda_0,\mu_0)}_{H^s(\Omega_1)}^2 \,,
    \end{equation*}
and with $\delta_k \to 0$ as $k\to \infty$, has a convergent subsequence $(\lambda_{n_k},\mu_{n_k})$. Furthermore, the limit of every convergent subsequence is an $(\lambda_0,\mu_0)$-MNS. Moreover, if the $(\lambda_0,\mu_0)$-MNS $(\lambda^\dagger,\mu^\dagger)$ is unique, then
    \begin{equation*}
        \lim\limits_{\delta_k \to 0} \kl{\lambda_{\alpha(\delta_k)}^{\delta_k},\mu_{\alpha(\delta_k)}^{\delta_k}} = (\lambda^\dagger,\mu^\dagger)\,.
    \end{equation*}
\end{theorem}
\begin{proof}
Due to Lemma~\ref{lem_existence} and Proposition~\ref{prop_uft_continuous}, Theorem~\ref{theorem_convergence_restricted} is applicable and yields the assertion.    
\end{proof}

In order to establish convergence rates in our linear elasticity setting, we have to show the Fr\'echet differentiability of $\uft$ as well as the bounds \eqref{Lipschitz_u_infty} and \eqref{Lipschitz_Du}. For this, we first consider the following result.

\begin{proposition}\label{prop_uft_Frechet}
Let Assumption~\ref{ass_linear_elastic} hold, and let $\mub > 0$ and $s > N/2$. Then $\uft = \uft(\lambda,\mu) \in V$ defined by \eqref{a=l} and considered as a mapping from $\Msm \to V$ is continuously Fr\'echet differentiable for all $(\lambda,\mu) \in \Msm$. Furthermore, $\uft'(\lambda,\mu)(h_\lambda,h_\mu)$ is characterized as the unique solution $\zf \in V$ of 
    \begin{equation}\label{a=l_diff}
        a_{\lambda,\mu}(\zf,\vf) = -a_{h_\lambda,h_\mu}(\uft(\lambda,\mu) + \Phi,\vf) \,,
        \qquad \forall \, \vf \in V\,.
    \end{equation}
Moreover, for each $(\lambda,\mu) \in \Msm$ there exists a $C_{\lambda,\mu} > 0$ such that for all $(\bar{\lambda},\bar{\mu}) \in \Msm$ there holds
    \begin{equation*}
        \norm{\uft'(\lambda,\mu)(h_\lambda,h_\mu) - \uft'(\bar{\lambda},\bar{\mu})(h_\lambda,h_\mu) }_V 
        \leq C_{\lambda,\mu} \norm{(h_\lambda,h_\mu)}_X \norm{(\lambda,\mu) - (\bar{\lambda},\bar{\mu})}_X \,.
    \end{equation*}
\end{proposition}
\begin{proof}
The continuous Fr\'echet differentiability of $\uft$ and the characterization \eqref{a=l_diff} follows directly from \cite[Theorem~3.2]{Hubmer_Sherina_Neubauer_Scherzer_2018}. There, the variational form \eqref{a=l_diff} is expressed in the equivalent form
    \begin{equation*}
        \uft'(\lambda,\mu)(h_\lambda,h_\mu) = - A_{\lambda,\mu}^{-1} \kl{ A_{h_\lambda,h_\mu} \uft(\lambda,\mu) + \tilde{A}_{h_\lambda,h_\mu} \Phi  } \,,
    \end{equation*}
where $A_{\lambda,\mu}:= \tilde{A}_{\lambda,\mu}\vert_V$ and the bounded linear operator $\tilde{A}_{\lambda,\mu}$ is defined by
    \begin{equation*}
        \tilde{A}_{\lambda,\mu} \, : \, H^1(\Omega_1)^N \to V^* \,,
        \qquad 
        \vf \mapsto \kl{ \bar{\vf} \mapsto a_{\lambda,\mu}(\vf,\bar{\vf}) } \,.
    \end{equation*}
Using this equivalent representation, we find that for each $(\bar{\lambda},\bar{\mu}) \in \Msm $ there holds
    \begin{equation*}
        \uft'(\lambda,\mu)(h_\lambda,h_\mu) - \uft'(\bar{\lambda},\bar{\mu})(h_\lambda,h_\mu)  = - A_{\lambda,\mu}^{-1}  \kl{ A_{h_\lambda,h_\mu} \kl{ \uft(\lambda,\mu) -   \uft(\bar{\lambda},\bar{\mu}) } }  \,.
    \end{equation*}
Furthermore, due to \cite[Proposition~3.1]{Hubmer_Sherina_Neubauer_Scherzer_2018}, there exists constants $C_1,C_2 > 0$ such that
    \begin{equation*}
        \norm{A_{\lambda,\mu}^{-1}}_{V^*,V} \leq C_1 \,,
        \qquad \text{and} \qquad
        \norm{A_{\lambda,\mu}}_{V,V^*} \leq C_2 \norm{(\lambda,\mu)}_{L_\infty(\Omega_1)} \leq C_2 c_E^s \norm{(\lambda,\mu)}_X \,,
    \end{equation*}
and thus we obtain
    \begin{equation*}
    \begin{split}
        \norm{\uft'(\lambda,\mu)(h_\lambda,h_\mu) - \uft'(\bar{\lambda},\bar{\mu})(h_\lambda,h_\mu) }_V 
        &= \norm{ A_{\lambda,\mu}^{-1}  \kl{ A_{h_\lambda,h_\mu} \kl{ \uft(\lambda,\mu) -   \uft(\bar{\lambda},\bar{\mu}) } } }_V
        \\
        & \leq
        C_1 C_2 c_E^s \norm{(h_\lambda,h_\mu)}_X \norm{ \uft(\lambda,\mu) -   \uft(\bar{\lambda},\bar{\mu}) }_V\,,
    \end{split}
    \end{equation*}
which together with \eqref{uft_continuous} now yields the assertion.    
\end{proof}

Again, the above results transfer to the non-homogenized solution $\uf = \uft + \Phi \in H^1(\Omega_1)^N$. In particular, $\uf$ is continuously Fr\'echet differentiable with $\uf'(\lambda,\mu)(h_\lambda,h_\mu) = \uft'(\lambda,\mu)(h_\lambda,h_\mu)$, and
    \begin{equation}\label{helper_uf_02}
        \norm{\uf'(\lambda,\mu)(h_\lambda,h_\mu) - \uf'(\bar{\lambda},\bar{\mu})(h_\lambda,h_\mu) }_V 
        \leq C_{\lambda,\mu} \norm{(h_\lambda,h_\mu)}_X \norm{(\lambda,\mu) - (\bar{\lambda},\bar{\mu})}_X \,, 
    \end{equation}
for $C_{\lambda,\mu}$ as in Proposition~\ref{prop_uft_Frechet} and every $(\bar{\lambda},\bar{\mu}) \in \Msm$. Hence, together with $\Omega \subset \Omega_1$ and $L_2(\Omega_1) \subset V$, we find that condition \eqref{Lipschitz_Du} in Part~2 of Assumption~\ref{ass_differentiability} is satisfied for our linear elastic setting as well.

In order to establish the remaining condition \eqref{Lipschitz_u_infty}, note that \eqref{a=l_diff} is the weak form of the BVP
    \begin{equation}\label{BVP_diff}
    \begin{alignedat}{3}
        -\operatorname{div}\kl{\sigma_{\lambda,\mu}(\zf)} & = \operatorname{div}\kl{\sigma_{h_\lambda,h_\mu}(\uft(\lambda,\mu) + \Phi)} \,, \qquad &&\text{in} \quad \Omega_1 
        \\ 
        \zf & = 0 \,, \qquad &&\text{on} \quad \Gamma_D  \,,
        \\
        \sigma_{\lambda,\mu}(\zf) \nv &= -\sigma_{h_\lambda,h_\mu}(\uf(\lambda,\mu)+\Phi) \nv \,,  \qquad &&\text{on} \quad \Gamma_T \,.
    \end{alignedat}    
    \end{equation}
Hence, \eqref{Lipschitz_u_infty} amounts to an interior $L_\infty$-estimate of the weak solution of the above BVP. However, in order to apply standard regularity results for elliptic PDEs \cite{Gilbarg_Trudinger_1998,Necas_2011,McLean_2000}, we first have to ensure that the right hand side in \eqref{BVP_diff}, and thus $\uf(\lambda,\mu)$ is sufficiently regular. For this, we have the following result.

\begin{lemma}\label{lem_regularity}
Let Assumption~\ref{ass_linear_elastic} hold, let $s > N/2 + 1$, $\mub > 0$, $(\lambda,\mu) \in \Msm$, and assume that $\ff \in L_2(\Omega_1)^{N}$, $\gfD \in H^{\frac{3}{2}}(\Gamma_D)^N$, $\Phi \in H^2(\Omega_1)^N$. Then for every bounded, open, connected Lipschitz domain $\Omega_2 \subset \Omega_1$ with $\overline{\Omega_2} \Subset \Omega_1$ the unique solution $\uft(\lambda,\mu)$ of \eqref{a=l} satisfies $\uft\vert_{\Omega_2} \in H^2(\Omega_2)^N$ and $-\operatorname{div}(\sigma_{\lambda,\mu}(\uft)) = f$ pointwise a.e.\ in $\Omega_2$. Furthermore, there is a $c_R = c_R(\lambda,\mu,\Omega_1,\Omega_2) > 0$ such that
    \begin{equation}\label{uft_H2}
        \norm{\uft}_{H^2(\Omega_2)} \leq c_R \kl{\norm{\uft}_{H^1(\Omega_1)} + \norm{\ff}_{L_2(\Omega_1)}} \,. 
    \end{equation}
\end{lemma}
\begin{proof}
This follows directly from \cite[Theorem~4.16]{McLean_2000} observing the regularity of $(\lambda,\mu) \in \Msm$.
\end{proof}

With these preliminaries, we can now establish the following regularity result for $\uft'(\lambda,\mu)$.

\begin{proposition}\label{prop_uft_D_H2}
Let the assumptions of Lemma~\ref{lem_regularity} hold and assume that there exists a bounded, open, connected Lipschitz domain $\Omega' \subset \Omega_1$ with $\overline{\Omega} \Subset \Omega'$ and $\overline{\Omega'} \Subset \Omega_1$. Then the unique solution $\zf = \uft'(\lambda,\mu)(h_\lambda,h_\mu)$ of \eqref{a=l_diff} satisfies $\zf\vert_{\Omega} \in H^2(\Omega)^N$ and $-\operatorname{div}(\sigma_{\lambda,\mu}(\uft)) = \operatorname{div}\kl{\sigma_{h_\lambda,h_\mu}(\uft(\lambda,\mu) + \Phi)}$ pointwise a.e.\ in $\Omega$. Furthermore, there is a constant $c_R' = c_R'(\lambda,\mu,N,s,\Omega,\Omega',\Omega_1) > 0$ such that
    \begin{equation}\label{uft_D_H2}
    \begin{split}
        &\norm{\uft'(\lambda,\mu)(h_\lambda,h_\mu)}_{H^2(\Omega)} 
        \\
        & \qquad
        \leq c_R' \norm{(h_\lambda,h_\mu)}_{W^{1,\infty}(\Omega_1)} 
        \kl{
        \norm{\ff}_{L_2(\Omega_1)}  + \norm{\gfT}_{H^{-\frac{1}{2}}(\Gamma_T)} + \kl{1 + \norm{(\lambda,\mu)}_{L_\infty(\Omega_1)} }\norm{\Phi}_{H^2(\Omega_1)} }\,.
    \end{split}    
    \end{equation} 
\end{proposition}
\begin{proof}
First, note that due to Lemma~\ref{lem_regularity} there holds $\uft(\lambda,\mu)\vert_{\Omega'} \in H^2(\Omega')^N$. Furthermore, we have
    \begin{equation}\label{helper_z_02}
        \norm{\operatorname{div}\kl{\sigma_{h_\lambda,h_\mu}(\uft(\lambda,\mu) + \Phi)}}_{L_2(\Omega')}  
        \leq 
        c_G \max\Kl{\norm{h_\lambda}_{W^{1,\infty}(\Omega')},\norm{h_\mu}_{W^{1,\infty}(\Omega')}} \norm{\uft(\lambda,\mu) + \Phi}_{H^2(\Omega')} \,,
    \end{equation} 
for some $c_G = c_G(N)$, which is finite since $(h_\lambda,h_\mu)\in\Msm \subset W^{1,\infty}(\Omega_1)$ for $s > N/2 +1$. Next, we want to apply \cite[Theorem~4.16]{McLean_2000} to obtain the $H^2(\Omega)$ regularity of $\zf := \uft'(\lambda,\mu)(h_\lambda,h_\mu)$. For this, it is formally required that $\ff_1 := \operatorname{div}\kl{\sigma_{h_\lambda,h_\mu}(\uft(\lambda,\mu) + \Phi)} \in L_2(\Omega_1)$. However, a close inspection of the proof shows that $\ff_1 \in L_2(\Omega')$ is sufficient (cf.~also \cite{Necas_2011,Gilbarg_Trudinger_1998,Evans_1998}), and we obtain that $\zf\vert_{\Omega} \in H^2(\Omega)^N$ with
    \begin{equation}\label{helper_z_01}
        \norm{\zf}_{H^2(\Omega)} = \norm{\uft'(\lambda,\mu)(h_\lambda,h_\mu)}_{H^2(\Omega)} \leq \tilde{c}_R \kl{\norm{\uft'(\lambda,\mu)(h_\lambda,h_\mu)}_{H^1(\Omega_1)} + \norm{\ff_1}_{L_2(\Omega')}} \,. 
    \end{equation}
for some constant $\tilde{c}_R = \tilde{c}_R(\lambda,\mu,\Omega,\Omega_1,\Omega_2)$. Furthermore, as in the proof of Proposition~\ref{prop_uft_Frechet}, we obtain
    \begin{equation*}
    \begin{split}
        \norm{\uft'(\lambda,\mu)(h_\lambda,h_\mu)}_{H^1(\Omega_1)}
        &\leq 
        C \norm{(h_\lambda,h_\mu)}_{L_\infty(\Omega_1)} \norm{\uft(\lambda,\mu) + \Phi}_V 
        \\
        &\leq
        C \norm{(h_\lambda,h_\mu)}_{L_\infty(\Omega_1)} \kl{\norm{\uft(\lambda,\mu)}_{H^1(\Omega_1)} + \norm{\Phi}_{H^1(\Omega_1)} }
        \,.
    \end{split}
    \end{equation*}
for some constant $C = C(s)$. Combining this with \eqref{helper_z_01} and \eqref{helper_z_02}, we thus find that
    \begin{equation*}
    \begin{split}
        &\norm{\uft'(\lambda,\mu)(h_\lambda,h_\mu)}_{H^2(\Omega)} \leq  
        \tilde{c}_R \kl{\norm{\uft'(\lambda,\mu)(h_\lambda,h_\mu)}_{H^1(\Omega_1)} + \norm{\ff_1}_{L_2(\Omega')}}
        \\
        &\qquad\leq
        \tilde{c}_R \Big( C \norm{(h_\lambda,h_\mu)}_{L_\infty(\Omega_1)} \kl{\norm{\uft(\lambda,\mu)}_{H^1(\Omega_1)} + \norm{\Phi}_{H^1(\Omega_1)} }   
        \\
        & \qquad\qquad\qquad +
        c_G \max\Kl{\norm{h_\lambda}_{W^{1,\infty}(\Omega')},\norm{h_\mu}_{W^{1,\infty}(\Omega')}} \norm{\uft(\lambda,\mu) + \Phi}_{H^2(\Omega')}      
        \Big) 
        \\
        & \qquad
        \leq \tilde{c}_R (C+2c_G) \norm{(h_\lambda,h_\mu)}_{W^{1,\infty}(\Omega_1)} \kl{\norm{\uft(\lambda,\mu)}_{H^1(\Omega_1)} + \norm{\uft(\lambda,\mu)}_{H^2(\Omega')} + 2 \norm{\Phi}_{H^2(\Omega_1)} } \,,
    \end{split}    
    \end{equation*}
which together with \eqref{uft_H2} yields
    \begin{equation*}
    \begin{split}
        &\norm{\uft'(\lambda,\mu)(h_\lambda,h_\mu)}_{H^2(\Omega)} 
        \\
        & \qquad
        \leq \tilde{c}_R (C+2c_G) \norm{(h_\lambda,h_\mu)}_{W^{1,\infty}(\Omega_1)} \kl{(1+c_R)\norm{\uft(\lambda,\mu)}_{H^1(\Omega_1)} + c_R \norm{\ff}_{L_2(\Omega_1)}  + 2 \norm{\Phi}_{H^2(\Omega_1)} } \,,
    \end{split}    
    \end{equation*}    
Finally, using \eqref{uft_bound} we obtain
    \begin{equation*}
    \begin{split}
        &\norm{\uft'(\lambda,\mu)(h_\lambda,h_\mu)}_{H^2(\Omega)} 
        \leq \tilde{c}_R (C+2c_G) \norm{(h_\lambda,h_\mu)}_{W^{1,\infty}(\Omega_1)} 
        \Big( 
        c_R \norm{\ff}_{L_2(\Omega_1)}  + 2 \norm{\Phi}_{H^2(\Omega_1)}
        \\
        & \quad
        + (1+c_R)C \kl{ \norm{\ff}_{H^{-1}(\Omega_1)} + \norm{\gfT}_{H^{-\frac{1}{2}}(\Gamma_T)} 
        + \norm{(\lambda,\mu)}_{L_\infty(\Omega_1)} \norm{\Phi}_{H^1(\Omega_1)} } \Big) \,,
    \end{split}    
    \end{equation*}   
which after rearranging the terms now yields the assertion.
\end{proof}

Using the above interior regularity result, we obtain the following $L_\infty$-estimate for $\uft'(\lambda,\mu)$.

\begin{corollary}\label{cor_Lipschitz}
Let $N \in \Kl{1,2,3}$, $s > N/2 + 1$, and let the assumptions of Proposition~\ref{prop_uft_D_H2} hold. Then
    \begin{equation*}
    \begin{split}
        \norm{\uft'(\lambda,\mu)(h_\lambda,h_\mu)}_{L_\infty(\Omega)} 
        \leq C_\infty \norm{(h_\lambda,h_\mu)}_{X} 
        \kl{
        \norm{\ff}_{L_2(\Omega_1)}  + \norm{\gfT}_{H^{-\frac{1}{2}}(\Gamma_T)} + \kl{1 + \norm{(\lambda,\mu)}_{X} }\norm{\Phi}_{H^2(\Omega_1)} }\,.
    \end{split}    
    \end{equation*}
for some constant $C_\infty = C_\infty(\lambda,\mu,N,s,\Omega,\Omega',\Omega_1) > 0$.
\end{corollary}
\begin{proof}
Due to the Sobolev embedding inequality \eqref{embedding_inequality} it follows that for $N \in \Kl{1,2,3}$ there holds
    \begin{equation*}
        \norm{\uft'(\lambda,\mu)(h_\lambda,h_\mu)}_{L_\infty(\Omega)} \leq c_E^s \norm{\uft'(\lambda,\mu)(h_\lambda,h_\mu)}_{H^2(\Omega)} \,.
    \end{equation*}
which together with \eqref{uft_D_H2} and $X = H^s(\Omega_1) \subset W^{1,\infty}(\Omega_1)$ for $s > N/2 +1$ yields the assertion.
\end{proof}

Since $\uf(\lambda,\mu) = \uft(\lambda,\mu) + \Phi$, and thus $\uf'(\lambda,\mu) = \uft'(\lambda,\mu)$, the above estimate also holds for $\uf'(\lambda,\mu)$, which establishes \eqref{Lipschitz_u_infty}. Hence, all model assumptions on $\uf(\lambda,\mu)$ are satisfied, as we summarize below.

\begin{proposition}\label{prop_elast_u}
Let $N \in \Kl{1,2,3}$, $s > N/2 + 1$, and let the assumptions of Proposition~\ref{prop_uft_D_H2} hold. Furthermore, let $\uf(\lambda,\mu) = \uft(\lambda,\mu) + \Phi$, with $\uft(\lambda,\mu)$ defined by \eqref{a=l}, and for $(\lambda_0,\mu_0) \in \Msm$ let $(\lambda^*,\mu^*)$ be an $(\lambda_0,\mu_0)$-MNS of the equation $\mI_2 \circ G(\lambda,\mu) = \mI_1$. Then $\uf$ satisfies \eqref{Lipschitz_u_infty}, \eqref{Lipschitz_u}, \eqref{Lipschitz_Du} with   
    \begin{equation*}
    \begin{split}
        C_1^* &=  C_\infty 
        \kl{
        \norm{\ff}_{L_2(\Omega_1)}  + \norm{\gfT}_{H^{-\frac{1}{2}}(\Gamma_T)} + \kl{1 + \norm{(\lambda,\mu)}_{X} }\norm{\Phi}_{H^2(\Omega_1)} }  \,,
    \end{split}
    \end{equation*}
and $C_2^* = C_3^* = C_{\lambda^*,\mu^*}$, where $C_{\lambda^*,\mu*}$ is the constant $C_{\lambda,\mu}$ from  Proposition~\ref{prop_uft_continuous} for $(\lambda,\mu) = (\lambda^*,\mu^*)$.
\end{proposition}
\begin{proof}
This directly follows from the above results and comments. In particular, Proposition~\ref{prop_uft_continuous} and \eqref{helper_uf_01} establish \eqref{Lipschitz_u}, Proposition~\ref{prop_uft_Frechet} and \eqref{helper_uf_02} show \eqref{Lipschitz_Du}, and Corollary~\ref{cor_Lipschitz} yields \eqref{Lipschitz_u_infty}. 
\end{proof}

With this, we obtain the following convergence rate result for the IIM in our linear elastic setting.

\begin{theorem}\label{thm_rates_elastic}
Let $N \in \Kl{1,2,3}$, $s > N/2 +1$, $0 < \mub \in \R$, and let Assumption~\ref{ass_linear_elastic} hold. In addition, assume $\ff \in L_2(\Omega_1)^N$, $g_D \in H^{\frac{3}{2}}(\Gamma_D)^N$, $\Phi \in H^2(\Omega_1)$, and that $\Omega\,, \Omega_2 \subset \R^N$ are open and bounded domains with a Lipschitz continuous boundary. Also, assume that there exists a bounded, open, connected Lipschitz domain $\Omega' \subset \Omega_1$ with $\overline{\Omega} \Subset \Omega'$ and $\overline{\Omega'} \Subset \Omega_1$. Furthermore, let $\uf(\lambda,\mu) := \uft(\lambda,\mu) + \Phi$, where $\uft(\lambda,\mu)$ is defined by \eqref{a=l} for $(\lambda,\mu) \in \Msm$, with $\Msm$ as in \eqref{def_Msm}, and assume that \eqref{cond_domain_elasticity} holds. Moreover, let $\mI_2 \in W^{1,\infty}(\Omega_2)$ and $C_L := \norm{\mI_2}_{W^{1,\infty}(\Omega_2)}$, $\mI_1 \in L_2(\Omega_1)$, $(\lambda_0,\mu_0) \in \Msm$, assume that there exists an $(\lambda_0,\mu_0)$-MNS $(\lambda^*,\mu^*)$ of $\mI_2 \circ G(\lambda^*,\mu^*) = \mI_1$, and that there is a noise level $\delta_{\mI_1} > 0$ such that $\Vert \mI_1 - \mI_1^\delta\Vert \leq \delta_{\mI_1}$ for $\mI_1^\delta \in L_2(\Omega_1)$. Finally, assume that there exists a $\wf \in L_2(\Omega)$ with
    \begin{equation*}
        (\lambda^\dagger,\mu^\dagger) - (\lambda_0,\mu_0)  = \mathcal{F}'(\af^\dagger)^*\wf \,,
        \qquad \text{and} \qquad
        C \norm{\wf} < 1 \,,
    \end{equation*}
where $C := C_L(C_1^*C_2^* + C_3^*)$ with $C_k^*$ as in Proposition~\ref{prop_elast_u}. Then for $\alpha \sim \delta_{\mI_1}$, a minimizer $(\lambda_\alpha^\delta,\mu_\alpha^\delta)$ of 
    \begin{equation*}
        \norm{\mI_2\circ G(\lambda,\mu) - \mI_1^\delta}_{L_2(\Omega)}^2 + \alpha \norm{(\lambda,\mu) - (\lambda_0,\mu_0)}_{H^s(\Omega_1)}^2 \,,
    \end{equation*}
satisfies
    \begin{equation*}
        \norm{(\lambda_\alpha^\delta,\mu_\alpha^\delta) - (\lambda^\dagger,\mu^\dagger) }_{H^s(\Omega_1)} = \mathcal{O}\kl{\sqrt{\delta_{\mI_1}}} \,,
        \qquad \text{and} \qquad 
        \norm{\mI_2\circ G(\lambda_\alpha^\delta,\mu_\alpha^\delta) - \mI_1^{\delta}}_{L_2(\Omega)} = \mathcal{O}\kl{\delta_{\mI_1}} \,.
    \end{equation*}
\end{theorem}
\begin{proof}
Due to the above results, in particular Proposition~\ref{prop_uft_continuous} and Proposition~\ref{prop_elast_u}, all assumptions of Theorem~\ref{theorem_rates} are satisfied, which now yields the assertion. 
\end{proof}

\begin{remark}
The above result essentially states that in our linear elastic setting, the IIM is an order optimal regularization method when searching for sufficiently regular Lam\'e parameters $\lambda,\mu$, as long as all the involved quantities are sufficiently ``reasonable''. Note that the restriction $s > N/2 + 1$ instead of $s > N/2$, the increased regularity of $\ff$, $\gfD$, and $\Phi$ in Theorem~\ref{thm_rates_elastic} compared to Assumption~\ref{ass_linear_elastic}, and the existence of the intermediate domain $\Omega'$, are only required for establishing the $L_\infty$-estimate \eqref{Lipschitz_u_infty} via the interior regularity of $\uft'(\lambda,\mu)$. However, depending on the concrete setting of the BVP \eqref{linear_elastic}, the estimate \eqref{Lipschitz_u_infty} and thus Theorem~\ref{thm_rates_elastic} may be obtained under weaker assumptions as well.

For example, a close inspection of the proofs of the interior regularity of $\uf(\lambda,\mu)$ and $\uf'(\lambda,\mu)$ shows that $\Phi$ only has to be locally $H^2$ regular, and thus the increased regularity on $\gfD$ can in fact be dropped. The same is true for the Lam\'e parameters $(\lambda,\mu)$, which technically only have to be in $H^s(\Omega'')$ with $s > N/2 + 1$ on a domain $\Omega'' \subset \Omega_1$ satisfying $\overline{\Omega'} \Subset \Omega''$ and $\overline{\Omega''} \Subset \Omega_1$, while outside $\Omega''$ it is sufficient that $s > N/2$ as before. While this could in fact be accommodated for by redefining the admissible set $\Msm$, we have refrained from doing so in this paper to avoid an overly technical exposition.

Furthermore, in certain settings of the BVP \eqref{linear_elastic}, it is possible to obtain full $H^2(\Omega_1)$ regularity for $\uf(\lambda,\mu)$ and thus also for $\uf'(\lambda,\mu)$. In this case, the existence of the intermediary domain $\Omega'$ may be dropped entirely in the above analysis. One example is the case of pure Dirichlet boundary conditions, i.e., $\Gamma_D = \partial \Omega$ and $\Gamma_T = \emptyset$, for which $H^2(\Omega_1)$ regularity follows if $g_D \in H^{\frac{3}{2}}(\Omega_1)$ and $\partial \Omega_1 \in C^{1,1}$; cf.\ \cite[Theorem~4.18]{McLean_2000}. However, for mixed boundary conditions, full $H^2(\Omega_1)$ regularity can typically not be obtained except in very specific circumstances. Nevertheless, if $\Omega$ ``touches'' $\partial \Omega_1$ only on the Dirichlet boundary $\Gamma_D$ sufficiently far away from $\Gamma_T$, then due to \cite[Theorem~4.18]{McLean_2000} one still has $H^2(\Omega)$ regularity assuming $g_D \in H^{\frac{3}{2}}(\Gamma_D)$ and $\partial \Gamma_D \in C^{1,1}$, which is sufficient for establishing \eqref{Lipschitz_u_infty}.
\end{remark}

Finally, we consider the IIM for the full noise case in our linear elastic setting. In particular, we assume that $\ff$, $\gfD$, and $\gfT$ are contaminated by noise, and thus $\uf^\delta$ satisfies the noisy BVP 
    \begin{equation}\label{linear_elastic_noisy}
    \begin{alignedat}{3}
        -\operatorname{div}{\sigma_{\lambda,\mu}(\uf^\delta)} & = \ff^\delta \,, \qquad &&\text{in} \quad \Omega_1 \,, 
        \\ 
        \uf^\delta & = \gfD^\delta \,, \qquad &&\text{on} \quad \Gamma_D  \,,
        \\
        \sigma_{\lambda,\mu}(\uf^\delta) \nv &= \gfT^\delta \,,  \qquad &&\text{on} \quad \Gamma_T \,.
    \end{alignedat} 
    \end{equation}
This situation appears in practice if the applied forces in an elastography experiment are only known (measured) up to a certain accuracy. Concerning the noise bound \eqref{cond_noise_model}, we have the following result.

\begin{proposition}
Let $\ff^\delta \in H^{-1}(\Omega_1)^{N}$, $\gfD^\delta \in H^{\frac{1}{2}}(\Gamma_D)^N$, $\gfT^\delta \in H^{-\frac{1}{2}}(\Gamma_T)^N$, and assume that there is a function $\Phi^\delta \in H^1(\Omega_1)^N$ such that $\Phi^\delta \vert_{\Gamma_D} = \gfD^\delta$. Furthermore, let $\uf$ and $\uf^\delta$ denote the (unique) weak solutions of the BVPs \eqref{linear_elastic} and \eqref{linear_elastic_noisy}, respectively. Then there exists a constant $C > 0$ such that
    \begin{equation}\label{uft_bound_nois}
        \norm{\uf-\uf^\delta}_{H^1(\Omega_1)} 
        \leq
        C \kl{ \norm{\ff-\ff^\delta}_{H^{-1}(\Omega_1)} + \norm{\gfT-\gfT^\delta}_{H^{-\frac{1}{2}}(\Gamma_T)} 
        + \norm{(\lambda,\mu)}_X \norm{\Phi-\Phi^\delta}_{H^1(\Omega_1)} } \,.
    \end{equation}
\end{proposition}
\begin{proof}
First, note that $\uf = \uft + \Phi$ and $\uf^\delta = \uft^\delta + \Phi^\delta$, where $\uft \in V$ is the unique solution of the variational problem \eqref{a=l} and $\uft^\delta \in V$ is the unique weak solution of the variational problem
    \begin{equation*}
        a_{\lambda,\mu}(\uft,\vf) = l^\delta(\vf) - a_{\lambda,\mu}(\Phi^\delta,\vf) \,,  
        \qquad
        \forall \, \vf \in V \,,
    \end{equation*}
where
    \begin{equation*}
        l^\delta(\vf) := \spr{\ff^\delta,\vf}_{H^{-1}(\Omega_1),H^1(\Omega_1)} + \spr{\gfT^\delta,\vf}_{H^{-\frac{1}{2}}(\Gamma_T),H^{\frac{1}{2}}(\Gamma_T)} \,.
    \end{equation*}
Hence, the difference $\zf := (\uft - \uft^\delta) \in V$ is the unique solution of 
    \begin{equation*}
        a_{\lambda,\mu}(\zf,\vf) = \kl{l(\zf) - l^\delta(\zf)} - a_{\lambda,\mu}(\Phi-\Phi^\delta,\vf) \,,  
        \qquad
        \forall \, \vf \in V \,,
    \end{equation*}
and thus it follows as in Lemma~\ref{lem_existence} that there exists a constant $C> 0$ such that
    \begin{equation*}
        \norm{\uft-\uft^\delta}_{H^1(\Omega_1)} 
        \leq
        C \kl{ \norm{\ff-\ff^\delta}_{H^{-1}(\Omega_1)} + \norm{\gfT-\gfT^\delta}_{H^{-\frac{1}{2}}(\Gamma_T)} 
        + \norm{(\lambda,\mu)}_X \norm{\Phi-\Phi^\delta}_{H^1(\Omega_1)} } \,,
    \end{equation*}
which together with $(\uft - \uft^\delta) = (\uf - \uf^\delta)$ now yields the assertion.
\end{proof}

Now in order to establish the noise bound \eqref{cond_noise_model}, assume that there exist $\delta_{G,1},\delta_{G,2},\delta_{G,3}$ such that
    \begin{equation*}
        \norm{\ff-\ff^\delta}_{H^{-1}(\Omega_1)} \leq \delta_{G,1} \,,
        \qquad
        \norm{\gfT-\gfT^\delta}_{H^{-\frac{1}{2}}(\Gamma_T)} \leq \delta_{G,2} \,,
        \qquad
        \norm{\Phi-\Phi^\delta}_{H^1(\Omega_1)} \leq \delta_{G,3} \,.
    \end{equation*}
Then together with \eqref{uft_bound_nois} we obtain
    \begin{equation}
        \norm{\uf-\uf^\delta}_{H^1(\Omega_1)} 
        \leq
        C \kl{ \delta_{G,1} + \delta_{G,2}
        + \norm{(\lambda,\mu)}_X \delta_{G,3} } \,.
    \end{equation}
Hence, for a fixed $(\lambda^*,\mu^*) \in \Msm$ we obtain
    \begin{equation*}
        \sup_{(\lambda,\mu) \in \Msm \cap B_\eps(\lambda^*,\mu^*)} \norm{\uf-\uf^\delta}_{H^1(\Omega_1)} 
        \leq
        C \kl{ \delta_{G,1} + \delta_{G,2}
        + \kl{\norm{(\lambda^*,\mu^*)}_X + \eps} \delta_{G,3} } =: \delta_G   \,, 
    \end{equation*}
which established the noise bound \eqref{cond_noise_model}. With this, we can now transfer the convergence and convergence rates results from Section~\ref{sect_convergence} to our linear elastic setting as well. In particular, assume that $\mI_2^\delta = \mI_2 + \delta_{\mI_2} \xi$ with $\norm{\xi}_{L_\infty(\Omega_2)} < \infty$ and that there is a $\bar{\delta} > 0$ such that for all $\delta \leq \bar{\delta}$ there holds
    \begin{equation*}
        \forall\, (\lambda,\mu) \in \Msm \cap B_\eps((\lambda^*,\mu^*)) \,, \forall \x \in \Omega \, : \, G^\delta(\lambda,\mu)(\x) \in \Omega_2 \,.
    \end{equation*}
Then under the assumptions of Theorem~\ref{theorem_convergence_elastic} and Theorem~\ref{thm_rates_elastic}, the convergence and convergence rate results of Theorem~\ref{theorem_convergence} and Theorem~\ref{theorem_rates} transfer to our considered linear elastic setting, and we obtain that the IIM is a convergent and order optimal regularization method also in the full noise case.

\section{Numerical Experiments}\label{sect_numerics}

In this section, we consider the numerical application of our IIM to the quasi-static linear elastography problem as analyzed in Section~\ref{sect_appl_lin} on several examples with experimentally motivated simulated data. The considered setting is only one of many possible instances of elastography problems to which our general IIM approach can be applied. As noted above, our particular choice was motivated by the OCE experiments conducted in \cite{Krainz_Sherina_Hubmer_Liu_Drexler_Scherzer_2022}, in which the considered sample has several inclusions whose shape and location can be identified in advance, and where the Lam\'e parameters are piecewise constant. While we have shown above that the IIM can also be applied to much more general settings (non-constant parameters, unknown inclusion locations, general material models), the numerical setting considered below should be sufficient to numerically demonstrate the theoretical findings of this paper.

\subsection{Problem setting, discretization, and implementation}
\label{sect_problem_setting}

For the numerical experiments presented below, we consider a rectangular material sample ($N=2$) which is fixed from below and compressed by a certain distance $c_D$ from above, while left free to expand on the sides. This setting is motivated by the quasi-static OCE experiments conducted in \cite{Krainz_Sherina_Hubmer_Liu_Drexler_Scherzer_2022}, which due to the symmetry and physical properties of the considered silicone samples reduce to just such a 2D elastography setting. Now, assuming a linear elastic material, the BVP \eqref{linear_elastic} for the displacement field $\uf$ becomes
    \begin{equation}\label{linear_elastic_numerics}
    \begin{alignedat}{3}
        -\operatorname{div}{\sigma_{\lambda,\mu}(\uf)} & = 0 \,, \qquad &&\text{in} \quad \Omega_1 := (0,l_{x_1}) \times (0,l_{x_2}) \,, 
        \\ 
        \uf & = 0 \,, \qquad &&\text{on} \quad \Gamma_{D,1} := [0,l_{x_1}] \times \Kl{0}  \,,
        \\ 
        \uf & = (0,-c_D) \,, \qquad &&\text{on} \quad \Gamma_{D,2} := [0,l_{x_1}] \times \Kl{l_{x_1}}  \,,
        \\
        \sigma_{\lambda,\mu}(\uf) \nv &= 0 \,,  \qquad &&\text{on} \quad \Gamma_T := \Kl{0,l_{x_1}} \times [0,l_{x_2}] \,,
    \end{alignedat} 
    \end{equation}
where $\Gamma_D = \Gamma_{D,1} \cup \Gamma_{D,2}$. Note that in this case, one can simply choose $\Phi(x_1,x_2) = (0,-(c_D x_2)/l_{x_2})$ for the homogenization function in Assumption~\ref{ass_linear_elastic}. Again motivated by \cite{Krainz_Sherina_Hubmer_Liu_Drexler_Scherzer_2022}, we choose $l_{x_1} = 6.8\mm$ and $l_{x_2} = 2.9\mm$ for the sample dimensions, and $c_D = 0.267\mm$ for the applied compression. Furthermore, we assume that the sample itself consists of several inclusions embedded within a homogeneous background material, each with different (constant) material parameters. Mathematically, this implies
    \begin{equation}
        \lambda(\x) = \sum_{k=1}^K \lambda_k \chi_{D_k}(\x) \,, 
        \qquad \text{and} \qquad
        \mu(\x) = \sum_{k=1}^K \mu_k \chi_{D_k}(\x)
        \,,
    \end{equation}
where $0 < \lambda_k,\mu_k \in \R$ and the domains $D_k \subseteq \Omega_1$ correspond to the inclusions and background of the sample. Here, we assume that the $D_k$ are given, which is e.g.\ the case in the OCE experiments presented in \cite{Krainz_Sherina_Hubmer_Liu_Drexler_Scherzer_2022}, where they are directly estimated from the measured OCT image of the undeformed sample. Hence, we only have to estimate the constant Lam\'e parameters $\lambda_k,\mu_k$ for $k \in \Kl{1,\dots,K}$.

\begin{figure}[ht!]
    \centering
    \includegraphics[width=0.39\textwidth, clip=true, trim={50pt 150pt 50pt 150pt}]{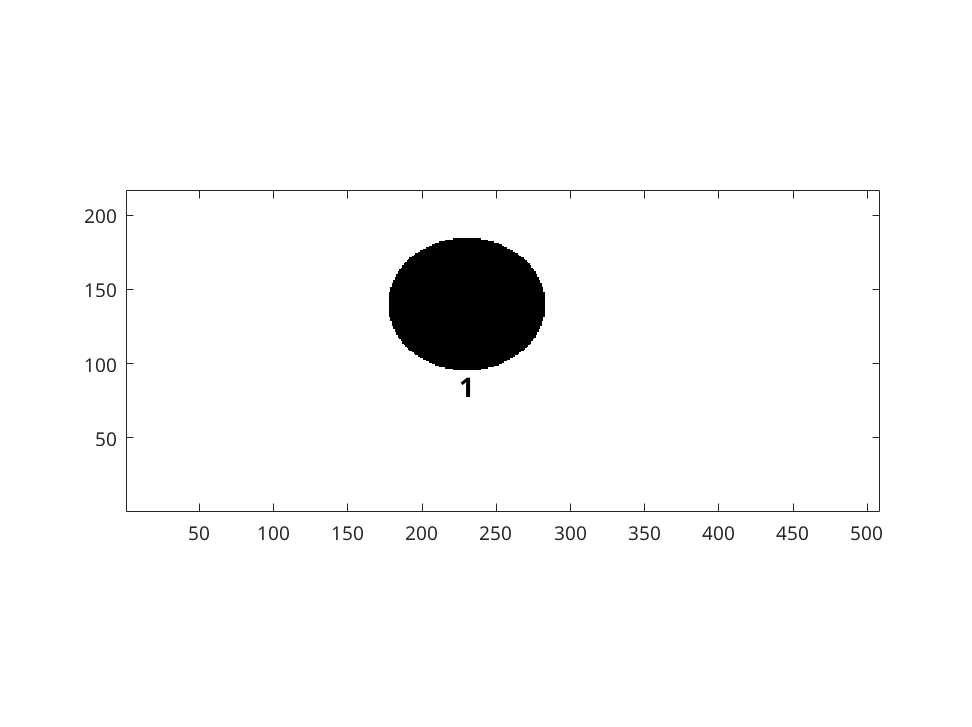}
    \includegraphics[width=0.4\textwidth, clip=true, trim={50pt 150pt 50pt 150pt}]{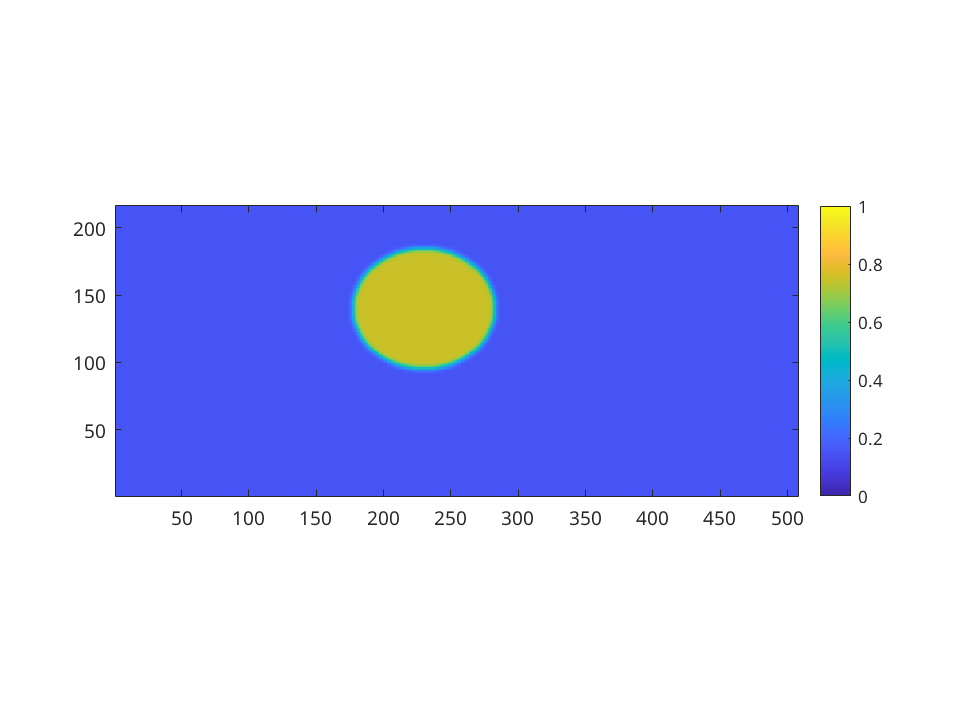}
    \\
    \includegraphics[width=0.39\textwidth, clip=true, trim={50pt 150pt 50pt 150pt}]{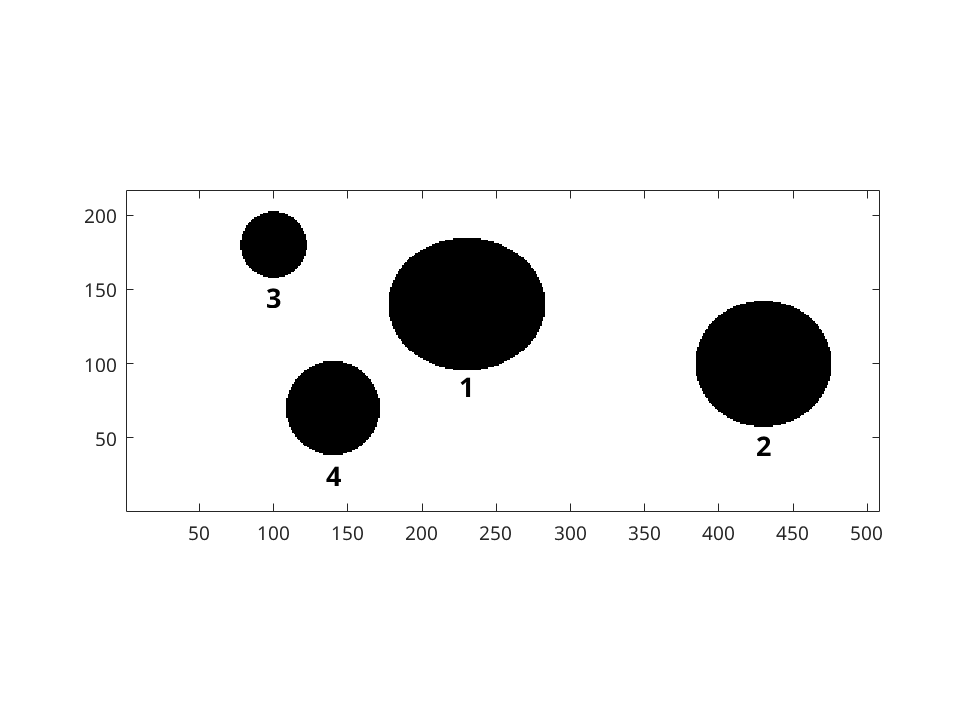}
    \includegraphics[width=0.4\textwidth, clip=true, trim={50pt 150pt 50pt 150pt}]{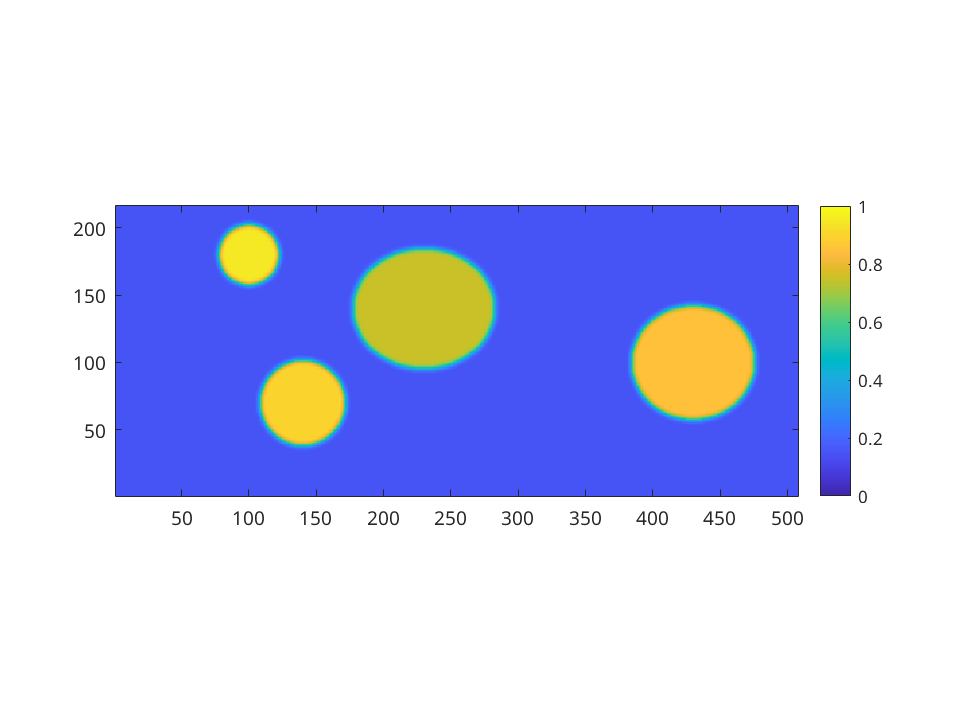}
    \caption{Schematic depiction of simulated material sample structure (left) and synthetic OCT images before the application of artificial speckle (right). Inclusion numbering corresponding to Table~\ref{tab_ground_truth}.}
    \label{fig_samples_setup}
\end{figure}

\begin{table}[ht!]
    \centering
    \resizebox{0.45\textwidth}{!}{
    \begin{tabular}{c|c|c|c|c|}
        & \multicolumn{4}{c|}{\textbf{Ground truth}} \\
        \cline{2-5}
        Entity & $E$, kPa & $\nu$ & $\lambda$, kPa & $\mu$, kPa\\
        \hline
        Background & 100 & 0.45 & 310.3448 & 34.4828\\
        Inclusion 1 & 200 & 0.45 & 620.6897 & 68.9655\\
        Inclusion 2 & 50 & 0.45 & 155.1724 & 17.2414\\
        Inclusion 3 & 75 & 0.45 & 232.7586 & 25.8621\\
        Inclusion 4 & 150 & 0.45 & 465.5172 & 51.7241
    \end{tabular}
    }
    \caption{Ground-truth values of the Lam\'e parameters $\lambda,\mu$ and the corresponding Young's modulus $E$ and Poisson ratio $\nu$ for the simulated material sample schematically depicted in Figure~\ref{fig_samples_setup}.}
    \label{tab_ground_truth}
\end{table}

For our tests, we consider two different sample configurations, one with a single inclusion, and one with four inclusions of varying sizes, which are schematically depicted in Figure~\ref{fig_samples_setup}~(left). The corresponding ground truth values of the Lam\'e parameters $\lambda,\mu$ are summarized in Table~\ref{tab_ground_truth}, together with the converted values of Young's modulus $E$ and the Poisson ratio $\nu$, which are connected via
    \begin{equation*}
        E = \frac{\mu(3\lambda + 2\mu)}{\lambda + \mu} \,,
        \qquad \text{and} \qquad
        \nu = \frac{\lambda}{2(\lambda + \mu)} \,.
    \end{equation*}
In order to create the synthetic image $\mI_1$, the domain $\Omega_1$ is first subdivided into a uniform $508 \times 216$ pixel grid, on which a constant background value of $0.15$ is chosen. Then, the different inclusions with brightness values between $0.75$ and $0.95$ are added to this background, followed by a slight filtering of the image mimicking the smoothness induced by OCT image acquisition; see Figure~\ref{fig_samples_setup}~(right). Finally, a random speckle pattern is added to the image, and the resulting image is rescaled into $[0,1]$, yielding the noise-free synthetic OCT image $\mI_1$; see Figure~\ref{fig_samples}~(left). 

\begin{figure}
    \centering
    \includegraphics[width=0.4\textwidth, clip=true, trim={50pt 180pt 50pt 180pt}]{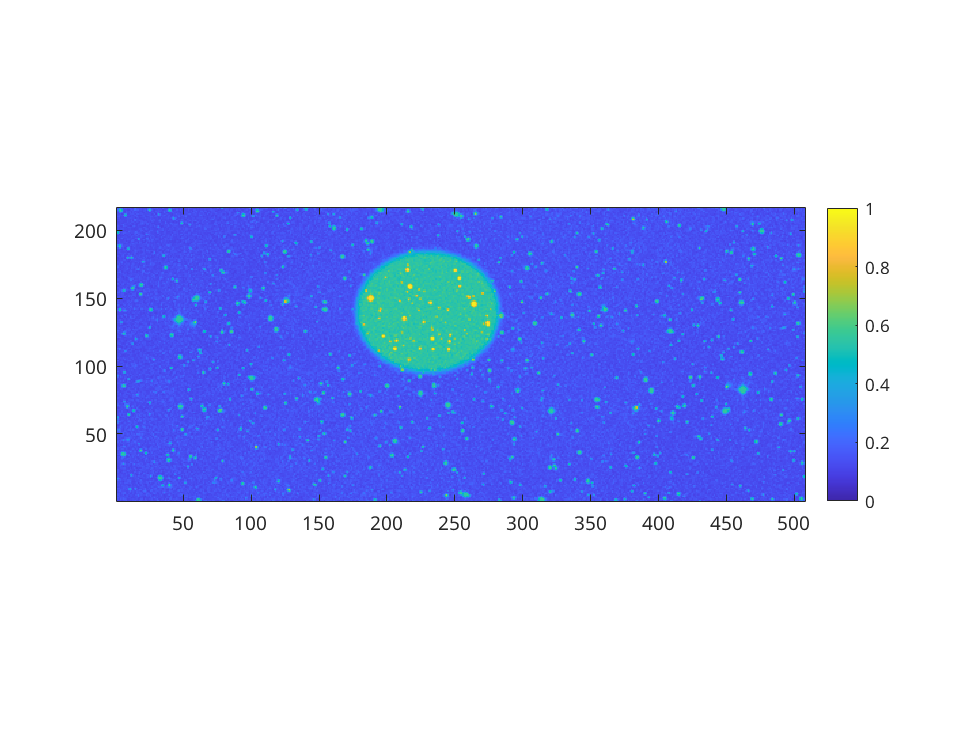}
    \includegraphics[width=0.4\textwidth, clip=true, trim={50pt 180pt 50pt 180pt}]{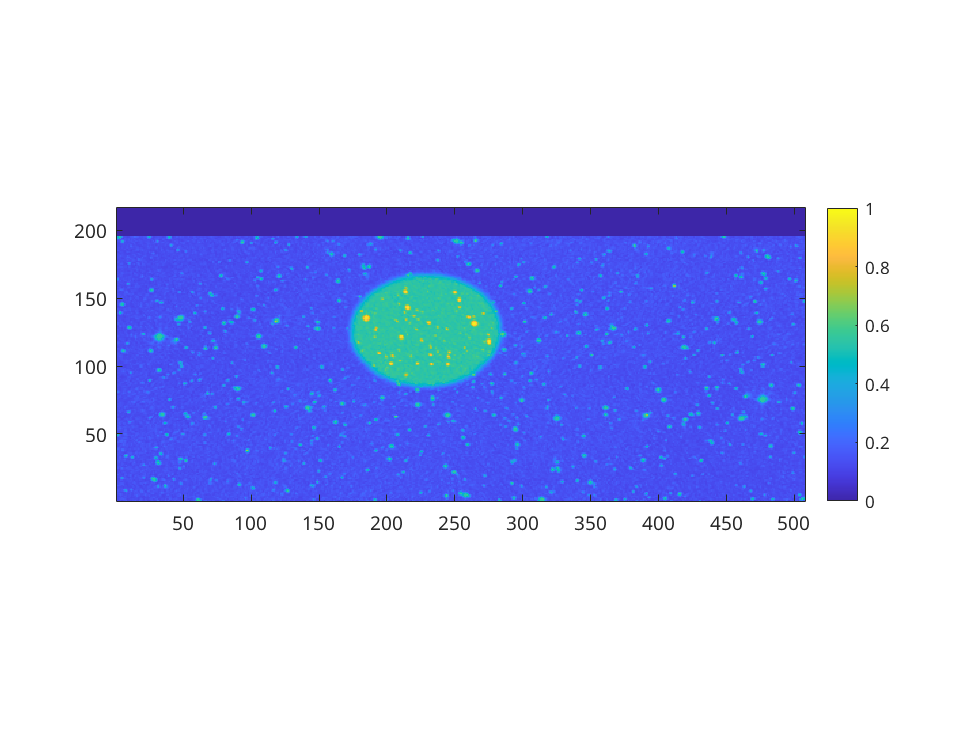}
    \includegraphics[width=0.4\textwidth, clip=true, trim={50pt 180pt 50pt 180pt}]{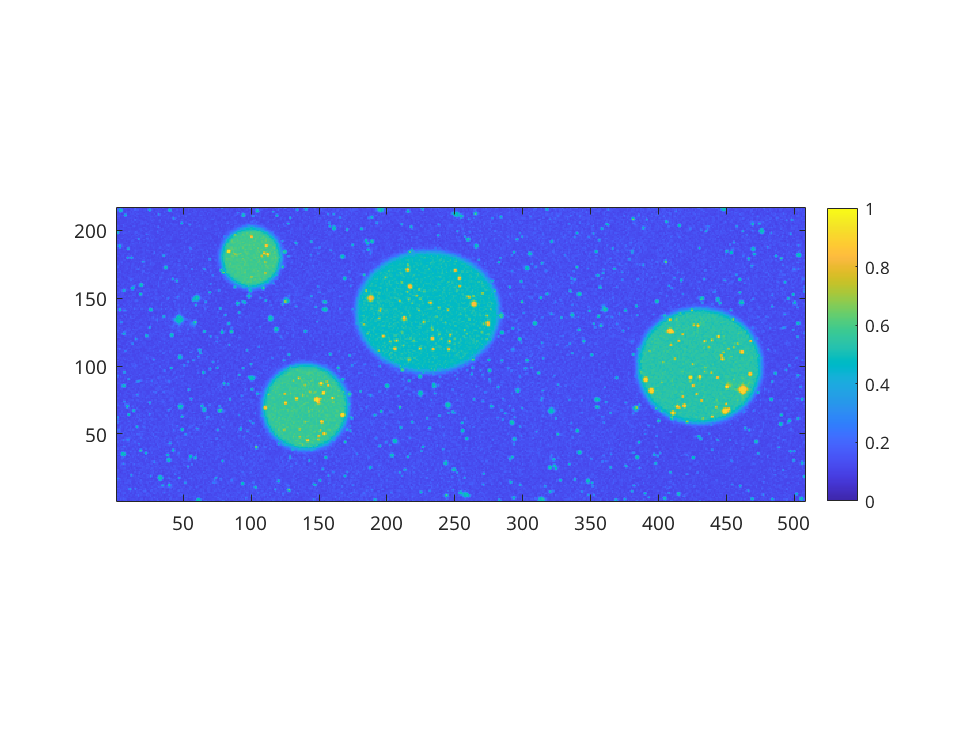}
    \includegraphics[width=0.4\textwidth, clip=true, trim={50pt 180pt 50pt 180pt}]{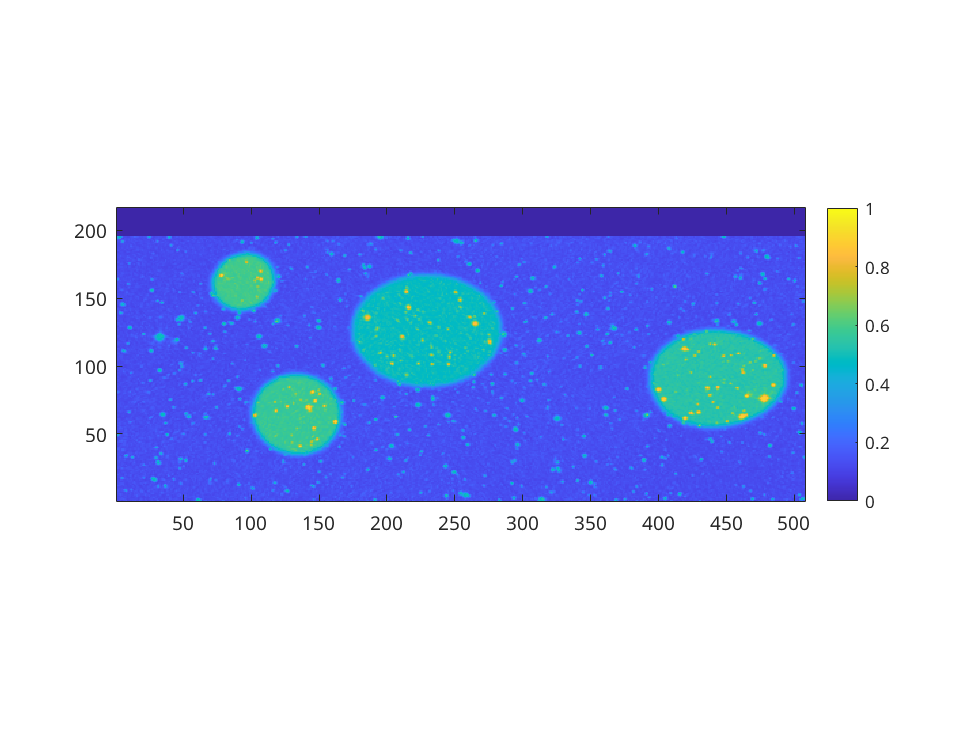}
    \caption{Synthetic OCT images (including speckle) before (left) and after compression (right), corresponding to $\mI_1$ and $\mI_2$.}
    \label{fig_samples}
\end{figure}

In order to numerically solve the BVP \eqref{linear_elastic_numerics}, and to compute the compressed image $\mI_2$, we use the finite element Python library FEniCSx \cite{FenicsX}. First, the domain $\Omega_1$ is sudivided into a uniform triangular mesh, such that each of the $508 \times 216$ pixels is subdivided into exactly $2$ right triangles. Then, using interpolation, the synthetic image $\mI_1$ is converted into a FEM function on this triangular mesh. Next, using the ground truth values of $\lambda,\mu$, the BVP \eqref{linear_elastic_numerics} is solved for $\uf$ using a continuous Galerkin method of degree $1$ (piecewise linear elements). Note that in our pixel discretization, the applied displacement $c_D = 0.267\mm$ corresponds to a displacement of exactly $20$ pixels. In order to create the compressed image $\mI_2$, the computed displacement field $\uf$ is used to deform the mesh underlying $\mI_1$, thereby effectively deforming the image itself. The resulting image $\mI_2$ is then interpolated back onto the non-deformed mesh, thereby cutting off the outwards-bulging sides of the sample. For our considered synthetic samples, the resulting deformed images are depicted in Figure~\ref{fig_samples}~(right). 

Note that real OCT scans do not (necessarily) have a consistent physical unit, but are mainly qualitative intensity images in an arbitrary value range, which are typically rescaled for further processing. This is reflected by our rescaling of the synthetic OCT scans, which are thus unitless intensity images.

In order to estimate the Lam\'e parameters $\lambda_k,\mu_k$ corresponding to the inclusions and background via the IIM from the images $\mI_1$ and $\mI_2$, we follow Section~\ref{subsect_piecewise}, and define the adapted IIM functional
    \begin{equation}\label{IIM_actual}
        \mathcal{T}_{\alpha,K}\kl{(\lambda_k, \mu_k)_{k=1}^K} :=  \norm{\mI_2 \circ G\kl{\sum_{k=1}^K (\lambda_k,\mu_k) \chi_{D_k}}- \mI_1}^2_{L_2(\Omega)} + \alpha \norm{\sum_{k=1}^K (\lambda_k,\mu_k) \chi_{D_k}}_{L_2(\Omega_1)}^2\,. 
    \end{equation}
Here, we use $\Omega = \Omega_1$, which is possible in our implementation based on grid-deformation. The minimization of $\mathcal{T}_{\alpha,K}$ is restricted to $10$--$1000$kPA for both $\lambda,\mu$, corresponding to the admissible set
    \begin{equation*}
        \mathcal{M} := \Kl{(\lambda_k, \mu_k)_{k=1}^K \in \R^K \, \big\vert \, 10 \leq \lambda_k,\mu_k \leq 1000 \,\,\, \forall \, k = \Kl{1,\dots,K}} \,.
    \end{equation*}
For solving the resulting finite-dimensional optimization problem, we use the SciPy implementation of the Nelder-Mead algorithm\cite{Bonnans_Gilbert_Lemarechal_Sagastizabal_2006} with an initial guess $(\lambda_0,\mu_0)$ corresponding to $(E_0,\nu_0) = (150,0.45)$, i.e., the values of inclusion~4. This derivative-free method was specifically chosen to avoid differentiation of the images $\mI_1, \mI_2$, and due to its ability to avoid local minima. For stopping the iteration, the built in tolerance was used, although as discussed in Section~\ref{subsect_numerics_noise}, other stopping rules may be beneficial. Finally, note that in all the numerical experiments, relative errors are measured in the $L_2(\Omega_1)$-norm.

Concerning computational costs, note that since in the setting of the considered numerical examples our proposed IIM approach estimates the material parameters via solving the minimization problem \eqref{IIM_actual}, the optimization algorithm chosen for this task directly determines the overall computational cost. Similarly, due to the wide array of different two-step approaches, and depending on the methods chosen to implement their two steps, it is difficult to provide a meaningful and comprehensive comparison of their overall computational costs in relation to the IIM. However, note that in most cases, the repeated computation of the displacement field $\uf(\af)$ for different material parameters $\af$ required in these approaches is typically the numerically most expensive part. For example, the idealized two-step NLI approach introduced in Section~\ref{subsect_NLI} requires three such function evaluation per iteration, while the Nelder-Mead algorithm used in our implementation of the IIM typically requires only two evaluations per iteration \cite{Bonnans_Gilbert_Lemarechal_Sagastizabal_2006}. However, different implementations may require different numbers of function evaluations, and the chosen methods may also require different numbers of iterations. Finally, while the first step in two-step methods does not require any additional evaluations of $\uf(\af)$, it may still involve a number of computationally expensive calculations depending on the chosen approach.

\subsection{Numerical results: noise-free images}

\begin{table}[ht!]
    \centering
    \resizebox{0.8\textwidth}{!}{
    \begin{tabular}{c|c|c|c|c|c|c|}
        & \textbf{Test1} & \textbf{Test2} & \multicolumn{2}{c|}{\textbf{Test3}} & \multicolumn{2}{c|}{\textbf{Test4}} \\
        \hline
        Entity & $\mu$, kPa & $\mu$, kPa & $\lambda$, kPa & $\mu$, kPa & $\lambda$, kPa & $\mu$, kPa\\
        \hline
        Background & 34.4340 & 33.9581 & 308.0995 & 34.2053 & 364.3680 & 40.2402 \\
        Inclusion 1 & 68.8290 & 67.6933 & 621.1630 & 68.3490 & 567.4982 & 77.3712 \\
        Inclusion 2 & & 15.7453 & & & 336.2634 & 20.8033 \\
        Inclusion 3 & & 33.4293 & & & 508.9541 & 54.3080 \\
        Inclusion 4 & & 68.0705 & & & 641.7964 & 60.4081 \\
        \hline
        Relative error in $\lambda$ and $\mu$ (\%) & 0.0002 & 0.6184 & 0.0041 & 0.0069 & 4.8830 & 2.8301 \\
        \hline
        Relative error in $(\lambda,\mu)$ (\%) &  & & \multicolumn{2}{c|}{0.0080} & \multicolumn{2}{c|}{5.6439}\\
        \hline
    \end{tabular}} 
    \caption{Numerical results with noise-free images $\mI_1$ and $\mI_2$. Test1 and Test3 consider the synthetic OCT sample with a single inclusion, cf.~Figure~\ref{fig_samples}~(top), while Test2 and Test4 consider the sample with four inclusions; cf.~Figure~\ref{fig_samples}~(bottom). As indicated, in Test1 and Test2 only $\mu$ is reconstructed (with $\lambda$ fixed to the ground truth), while in Test3 and Test4 both Lam\'e parameters $\lambda,\mu$ are unknown.}
    \label{tab:noise-free}
\end{table}

For our first series of tests, we consider the case of noise-free images $\mI_1$ and $\mI_2$. Here, no regularization (or only a very small amount) is necessary, and thus we choose $\alpha = 0$ in the IIM functional \eqref{IIM_actual}. In total, we consider four tests, two on the synthetic OCT sample with a single inclusion (Test1 and Test3) and two on the sample with four inclusions (Test2 and Test4); cf.~Figure~\ref{fig_samples}. For each of them, we consider both the reconstruction of $\mu$ with known and fixed values of $\lambda$ (Test1 and Test2), as well as the combined reconstruction of $\lambda,\mu$ (Test3 and Test4). The results are summarized in Table~\ref{tab:noise-free}.

First, note that in the case of a single inclusion with a known $\lambda$, the values of $\mu$ can be reconstructed accurately for both the background and the inclusion up to a negligible error (Test1). In the case of four inclusions, but still with a known $\lambda$, the accuracy decreases, but remains very good overall, with a total relative error of $0.62 \%$ (Test2). Note in particular that the values of $\mu$ for the background and the two large inclusions (1 and 2) are more accurately reconstructed than those for the two smaller inclusions (3 and 4). The largest individual reconstruction error is obtained for inclusion 4, likely because it is located furthest away from the compression boundary; cf.~Figure~\ref{fig_samples_setup}. Next, consider the case of a single inclusion with both $\lambda,\mu$ unknown (Test3). Even though from a theoretical point of view the recovery of both $\lambda,\mu$ simultaneously may exhibit issues due to non-uniqueness \cite{Barbone_Gokhale_2004}, we here obtain a fairly accurate reconstruction with a total relative error of only $0.008\%$. However, in the case of four inclusions, these non-uniqueness issues become more severe (Test4). While the reconstructed values of $\mu$ are still reasonably close to the ground truth (except for inclusion~3), the recovered values of $\lambda$ for inclusions~2 to 5 exhibit major inaccuracies. The large errors in both $\lambda,\mu$ for inclusion~3 are probably due to its small size compared to the much larger inclusion~1 in its immediate vicinity. Nevertheless, despite the non-uniqueness issues inherent in the simultaneous estimation of $\lambda,\mu$ from a single-compression quasi-static elastography experiment \cite{Barbone_Gokhale_2004}, the total error of $5.64\%$ demonstrates that the IIM is able to extract valuable material parameter information even in this under-determined~case.

\subsection{Numerical results: noisy images}\label{subsect_numerics_noise}

\begin{table}[h]
   \centering
   \resizebox{0.69\textwidth}{!}{\begin{tabular}{c|c|c|c|c|c|c|c|}
   \textbf{Noise }& \multicolumn{2}{c|}{\textbf{Background}} & \multicolumn{2}{c|}{\textbf{Inclusion}} & \multicolumn{3}{c|}{\textbf{Relative error}}\\
   \textbf{level}, \% & $\lambda$, kPa & $\mu$, kPa &$\lambda$, kPa & $\mu$, kPa & $\delta(\lambda)$, \% & $\delta(\mu)$, \% & $\delta(\lambda,\mu)$, \%\\
   \hline
   1 & 308.1845 & 34.0332 & 605.8392 & 68.1528 & 0.0167 & 0.0165 & 0.0237\\
   2 & 316.9569 & 34.8810 & 582.6335 & 67.1193 & 0.1199 & 0.0264 & 0.1228\\
   3 & 320.0652 & 35.4594 & 648.6192 & 66.1607 & 0.1178 & 0.0983 & 0.1533\\
   4 & 315.4122 & 34.8090 & 658.8490 & 63.6868 & 0.1039 & 0.1413 & 0.1755\\
   5 & 315.7042 & 34.8288 & 676.3859 & 62.9689 & 0.1996 & 0.1817 & 0.2699\\
   6 & 316.4830 & 35.1304 & 717.5514 & 58.6866 & 0.5560 & 0.5423 & 0.7767\\
   7 & 270.2898 & 30.2139 & 757.1503 & 47.6295 & 2.5020 & 4.0342 & 4.7471\\
   8 & 310.1156 & 34.4707 & 674.3543 & 52.7214 & 0.1668 & 1.3333 & 1.3437\\
   9 & 354.4916 & 38.8773 & 992.1345 & 57.8832 & 7.7957 & 1.7291 & 7.9851\\
  10 & 248.9686 & 27.7818 & 981.4385 & 39.9672 & 10.8934 & 9.2095 & 14.2647\\    
 \end{tabular}
   }
   \caption{Numerical results with noisy images $\mI_1^\delta$ and $\mI_2^\delta$ for various (relative) noise levels $\delta$. Recovered Lam\'e parameters and corresponding relative errors, each obtained using $100$ Nelder-Mead iterations.}
   \label{tab:noisy-100th-iteration}
\end{table}

\begin{figure}[ht!]
    \centering
    \includegraphics[width=0.45\textwidth]{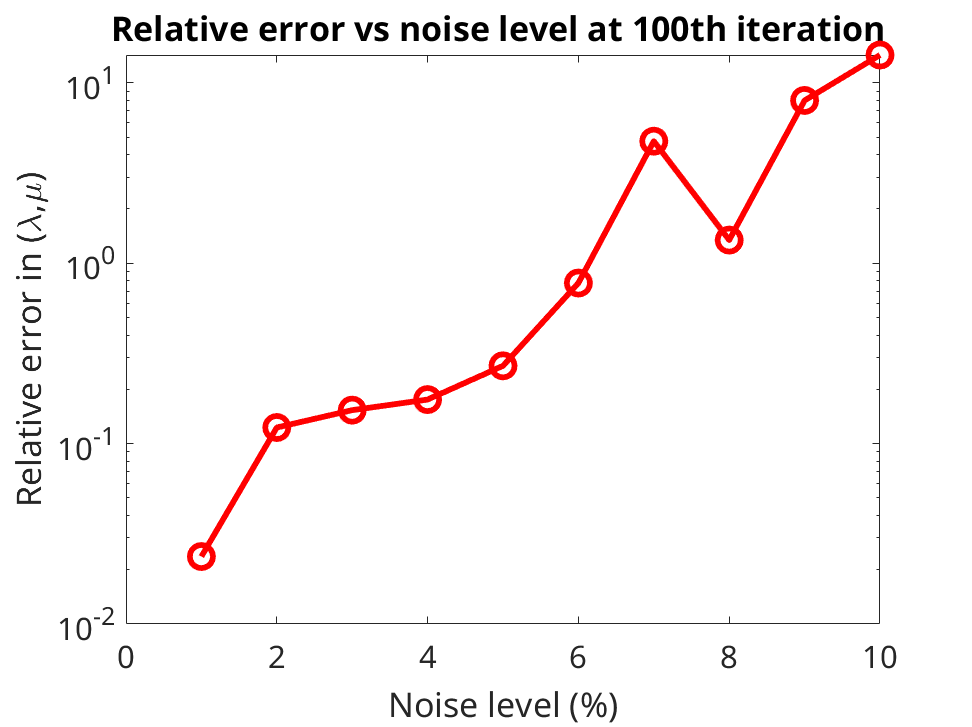} 
    \qquad
    \includegraphics[width=0.45\textwidth]{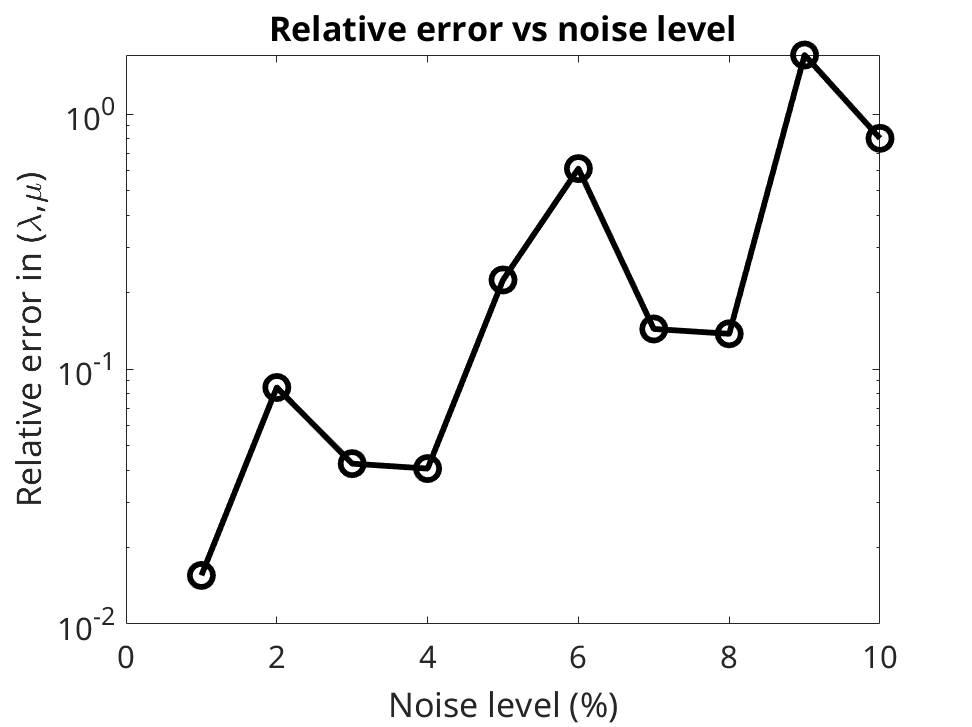}
    \caption{Numerical results with noisy images $\mI_1^\delta$ and $\mI_2^\delta$. Relative error vs noise level obtained after 100 Nelder-Mead iterations (left) and after the optimal number of Nelder-Mead iterations (right).}
    \label{fig:rates}
\end{figure}

For our second series of tests, we consider the case of noisy images $\mI_1^\delta$ and $\mI_2^\delta$ on the example of the synthetic OCT sample with a single inclusion; cf.~Figure~\ref{fig_samples_setup}~(top). To obtain these noisy images, uniformly distributed random noise with different (relative) noise levels $\delta$ is added to the noise-free images $\mI_1$ and $\mI_2$. Then, we use the IIM to reconstruct both Lam\'e parameters $\lambda,\mu$, where in the IIM functional \eqref{IIM_actual} we choose $\alpha = 0.1 \ast \delta$ as the regularization parameter, which is consistent with our developed theory. Note that the value $0.1$ was selected by considering the noise level $\delta = 1\%$ and testing different choices $\alpha = C \ast \delta$ for a wide range of possible constants $C$, with $C = 0.1$ found to be optimal. In practice, different selection criteria such as heuristic parameter choice rules have to be used instead \cite{Engl_Hanke_Neubauer_1996,Leonov_1991,Tikhonov_Glasko_1965,Hanke_Raus_1996,Reginska_1996,Hansen_OLeary_1993,Wahba_1990,Kindermann_2011,Kindermann_2013,Kindermann_Neubauer_2008,Raus_Haemarik_2018,Haemarik_Palm_Raus_2011,Hubmer_Sherina_Kindermann_Raik_2022}.

Table~\ref{tab:noisy-100th-iteration} summarizes the corresponding results, each obtained after $100$ iterations of the Nelder-Mead method. It can be seen that the IIM is able to accurately reconstruct both $\lambda,\mu$, with a relative error below $1\%$ even at a noise level of $6\%$. Furthermore, with the exception of the $7\%$ noise case, the relative errors are monotonically decreasing with decreasing noise level, in alignment with our convergence analysis; cf.~Figure~\ref{fig:rates} (left).

\begin{figure}[ht!]
    \centering    
    \boxed{\includegraphics[width=\linewidth,clip=true,trim={3.2cm 3.2cm 5.2cm 3.5cm}]{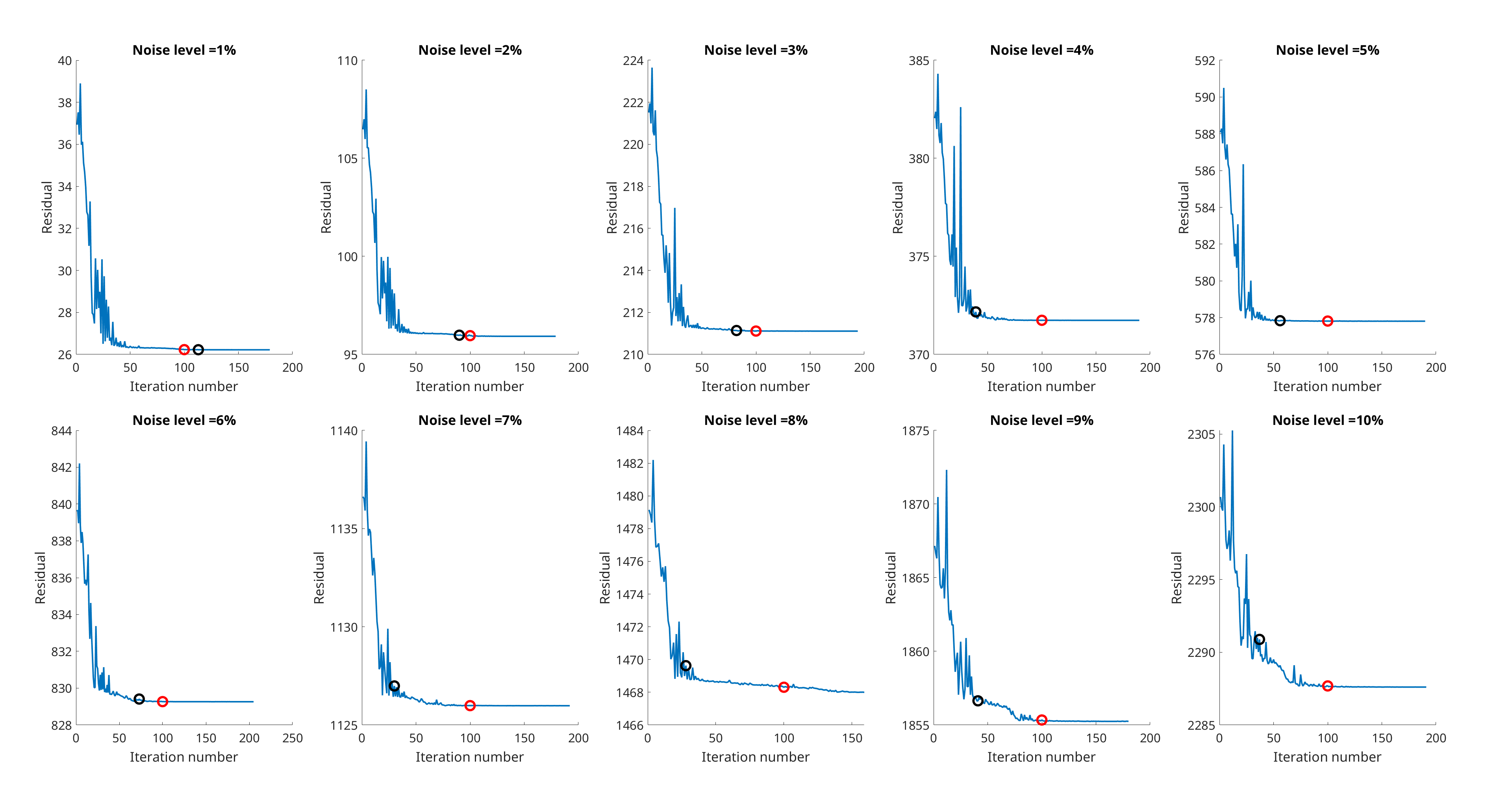}}
    \boxed{\includegraphics[width=\linewidth,clip=true,trim={3.2cm 3.2cm 5.2cm 3.5cm}]{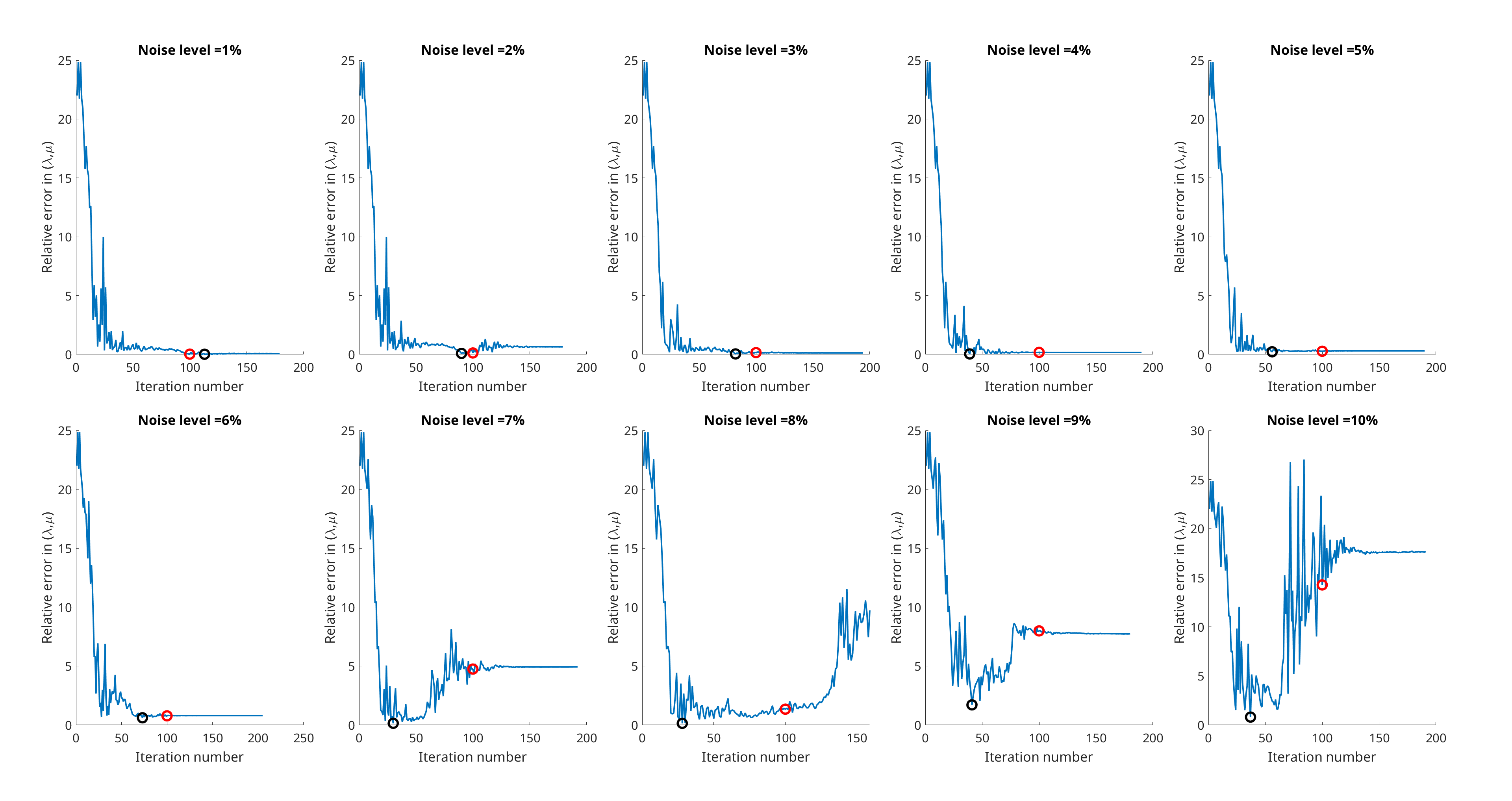}}  
    \caption{Numerical results with noisy images $\mI_1^\delta$ and $\mI_2^\delta$ for various (relative) noise levels $\delta$. Residual vs noise level (top) and relative error in $(\lambda, \mu)$ vs noise level (bottom). Iterates at which the minimal relative error is achieved are marked in black circles, while the $100$th iteration is marked in red circles.}
    \label{fig:noise-level-convergence}
\end{figure}

\begin{table}[ht!]
    \centering
    \resizebox{0.69\textwidth}{!}{\begin{tabular}{c|c|c|c|c|c|c|c|}
    \textbf{Noise} & \multicolumn{2}{c|}{\textbf{Background}} & \multicolumn{2}{c|}{\textbf{Inclusion}} & \multicolumn{3}{c|}{\textbf{Relative error}}\\
    \textbf{level, \%} & $\lambda$, kPa & $\mu$, kPa &$\lambda$, kPa & $\mu$, kPa & $\delta(\lambda)$, \% & $\delta(\mu)$, \% & $\delta(\lambda,\mu)$, \%\\
    \hline
    1 & 309.6878 & 34.2403 & 606.4352 & 67.9058 & 0.0124 & 0.0093 & 0.0155\\
    2 & 313.3472 & 34.9118 & 586.4783 & 66.7082 & 0.0766 & 0.0361 & 0.0847 \\
    3 & 315.8050 & 34.9571 & 636.7115 & 68.0436 & 0.0383 & 0.0185 & 0.0426 \\
    4 & 312.2792 & 34.7257 & 620.5090 & 71.7598 & 0.0030 &  0.0406 & 0.0407 \\
    5 & 318.0842 & 35.0892 & 669.3549 & 64.2301 & 0.1813 & 0.1311 & 0.2237 \\
    6 & 329.2834 & 36.6660 & 669.7078 & 63.0915 & 0.4034 & 0.4595 & 0.6115 \\
    7 & 306.5324 & 33.6843 & 664.3177 & 66.4250 & 0.1230 & 0.0743 & 0.1437 \\
    8 & 314.6889 & 35.1708 & 663.3236 & 66.1393 & 0.1194 & 0.0683 & 0.1375 \\
    9 & 357.5540 & 37.3752 & 608.8480 & 65.5647 & 1.6098 & 0.5667 & 1.7066 \\
    10 & 339.7543 & 34.6276 & 569.2395 & 67.3298 & 0.8036 & 0.0141 & 0.8038 
     \end{tabular}
    }
    \caption{Numerical results with noisy images $\mI_1^\delta$ and $\mI_2^\delta$ for various (relative) noise levels $\delta$. Recovered Lam\'e parameters and corresponding relative errors, each obtained using the optimal number of Nelder-Mead iterations.}
    \label{tab:noisy-best-iteration}
\end{table}

As noted above, for the minimization of the IIM functional \eqref{IIM_actual} the Nelder-Mead method is used, which is typically terminated when the standard deviation of the functional values falls below a certain tolerance. From the residual curves depicted in Figure~\ref{fig:noise-level-convergence} (top), one can see that for most noise levels, the residual is barely decreasing beyond the $100$th iteration, which is why it was chosen as the stopping index for all the tests presented in Table~\ref{tab:noisy-100th-iteration}. However, the relative error curves depicted in Figure~\ref{fig:noise-level-convergence} (bottom) show that for most noise levels, the optimal relative error is often already obtained after a (sometimes significantly) smaller number of iterations, where the residual still decreases. Table~\ref{tab:noisy-best-iteration} summarizes the recovered Lam\'e parameters corresponding to the ``optimal'' choice of stopping index, as well as the corresponding relative errors. As expected, the results are now significantly better than those presented in Table~\ref{tab:noisy-100th-iteration}, with the relative error never exceeding $2\%$ even for a noise level of $10\%$. However, while the relative error still decreases on average with decreasing noise level, it now no longer does so monotonically; cf.~Figure~\ref{fig:rates}~(right). Hence, while this ``optimal'' choice of stopping index is clearly not applicable in practice, it nevertheless indicates that the IIM can significantly benefit from a suitably terminated minimization procedure. Overall, our numerical tests demonstrate that both for noise-free and noisy images $\mI_1^\delta$ and $\mI_2^\delta$, and regardless of the chosen stopping index in the Nelder-Mead method, the IIM is able to stably and accurately reconstruct both Lam\'e parameters $\lambda,\mu$ in this case.

\subsection{Numerical results: effects of imprecise segmentation}
    \begin{figure}[ht!]
        \centering
        \includegraphics[width=0.25\linewidth,clip=true, trim={50pt 150pt 50pt 150pt}]{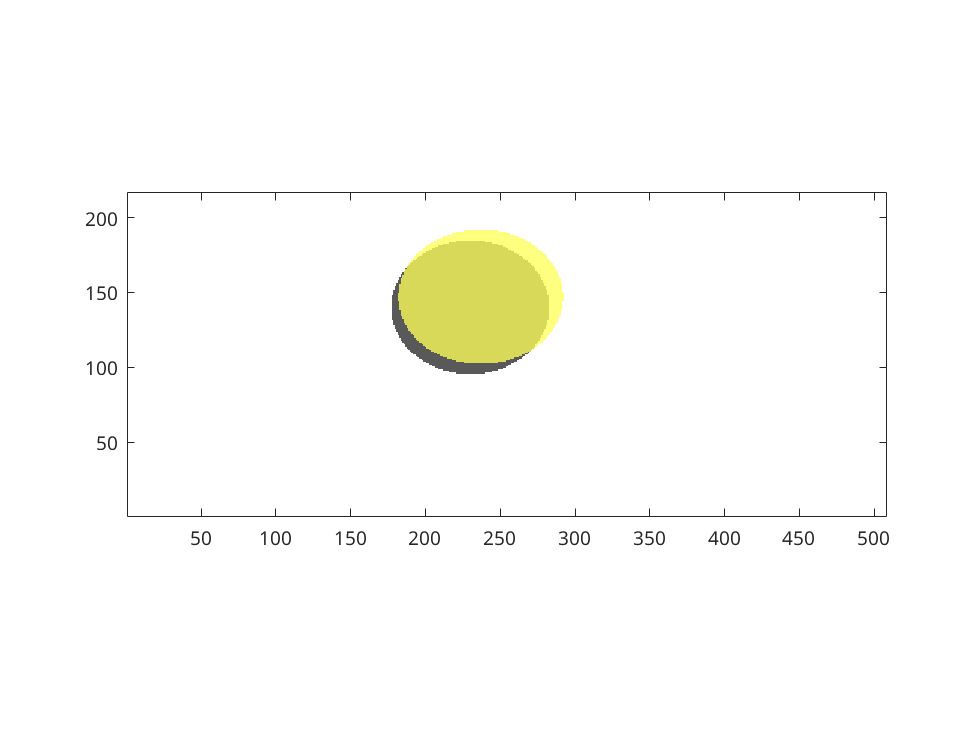}
         \includegraphics[width=0.25\linewidth, clip=true, trim={50pt 150pt 50pt 150pt}]{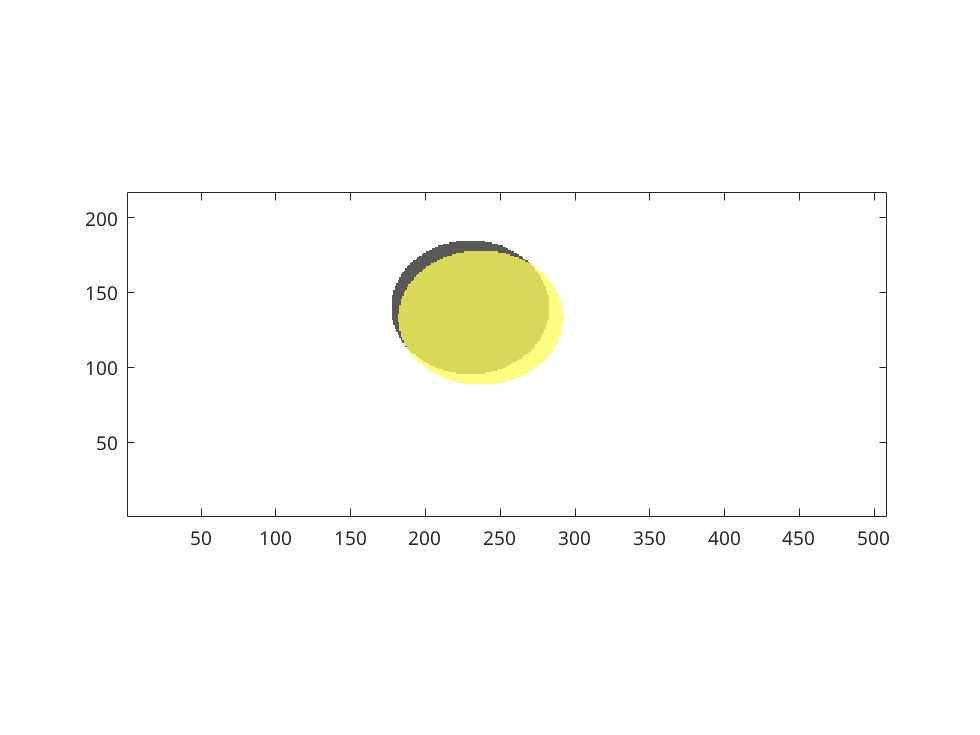}
        \caption{Sample with imprecisely segmented location and size of inclusion (yellow) superimposed with original inclusion (black). The imprecisely segmented inclusion is 6 pixels broader, and shifted by (7,7) pixels for Test5~(left) or by (7,-7) pixels for~Test6 (right), respectively.}
        \label{fig_imprecise}
    \end{figure}
    \begin{table}[ht!]
        \centering
        \resizebox{0.6\textwidth}{!}{
        \begin{tabular}{c|c|c|c|c|}
             & \multicolumn{2}{c|}{\textbf{Test5}} & \multicolumn{2}{c|}{\textbf{Test6}}  \\
             \hline
             Entity & $\lambda$, kPa & $\mu$, kPa & $\lambda$, kPa & $\mu$, kPa \\
             \hline
             Background & 322.9355 & 36.0826 & 279.4781 & 31.0844 \\
             Inclusion 1 & 607.0273 & 68.6364 & 627.4542 & 63.1942 \\
             \hline
             Relative error in $\lambda$ and $\mu$ (\%)  &  0.13 & 0.16 & 0.82 & 1.00 \\
             \hline
             Relative error in $(\lambda,\mu)$ (\%) &  \multicolumn{2}{c|}{0.21} & \multicolumn{2}{c|}{1.29}\\
             \hline
        \end{tabular}
        }
        \caption{Numerical results with noise-free images $\mI_1$ and $\mI_2$ of the synthetic OCT sample with a single inclusion, see Figure~\ref{fig_samples} (top), imprecisely  segmented as depicted in Figure~\ref{fig_imprecise}, (left) for Test5 and (right) for Test6.}
        \label{tab_imprecise}
    \end{table}

\noindent
In the numerical experiments conducted above, we have always assumed that the shape and location of the inclusions are known exactly. In physical experiments on samples or biopsies, this structural information is either obtained by directly segmenting the images $\mI_1$ and $\mI_2$, or via a complementary imaging modality. However, since real-world data are inevitably contaminated by noise, the resulting segmentation is imprecise as well, which can pose a challenge to the material parameter estimation in general, and our considered IIM realization in particular. However, the physical experiments conducted in \cite{Krainz_Sherina_Hubmer_Liu_Drexler_Scherzer_2022} show that the IIM is capable of handling real-world data, regardless of imprecisions in the segmentation. To further support these findings, we conducted two additional numerical experiments where the IIM is used with an imprecisely segmented inclusion, i.e., its location and shape are different from the true inclusion; see Test5 and Test6 depicted in Figure~\ref{fig_imprecise}. The reconstruction results are summarized in Table~\ref{tab_imprecise}, and show that although the resulting errors are higher than those in the corresponding Test3 in Table~\ref{tab:noise-free}, the obtained values of the Lam\'e parameters $\lambda,\mu$ are still fairly accurate. Note that when the imprecise location of the inclusion is closer to the compression boundary of the sample (Test5), the reconstruction is more accurate than when it is located further inside the sample (Test6).

\vspace{10pt}

\subsection{Numerical results: inclusions with sharp corners}

\begin{figure}[ht!]
        \centering
        \includegraphics[width=0.22\linewidth, clip=true, trim={50pt 150pt 50pt 150pt}]{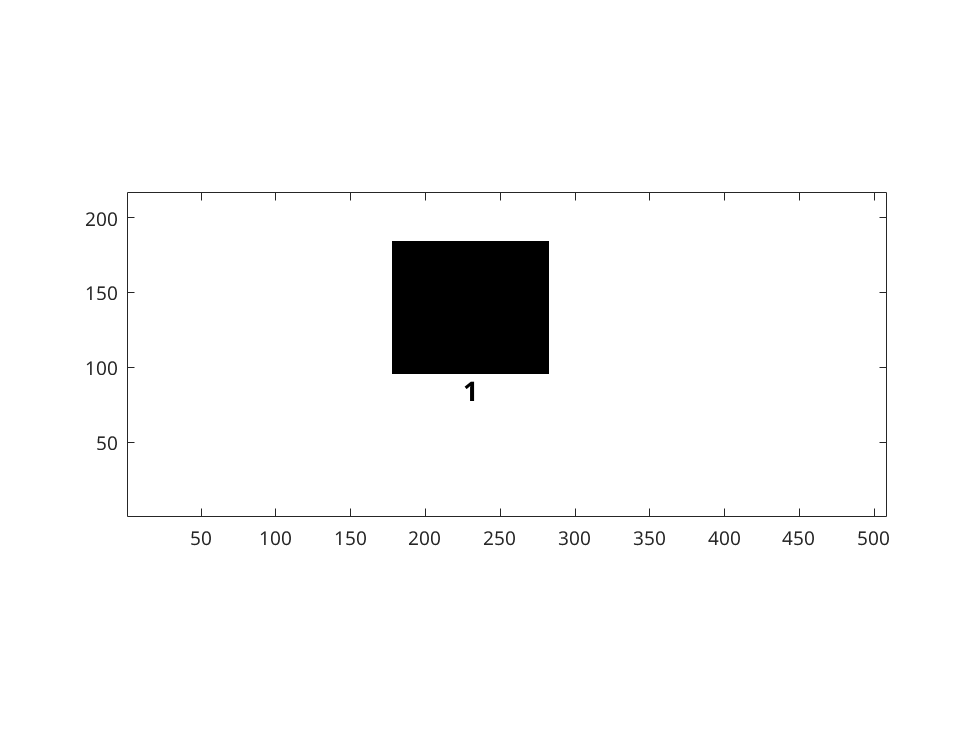}
        \includegraphics[width=0.25\linewidth, clip=true, trim={50pt 180pt 50pt 180pt}]{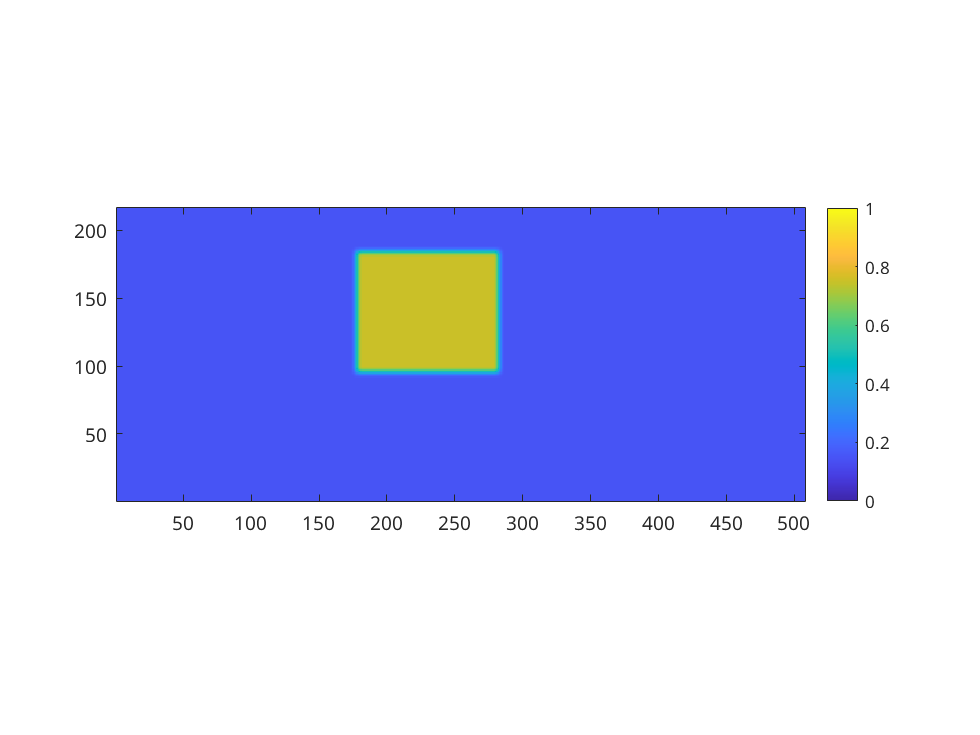}
        \includegraphics[width=0.25\textwidth, clip=true, trim={50pt 180pt 50pt 180pt}]{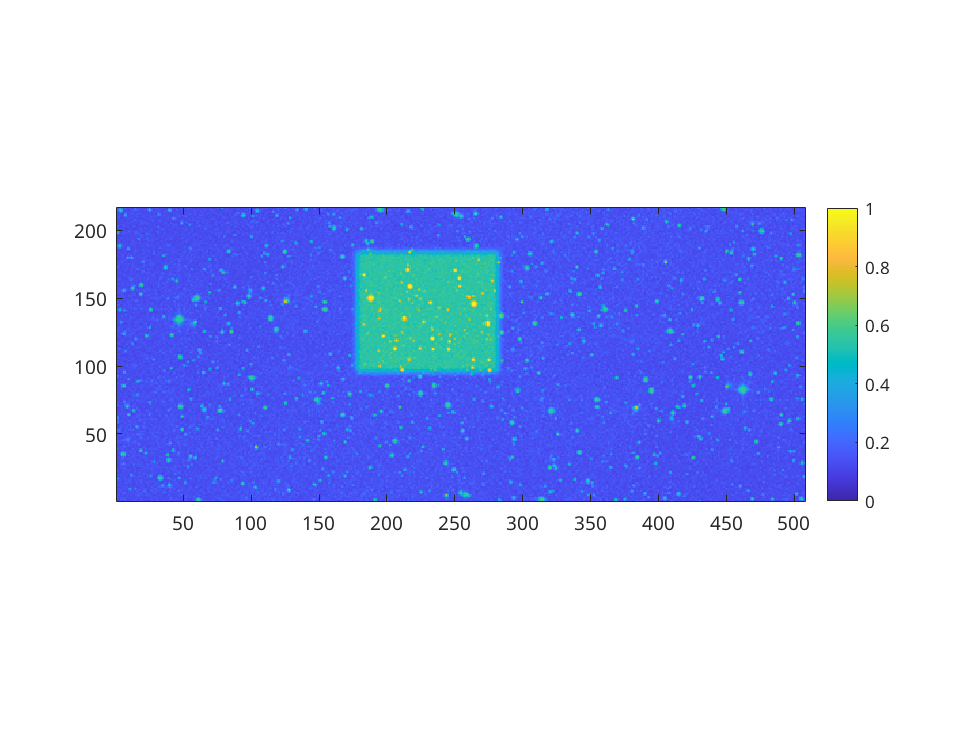}
        \includegraphics[width=0.25\textwidth, clip=true, trim={50pt 180pt 50pt 180pt}]{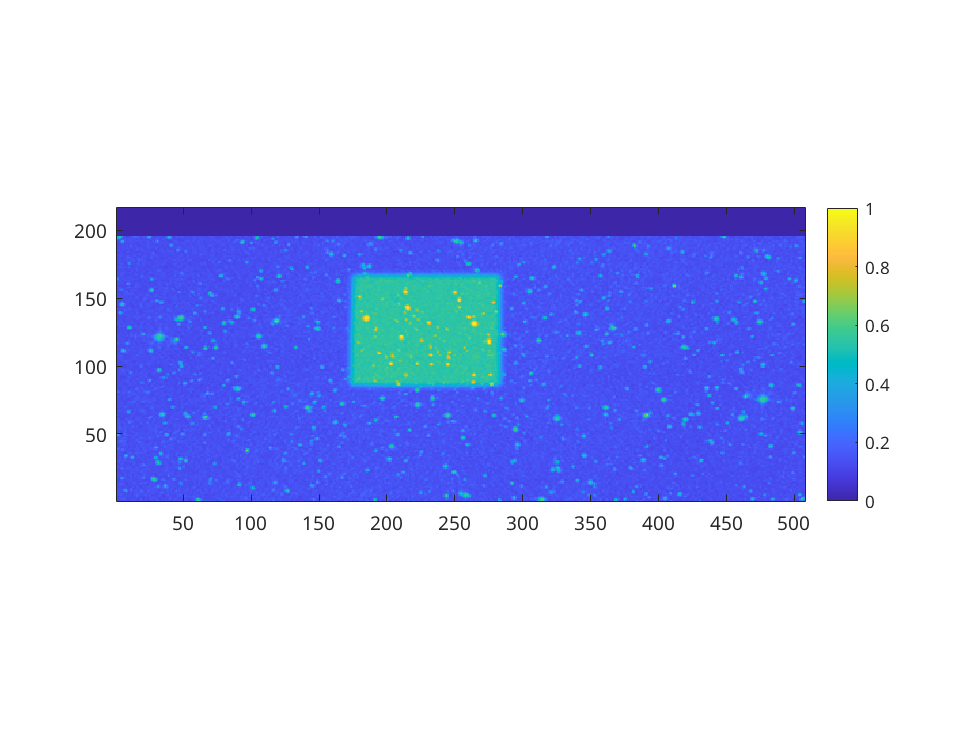}

        \includegraphics[width=0.22\linewidth, clip=true, trim={50pt 150pt 50pt 150pt}]{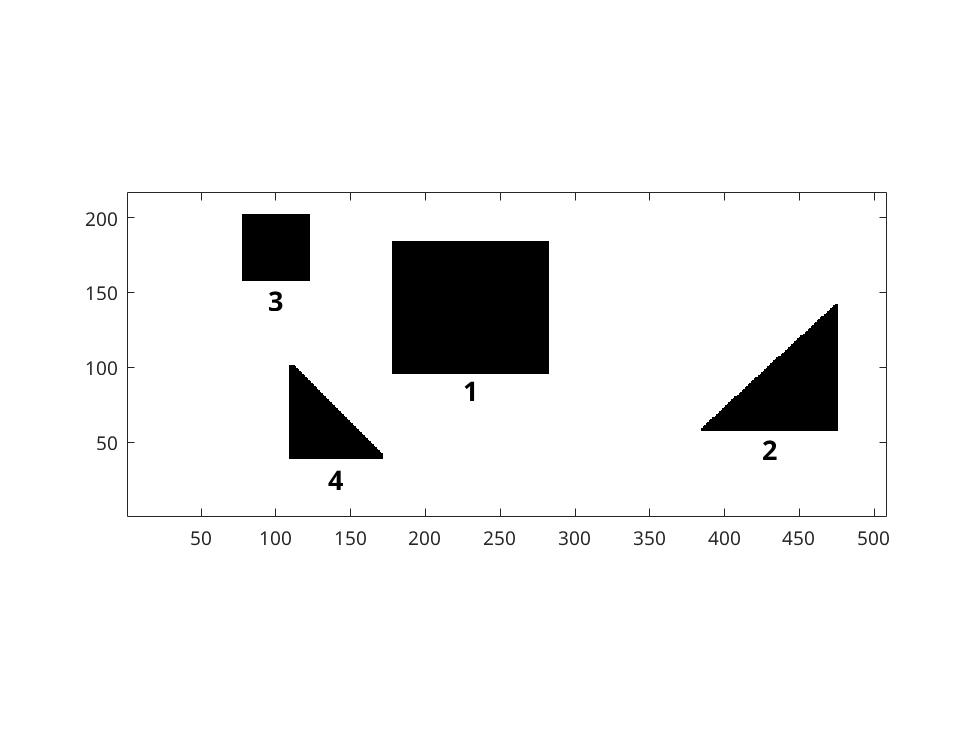}
        \includegraphics[width=0.25\linewidth, clip=true, trim={50pt 180pt 50pt 180pt}]{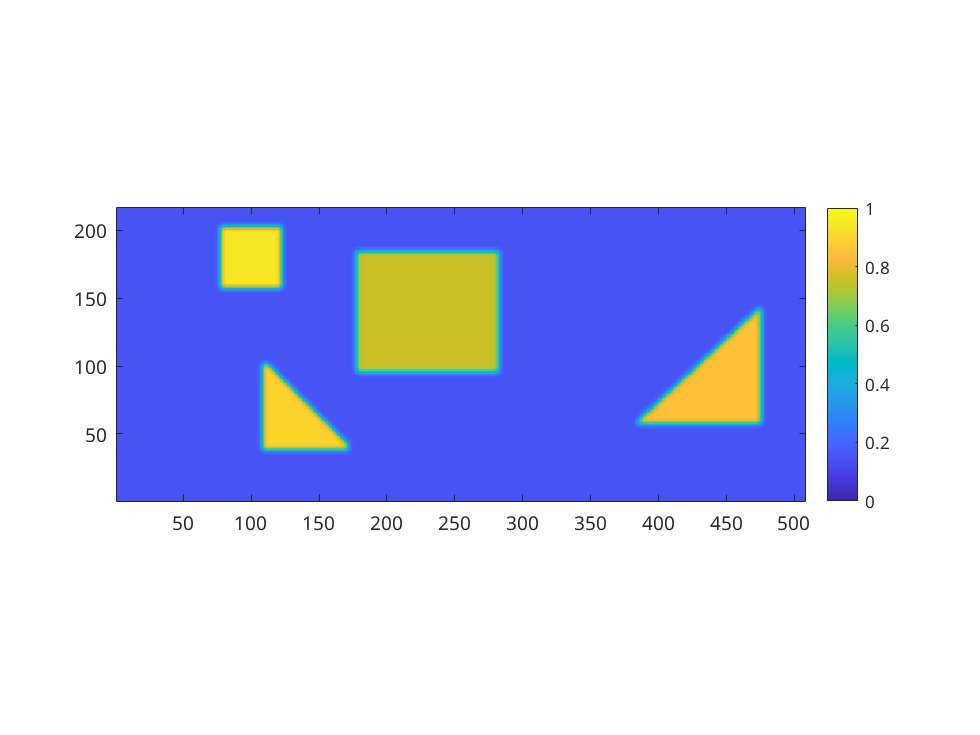}
        \includegraphics[width=0.25\textwidth, clip=true, trim={50pt 180pt 50pt 180pt}]{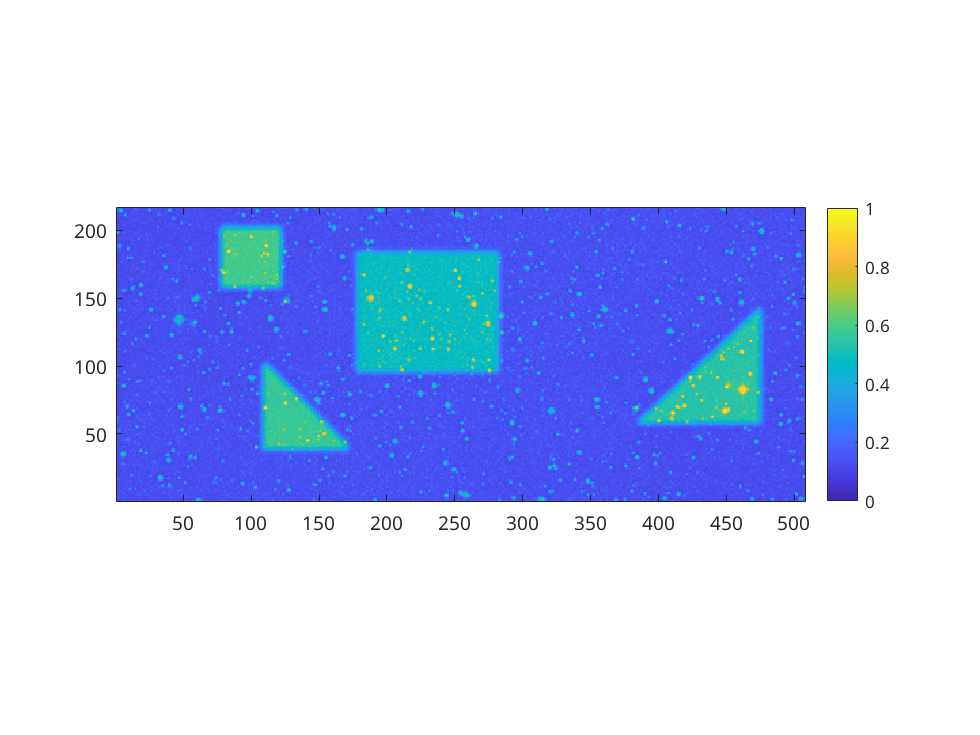}
        \includegraphics[width=0.25\textwidth, clip=true, trim={50pt 180pt 50pt 180pt}]{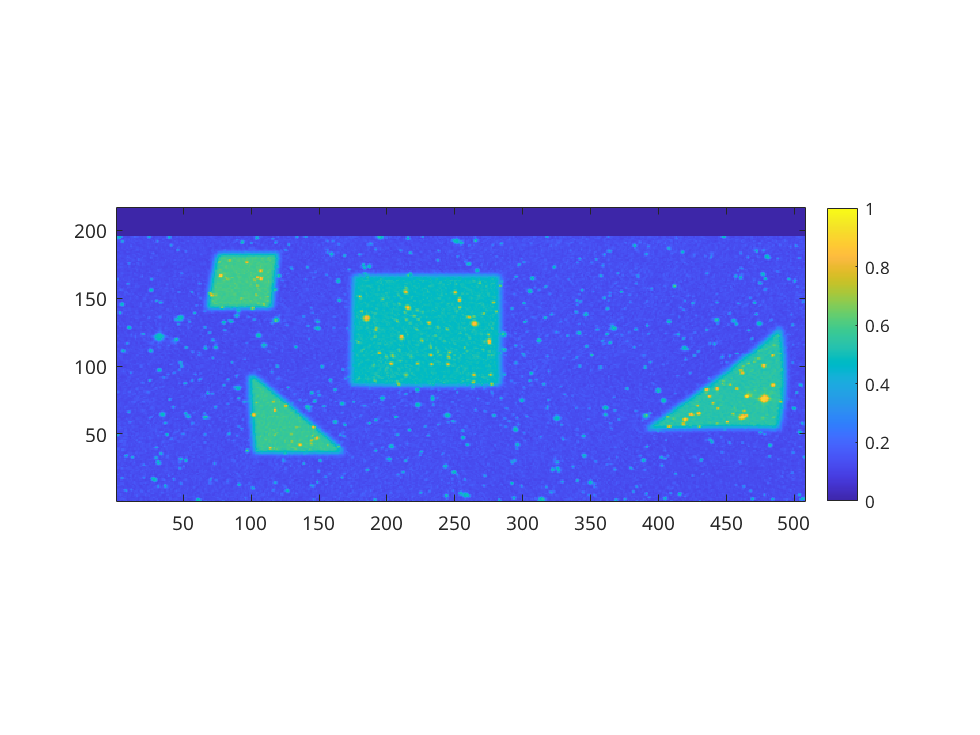}
        \caption{Samples containing inclusions with sharp corners. Schematic depiction of simulated material sample structure and synthetic OCT images before the application of artificial speckle (left, 1st and 2nd column). Inclusion numbering corresponding to Table~\ref{tab_ground_truth}. Synthetic OCT images (including speckle) before and after compression (right, 3rd and 4th column), corresponding to $\mI_1$ and $\mI_2$.}
        \label{fig_sharp_corners}
    \end{figure}

    \begin{table}[ht!]
        \centering
        \resizebox{0.6\textwidth}{!}{
        \begin{tabular}{c|c|c|c|c|}
             & \multicolumn{2}{c|}{\textbf{Test7}} & \multicolumn{2}{c|}{\textbf{Test8}}  \\
             \hline
             Entity & $\lambda$, kPa & $\mu$, kPa & $\lambda$, kPa & $\mu$, kPa \\
             \hline
             Background & 310.2610 & 34.4383 &  393.1418 & 43.9928 \\
             Inclusion 1 & 622.0583 & 68.7888 & 701.8782 & 79.9020 \\
             Inclusion 2 &  &  & 430.7879 & 24.0213 \\
             Inclusion 3 &  &  & 434.7446 & 35.8973 \\
             Inclusion 4 &  &  & 677.3655 & 75.9206 \\
             \hline
             Relative error in $\lambda$ and $\mu$ (\%)  & 0.0003 & 0.0001 & 7.1260 & 5.4498 \\
             \hline
             Relative error in $(\lambda,\mu)$ (\%) & \multicolumn{2}{c|}{0.0003} & \multicolumn{2}{c|}{8.9711} \\
             \hline
        \end{tabular}
        }
        \caption{Numerical results with noise-free images $\mI_1$ and $\mI_2$. Test7 considers the synthetic OCT sample with a single inclusion, see Figure~\ref{fig_sharp_corners} (top), while Test8 considers the sample with four inclusions; see Figure~\ref{fig_sharp_corners} (bottom). In Test7 and Test8 both Lam\'e parameters $\lambda,\mu$ are reconstructed.}
        \label{tab_sharp_corners}
    \end{table}

\noindent
In our fourth series of tests, we demonstrate the applicability of our IIM approach to samples containing inclusions with sharp corners, see Figure~\ref{fig_sharp_corners}. The samples are created following the same workflow as described in Section~\ref{sect_problem_setting}, and are assigned the same ground-truth values of the Lam\'e parameters $\lambda,\mu$ as those given in Table~\ref{tab_ground_truth}. The resulting reconstructions obtained via the IIM are summarized in Table~\ref{tab_sharp_corners}. Therein, Test7 corresponds to the sample with one inclusion and Test8 to that with four inclusions. In terms of precision, the results match the reconstructions in the noise-free case with round inclusions investigated in Test3 and Test4. The most likely reason why sharp corners do not cause significant challenges to our considered IIM realization is that in our tests, we assume that the location and shape of the inclusions are known. This should be compared with the results obtained via the idealized two-step approach discussed in the next section; see Test11 and Test12 in Figure~\ref{fig_nli_iim_comparison_corners} and Table~\ref{tab_nli_iim_comparison}.

\subsection{Numerical results: comparison to a two-step approach}\label{subsect_NLI}

    \begin{figure}[ht!]
        \centering
        \includegraphics[width=0.49\linewidth, clip=true, trim={50pt 50pt 20pt 40pt}]{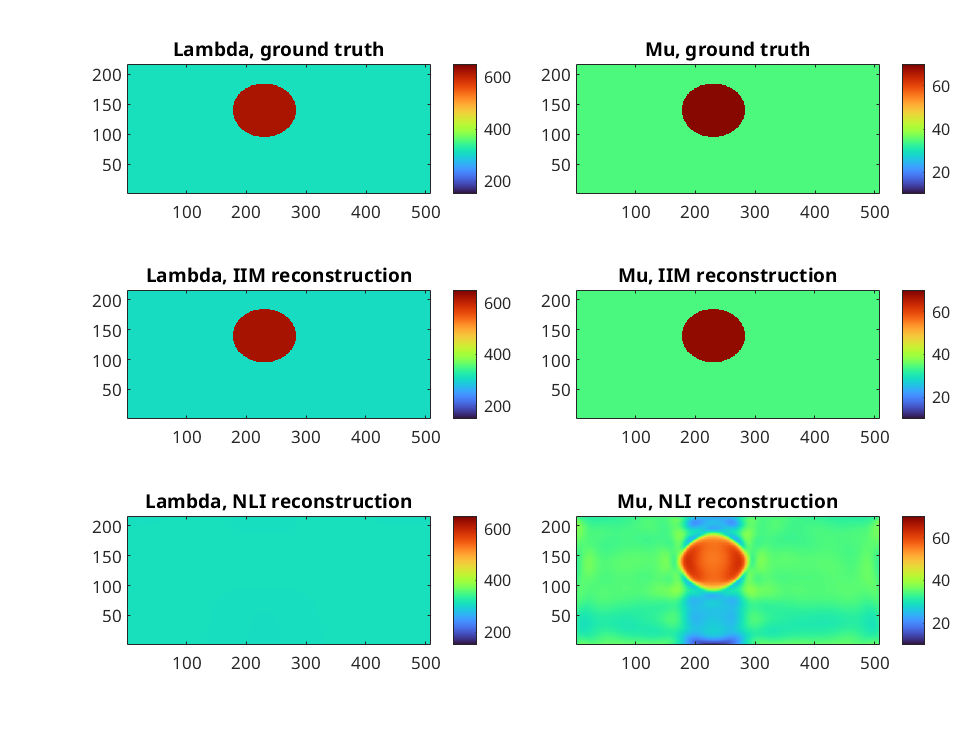}
        \includegraphics[width=0.49\linewidth, clip=true, trim={50pt 50pt 20pt 40pt}]{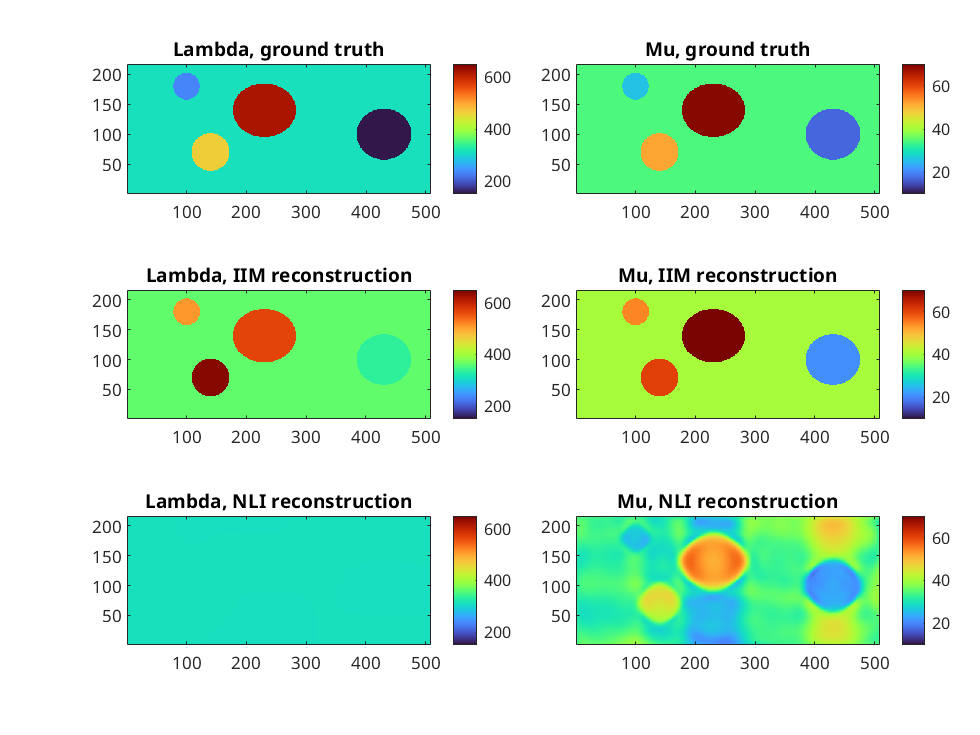}
        \caption{Ground-truth values of the Lam\'e parameters $\lambda, \mu$, in kPa, (top). Reconstruction results with the IIM (middle) and NLI (bottom) for samples with one and four inclusion, corresponding to Test9 (left) and Test10 (right), see Figure~\ref{fig_samples_setup}.}
        \label{fig_nli_iim_comparison}
    \end{figure}
    \begin{figure}[ht!]
        \centering
        \includegraphics[width=0.49\linewidth, clip=true, trim={50pt 50pt 20pt 40pt}]{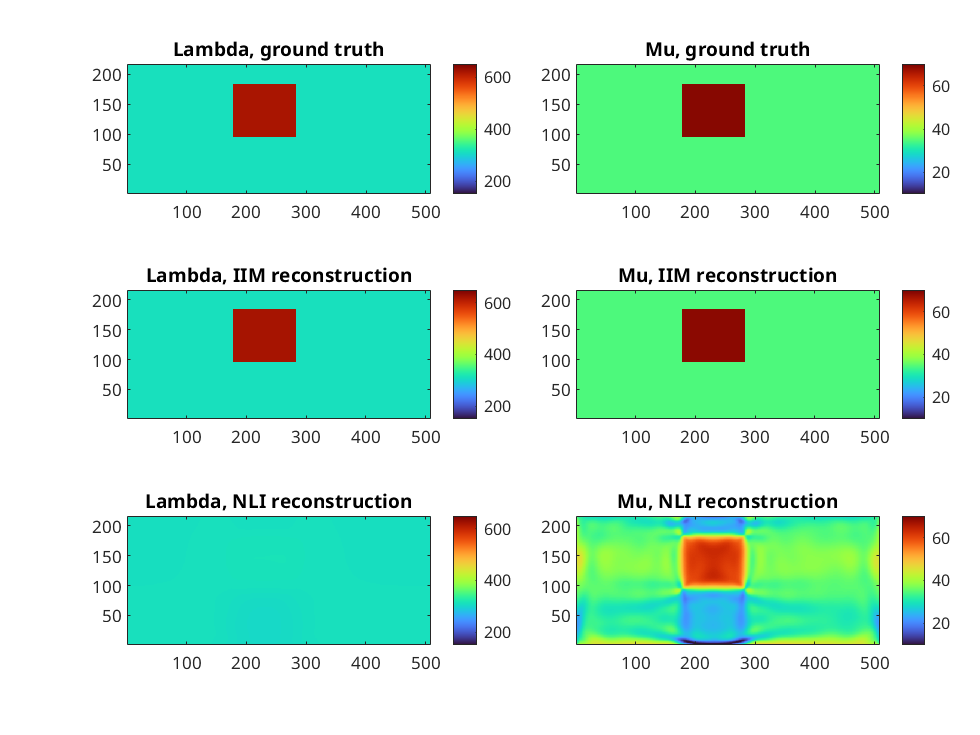}
        \includegraphics[width=0.49\linewidth, clip=true, trim={50pt 50pt 20pt 40pt}]{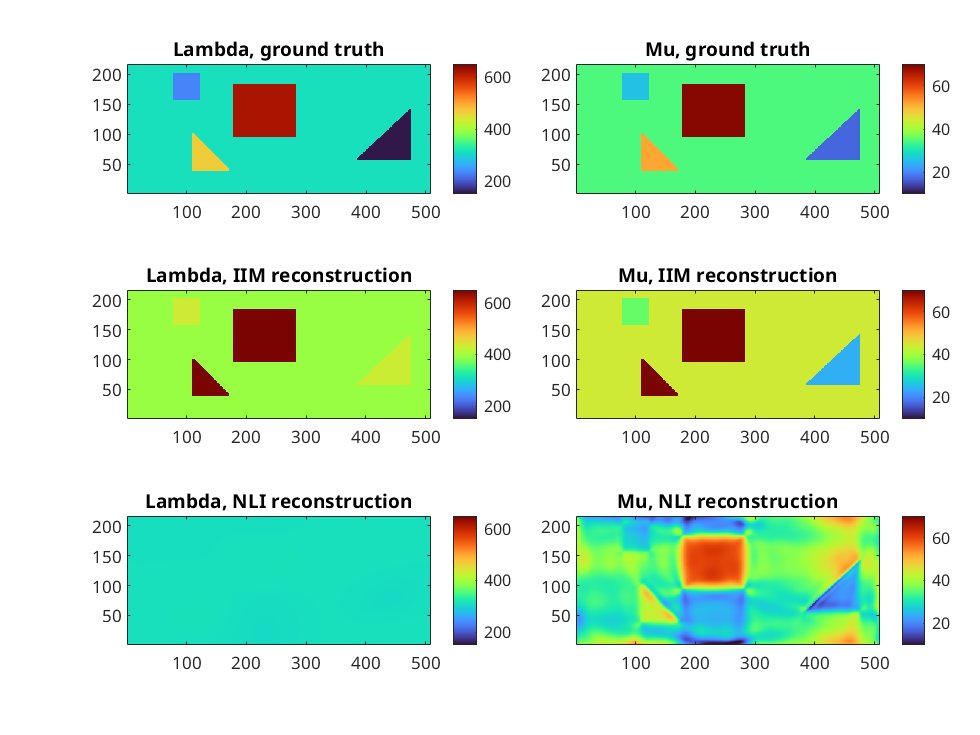}
        \caption{Ground-truth values of the Lam\'e parameters $\lambda, \mu$, in kPa, (top). Reconstruction results with the IIM (middle) and NLI (bottom) for samples with one and four inclusions with sharp corners, corresponding to Test11 (left) and Test12 (right), see Figure~\ref{fig_sharp_corners}.}
        \label{fig_nli_iim_comparison_corners}
    \end{figure}

    \begin{table}[ht!]
        \centering
        \resizebox{0.9\textwidth}{!}{
        \begin{tabular}{c|c|c|c|c|c|c|c|c|}
             & \multicolumn{2}{c|}{\textbf{Test9}} & \multicolumn{2}{c|}{\textbf{Test10}} & \multicolumn{2}{c|}{\textbf{Test11}} & \multicolumn{2}{c|}{\textbf{Test12}} \\
             \hline
             Entity & $\lambda$, kPa & $\mu$, kPa & $\lambda$, kPa & $\mu$, kPa & $\lambda$, kPa & $\mu$, kPa & $\lambda$, kPa & $\mu$, kPa \\
             \hline
             Background & 310.2504 & 32.9527 & 310.3402 & 33.6732 & 309.9822 & 33.3221 & 310.1802 & 33.5326 \\
             Inclusion 1 & 311.1474 & 56.4499 & 310.8157 & 51.7839 & 312.9600 & 58.0439 & 312.1196 & 56.1970 \\
             Inclusion 2 &  &  & 309.2315 & 23.2383 &  &  & 306.8507 & 22.2099 \\
             Inclusion 3 &  &  & 310.1465 & 26.8231 &  &  & 308.5985 & 26.9367 \\
             Inclusion 4 &  &  & 310.7073 & 43.4618 &  &  & 312.1722 & 43.3582 \\
             \hline
             Relative error in $\lambda$ and $\mu$ (\%)  & 22.7093 & 11.9388 & 23.7932 & 16.6251 & 26.1704 & 13.4400 & 26.6889 & 17.8903 \\
             \hline
             Relative error in $(\lambda,\mu)$ (\%) & \multicolumn{2}{c|}{25.6563} & \multicolumn{2}{c|}{29.0260} & \multicolumn{2}{c|}{29.4198} & \multicolumn{2}{c|}{32.1304} \\
             \hline
        \end{tabular}
        }
        \caption{Numerical results with an idealized two-step NLI approach using given exact displacement fields. Test9 and Test11 consider the synthetic OCT sample with a single inclusion, see~Figure~\ref{fig_samples} (top) and Figure~\ref{fig_sharp_corners} (top), while Test10 and Test12 consider the samples with four inclusions; see~Figure~\ref{fig_samples} (bottom) and Figure~\ref{fig_sharp_corners} (bottom). In all tests both Lam\'e parameters $\lambda,\mu$ are unknown. The given values for the Lam\'e parameters are mean values computed for each structure area of the considered sample: for Test9 from the reconstruction in Figure~\ref{fig_nli_iim_comparison} (left, bottom), for Test10 from the reconstruction in Figure~\ref{fig_nli_iim_comparison} (right, bottom), for Test11 from the reconstruction in Figure~\ref{fig_nli_iim_comparison_corners} (left, bottom) and for Test12 from the reconstruction in Figure~\ref{fig_nli_iim_comparison_corners} (right, bottom).}
        \label{tab_nli_iim_comparison}
    \end{table}

\noindent   
In the fifth series of tests, we compare our IIM with a particular instance of an (idealized) two-step approach. As noted in the introduction, two-step approaches first estimate the internal displacement field from given intensity images (here $\mI_1$ and $\mI_2$) and then, in a second step, estimate the material parameters from the computed displacement fields. Each step of this approach involves the solution of an inverse problem from potentially noisy data, which are solved sequentially, and thus noise and inversion errors are introduced effectively ``twice'' into the process of estimating the material parameters. In contrast, our IIM approach avoids this separation of steps, and directly estimates the unknown material parameters from the given intensity images. In this sense, the IIM (as well as other one-step methods) can be said to be more stable/robust with respect to noise and numerical errors. Hence, in order to provide a fair comparison, our chosen two-step approach is idealized in the following way: we assume that the first step can be solved exactly, i.e., the internal displacement field is known precisely. For the second step, we then use a nonlinear Landweber iteration (NLI) as proposed in \cite{Hubmer_Sherina_Neubauer_Scherzer_2018}, which is guaranteed to converge locally under standard assumptions. In total, we conduct four numerical experiments comparing our IIM with this idealized two-step approach, which are summarized in Figure~\ref{fig_nli_iim_comparison} (Test9 and Test10), Figure~\ref{fig_nli_iim_comparison_corners} (Test11 and Test12), and Table~\ref{tab_nli_iim_comparison}. Note that these tests also include the case of samples containing inclusions with sharp corners (Test11 and Test12). In order to compare the spatial reconstructions of the NLI approach with the reconstructions obtain with the IIM, we compute mean values for the reconstructed Lam\'e parameters $\lambda, \mu$ per structure area in Table~\ref{tab_nli_iim_comparison}. As can be observed from these results, the two-step NLI is able to recover the structure of all four samples in the Lam\'e parameter $\mu$, but not in $\lambda$. Furthermore, it also struggles to retrieve the correct contrast, and the quantitative accuracy is lost, resulting in an overall relative error in $\lambda, \mu$ of up to 25-32\%. Recall that in the idealized two-step approach used for these tests, we assume that the displacement field serving as input to NLI is given exactly, and thus in practice, where the displacement field first has to be estimated from intensity images, the overall accuracy will be even worse. On the other hand, the structure of the sample is used as prior information in the version of IIM applied in our numerical experiments. This a-priori information reduces the dimensionality of the problem, thereby helping with the quantitative correctness of the results, see Test3 and Test4 in Table~\ref{tab:noise-free}. This feature of IIM was found to be beneficial in practice for certain types of real-world data, when the sample structure is either segmented out from the data itself, or obtained via a complementary imaging modality \cite{Krainz_Sherina_Hubmer_Liu_Drexler_Scherzer_2022}.

\subsection{The IIM for experimental data}

Finally, note that the IIM has already been successfully applied for material parameter estimation in practice, in particular for the recovery of Young's modulus $E$ and the Poisson ratio $\nu$ in a quasi-static OCE experiment \cite{Krainz_Sherina_Hubmer_Liu_Drexler_Scherzer_2022}. There, the obtained recoveries of $E$ for a set of analyzed silicone rubber samples were found to be in good agreement with independently measured ground truth values. Since the difficulty in recovering Young's modulus $E$ and the Lam\'e parameter $\mu$ is roughly equivalent (and similarly for $\nu$ and $\lambda$), our numerical results are in good alignment with these practical findings.

\section{Conclusion}\label{sect_conclusion}

In this paper, we considered the intensity-based inversion method (IIM) for quantitative material parameter estimation in quasi-static elastography. As a one-step approach combining both image registration and regularized parameter reconstruction, the IIM has several advantages over two-step elastography approaches, which typically perform displacement field or strain reconstruction and parameter estimation separately. In particular, the IIM avoids some approximations and derivative computations commonly encountered in two-step approaches, and is thus more stable to data noise and model inaccuracies. Within the framework of inverse problems, we showed that the IIM is in fact a convergent regularization method with order-optimal rate of convergence under standard assumptions. Furthermore, we showed that these assumptions are for example satisfied in the practically relevant case of linear elastography. Finally, we discussed the implementation of the IIM and demonstrated its practical usefulness on numerical examples simulating an optical coherence elastography (OCE) experiment.

\section{Support}

This research was funded in part by the Austrian Science Fund (FWF) SFB 10.55776/F68 ``Tomography Across the Scales'', project F6805-N36 (Tomography in Astronomy) and project F6807-N36 (Tomography with Uncertainties), as well as 10.55776/P34981 ``New Inverse Problems of Super-Resolved Microscopy (NIPSUM)''. For open access purposes, the authors have applied a CC BY public copyright license to any author-accepted manuscript version arising from this submission. The financial support by the Austrian Federal Ministry for Digital and Economic Affairs, the National Foundation for Research, Technology and Development and the Christian Doppler Research Association is gratefully acknowledged.

\bibliographystyle{siam}
{\footnotesize
\bibliography{my.bib}
}

\newpage
\appendix

\end{document}